\documentclass[10pt]{amsart}

\usepackage{amssymb, amsmath, amsthm, amsfonts, stmaryrd, mathrsfs}
\usepackage{graphicx}
\usepackage{tikz}
\usetikzlibrary{shapes,arrows}
\usepackage{float}
\usepackage{hvfloat}
\usepackage{caption}
\usepackage{url}
\usepackage[all,arc,2cell]{xy}
\UseAllTwocells
\usepackage{enumerate}
\usepackage{color}

\usepackage{hyperref}
  \definecolor{dark-red}{rgb}{0.6,0.15,0.15}
   \definecolor{dark-blue}{rgb}{0.15,0.15,0.6}
   \definecolor{medium-blue}{rgb}{0,0,0.5}

\hypersetup{
    colorlinks, 
    linkcolor=dark-red,
    citecolor=dark-blue, urlcolor=medium-blue
}

\usepackage[nameinlink,capitalise,noabbrev]{cleveref}

\numberwithin{equation}{section}

\newtheorem{thm}{Theorem}[section]
\newtheorem{theorem}{Theorem}[section]

\newtheorem{cor}{Corollary}[section]

\newtheorem{Corollary}{Corollary}[section]

\newtheorem{proposition}{Proposition}[section]
\newtheorem{Proposition}{Proposition}[section]
\newtheorem{lem}{Lemma}[section]
\newtheorem{lemma}{Lemma}[section]
\newtheorem{Lemma}{Lemma}[section]

\theoremstyle{definition}
\newtheorem{defn}{Definition}[section]

\newtheorem{notation}{Notation}[section]

\newtheorem{rem}{Remark}[section]
\newtheorem{warn}{Warning}[section]

\newtheorem{slogan}{Slogan}[section]
\newtheorem*{term}{Terminology}

\makeatletter
\let\c@equation=\c@thm
\let\c@lem=\c@thm
\let\c@theorem=\c@thm
\let\c@lemma=\c@thm
\let\c@Theorem=\c@thm
\let\c@Lemma=\c@thm
\let\c@cor=\c@thm
\let\c@corollary=\c@thm
\let\c@Corollary=\c@thm
\let\c@conj=\c@thm
\let\c@conjecture=\c@thm
\let\c@prop=\c@thm
\let\c@proposition=\c@thm
\let\c@Proposition=\c@thm
\let\c@defn=\c@thm
\let\c@definition=\c@thm
\let\c@Definition=\c@thm
\let\c@notation=\c@thm
\let\c@note=\c@thm
\let\c@exmp=\c@thm
\let\c@ex=\c@thm
\let\c@exmps=\c@thm
\let\c@rem=\c@thm
\let\c@warn=\c@thm
\let\c@claim=\c@thm
\let\c@convention=\c@thm
\let\c@conventions=\c@thm
\let\c@quest=\c@thm
\let\c@facts=\c@thm
\let\c@slogan=\c@thm
\makeatother

\newcommand{\F}{\mathbb{F}}
\newcommand{\N}{\mathbb{N}}
\newcommand{\Z}{\mathbb{Z}}

\newcommand{\Q}{\mathbb{Q}}

\newcommand{\R}{\mathbb{R}}
\newcommand{\ZZ}{\mathbb{Z}}
\newcommand{\FF}{\mathbb{F}}

\newcommand{\smsh}{\wedge}
\newcommand{\xra}{\xrightarrow}

\newcommand{\pt}{\mathrm{pt}}

\newcommand{\afifteen}{y}
\newcommand{\aseven}{x}

\def\makeop#1{\expandafter\def\csname #1\endcsname{\mathop{\mathrm{#1}}\nolimits}}

\makeop{Gal}
\makeop{id}
\makeop{Mod}
\makeop{Hom}
\makeop{Tot}
\makeop{gr}
\makeop{Out}
\makeop{Hom}
\makeop{Ext}
\makeop{End}
\makeop{Aut}
\makeop{Tor}
\makeop{ev}
\makeop{hocolim}
\makeop{holim}
\makeop{map}
\makeop{id}
\makeop{colim}
\makeop{im}
\newcommand{\kappabar}{\bar{\kappa}}

\usepackage{xargs} 
\usepackage[textwidth=2.5cm]{todonotes}
\newcommandx{\irina}[2][1=]{\todo[linecolor=green,backgroundcolor=green!25,bordercolor=green,#1]{#2}}
\newcommandx{\cuong}[2][1=]{\todo[linecolor=red,backgroundcolor=red!25,bordercolor=red,#1]{#2}}
\newcommandx{\agnes}[2][1=]{\todo[linecolor=blue,backgroundcolor=blue!25,bordercolor=blue,#1]{#2}}

\makeatletter
\newcommand{\mylabel}[2]{#2\def\@currentlabel{#2}\label{#1}}
\makeatother

\newcommand{\Ceta}{C_{\eta}}

\title[The topological modular forms of $\mathbb{R}P^2$ and $\mathbb{R}P^2 \wedge \mathbb{C}P^2$]{The topological modular forms of $\mathbb{R}P^2$ and $\mathbb{R}P^2 \wedge \mathbb{C}P^2$} 
\date{\today}

\author[Agn\`es Beaudry]{Agn\`es Beaudry}
\address{Department of Mathematics, University of Colorado, Boulder, Boulder, CO, Campus Box 395, Boulder, CO, 80309, USA}
\email{agnes.beaudry@colorado.edu }
\author[Irina Bobkova]{Irina Bobkova}
\address{Department of Mathematics, Texas A\&M University, College Station, TX, 77843, USA}
\email{ibobkova@math.tamu.edu}
\author[Viet-Cuong Pham]{Viet-Cuong Pham}
\address{Max-Planck Institute for Mathematics, Vivatsgasse 7, 53111 Bonn, Germany}
\email{vcpham@math.unistra.fr}
\author[Zhouli Xu]{Zhouli Xu}
\address{Department of Mathematics, UC San Diego, La Jolla, CA 92093, USA}
\email{xuzhouli@ucsd.edu  }

\thanks{This material is based upon work supported by the National Science Foundation under grants No.~DMS--1906227, DMS-2005627 and DMS-2043485.}

\begin{document}
\maketitle
\begin{abstract}
In this paper, we study the elliptic spectral sequence computing $tmf_*(\R P^2)$ and $tmf_* (\mathbb{R} P^2 \wedge \mathbb{C} P^2)$. Specifically, we compute all differentials and resolve exotic extensions by $2$,  $\eta$, and $\nu$. For  $tmf_* (\mathbb{R} P^2 \wedge \mathbb{C} P^2)$, we also compute the effect of the $v_1$-self maps of $\mathbb{R} P^2 \wedge \mathbb{C} P^2$ on $tmf$-homology. \end{abstract}

\setcounter{tocdepth}{1}
\tableofcontents


\section{Introduction}
\subsection{Motivation}
Topological modular forms ($tmf$) are ubiquitous in algebraic topology and homotopy theory. The goal of this paper is to compute the $tmf$-homology of two spaces, namely $\mathbb{R}P^2$ and $ \mathbb{R}P^2 \wedge \mathbb{C}P^2$, and to determine the differentials and extensions in their elliptic spectral sequences.

We approach this problem from the point of view of stable homotopy theory. As is common, we let $V(0)$ denote the cofiber of multiplication by 2 on the sphere spectrum. Then 
\[V(0) \simeq \Sigma^{-1} \Sigma^{\infty} \R P^2\] and, via the suspension isomorphism, computing $tmf_*V(0) \cong \pi_* tmf \wedge V(0)$ is equivalent to computing the $tmf$-homology of $\R P^2$. Similarly, let $Y$ be the smash product of $V(0)$ with $C_{\eta}$, the cofiber of the stable Hopf map $\eta$. Then 
\[Y \simeq \Sigma^{-3}\R P^2 \wedge \mathbb{C} P^2\] 
and computing $tmf_*Y$ is equivalent to computing the $tmf$-homology of $\R P^2 \wedge \mathbb{C} P^2  $. In this paper, we compute the elliptic spectral sequence for both $tmf \wedge V(0)$ and $tmf \wedge Y$. From this computation, we deduce $tmf_*V(0)$ and $tmf_*Y$ and provide information about their module structure over $tmf_*$. In particular, we resolve all exotic $2,\eta, \nu$ extensions as well as compute the effect of $v_1$-self maps of $Y$ on $tmf_*Y$. Note that determining the $tmf_*$-module structure is much less straightforward than a simple degree-wise computation of $tmf_*V(0)$ or $tmf_*Y$.

Knowing the homology of basic spaces is part of a full understanding of any generalized homology theory. So we see these computations as having independent and fundamental interest. They are, at the very least, an addition to the slim bank of examples of computations in $tmf$-homology theory of spaces and finite spectra.

However, our motivation for doing this runs deeper and this computation is part of a more ambitious program, coming from chromatic homotopy theory. Specifically, our main goal in doing this computation is not just to understand the structure of $tmf_*V(0)$ and $tmf_*Y$ as $tmf_*$-modules, but more-so  {\emph{to fully compute their elliptic spectral sequences}. To explain this, we let $K(2)$ denote the Morava $K$-theory spectrum  and $E_2$ the Lubin-Tate spectrum (also often called Morava $E$-theory).

In the sequence of papers \cite{GHMV1, ghmr, HKM, GHM_pic, GH_bcdual, GoerssSplit, henn_res}, Goerss, Henn, Karamanov, Mahowald and Rezk carry out a program for studying $K(2)$--local homotopy theory at $p=3$ using the theory of \emph{finite resolutions}. These are sequences of spectra built from the $K(2)$-localization of $tmf$ (and $tmf$ with level structures) that resolve the $K(2)$-local sphere. Finite resolutions give rise to Bousfield-Kan spectral sequences. Let us call these \emph{finite resolution spectral sequences}. The input is $K(2)$-local $tmf$-homology, possibly with level structures, and the output is $K(2)$-local homotopy groups. The ultimate goal is to use finite resolutions to compute $\pi_*L_{K(2)}S^0$, but an intermediate step is the computations of the homotopy groups of $L_{K(2)}F$ for some key finite spectra $F$, such as the prime 3 Moore spectrum $V(0)$ \cite{HKM} and the cofiber of its $v_1$-self map, commonly denoted $V(1)$ \cite{GHMV1}. So, to use the finite resolution approach to $K(2)$-local homotopy, a key input is $\pi_*L_{K(2)}(tmf \wedge F)$. This can be computed via the $K(2)$-local $E_2$-based Adams-Novikov spectral sequence (which can also be cast as a homotopy fixed point spectral sequence).  This spectral sequence  receives a map from  the elliptic spectral sequence of $tmf \wedge F$. Understanding the elliptic spectral sequence of $tmf \wedge F$ thus provides key input for $K(2)$-local computations.

Recently, there have been significant advancements towards carrying out an analogous program at the prime $p=2$. See \cite{BeaudryResolutions, BeaudryMoore, BobkovaGoerss, BGH}. But the program is still in progress. For example, 
the only complete computation of the $K(2)$-local homotopy groups of a finite spectrum at $p=2$ is the computation of $\pi_*L_{K(2)}Z$ for $Z \in \mathcal{Z}$, where $\mathcal{Z}$ is the class of Bhattacharya-Egger spectra admitting a $v_2$-self map. See \cite{PE2,PE3} and also \cite{2019arXiv190913379B}. The motivation for this project is to add to this bank of computations, namely, to study $L_{K(2)}V(0)$, $L_{K(2)}Y$, but also $L_{K(2)}A_1$ where $A_1$ is the cofiber of a $v_1$-self map of $Y$. For this, we found the need to understand the elliptic spectral sequence of $tmf \wedge V(0)$, $tmf\wedge Y$ and $tmf\wedge A_1$. In \cite{PhamA1}, the third author computes a $K(2)$-local $E_2$-based Adams-Novikov spectral sequence converging to $\pi_*L_{K(2)}(tmf\wedge A_1)$. From this computation, one can deduce that of the elliptic spectral sequence of $tmf\wedge A_1$.

Here, we study instead the elliptic spectral sequences of $tmf \wedge V(0)$ and $tmf\wedge Y$. For $F$ either $V(0)$ or $Y$, $tmf_*F=0$ for $*<0$ and $tmf_{*}F$ is determined by its values in the range $0\leq * <192$. In this paper, we obtain the following result, where the definition of what we mean by \emph{exotic extensions} is given in \Cref{defnexoext}. 
\begin{theorem}
The elliptic spectral sequence for $tmf\wedge V(0)$ is depicted in Figures~\ref{V0-0-50}, \ref{figured5d7}, \ref{figd9andthensome} and \ref{figd9andthensomemore}. The $tmf$-homology of $V(0) \simeq \Sigma^{-1}\Sigma^{\infty}\R P^{2}$, namely
\[tmf_* V(0) \cong \widetilde{tmf}_{*+1}\R P^2,\]
 together with \textbf{all exotic $2, \eta$ and $\nu$ extensions} in the corresponding elliptic spectral sequence is as displayed in Figures~\ref{V0-48-96-ext} and \ref{V0-144-192-ext} in degrees $0\leq * <192$. 
 
Similarly, the elliptic spectral sequence for $tmf\wedge Y$  is depicted in  Figures~\ref{gensone}, \ref{d5d9one}, \ref{d5d9three}, \ref{d5d9four}, \ref{d11d23one} and \ref{d11d23four}. The $tmf$-homology of $Y \simeq \Sigma^{-3}\Sigma^{\infty}\R P^2 \wedge \mathbb{C} P^2$,  namely
\[tmf_* Y \cong \widetilde{tmf}_{*+3}\R P^2 \wedge \mathbb{C} P^2,\]
   together with \textbf{all exotic $2, \eta$ and $\nu$ extensions and almost all exotic $v_1$-extensions} in the corresponding elliptic spectral sequence is as displayed in Figures~\ref{exoextY096} and \ref{exoextY96144}  in degrees $0\leq * <192$.  In particular, 
   \[2(\widetilde{tmf}_*(\R P^2 \wedge \mathbb{C} P^2))=0.\]
\end{theorem}

\begin{rem}
Computing exotic extensions in this sense of \Cref{defnexoext} can (and does in some places here) leave ambiguity about the module structure. However, this definition of exotic extensions, which we borrowed from \cite{morestems}, is very standard in these kinds of large spectral sequence computations.
\end{rem}

\subsection{Methods and comparison with existing work}

To say a few words about our techniques, the major input in our computation is the elliptic spectral sequence of $tmf$, which was first computed by Hopkins and Mahowald \cite[Ch. 15]{tmfbook}, and later by Bauer \cite{tbauer}. The computation of the spectral sequence for $tmf_*V(0)$ is straightforward given that data, while that of $tmf_*Y$ is more intricate. The technique we use for the latter relies on an observation of the third author from \cite{PhamA1}.  For both $V(0)$ and $Y$, computation of the exotic extensions requires work and new input. Several techniques are used to achieve this, and the most interesting among these is probably the Brown-Comenetz ``self-duality'' of  $tmf_*V(0)$ and $tmf_*Y$. See \Cref{thm:duality}.

In \cite{BrunerRognesbook} (soon to be published), Bruner and Rognes do a thorough investigation of $tmf$.  A main tool used in \cite{BrunerRognesbook} to answer computational questions about $tmf$ and its modules is the \emph{classical Adams spectral sequence}. (Note that the study of the classical Adams spectral sequence of $tmf$ probably goes back to Hopkins and Mahowald, and later to Henriques in \cite[Chapter 13]{tmfbook}.) 
Among many other topics, including duality for topological modular forms which is relevant for our approaches, they study the classical 
Adams spectral sequence of $tmf$ smashed with many finite spectra, including a study of $tmf$ smashed with $V(0)$.
 In particular, they also compute $tmf_*V(0)$, determining all but a few $2,\eta, \nu$-multiplications as well as $v_1^4$-multiplications. Here, we deliberately use the word \emph{multiplication} in contrast to the word \emph{extension} discussed above to emphasize that Bruner--Rognes name all classes, which leads them to a more precise determination of multiplicative relations.
Recently, Bruner and Rognes shared their charts and an advanced copy of some of the chapters of their forthcoming book with us.
However, our results were obtained independently from theirs and via different methods. So the two approaches complement one another. 
We also use a few results on the classical Adams spectral sequence of $tmf_*$ which we verified against both \cite[Chapter 13]{tmfbook} and \cite[Chapters 5,9]{BrunerRognesbook}. Furthermore,  \cite[Chapter 10]{BrunerRognesbook} is a direct reference of \Cref{thm:duality}, which is used extensively in this paper.

Finally, we reiterate that for our applications, namely, as an input in the finite resolution approach to $K(2)$-local homotopy theory, it is important to understand specifically the elliptic spectral sequence instead of the classical Adams spectral sequence because of its close relationship to the homotopy fixed point spectral sequence, a key tool in chromatic homotopy theory (see the discussion above).

\subsection{Organization of the paper}
In \Cref{secbackground}, we discuss the elliptic spectral sequences and other key tools used later in the paper. In \Cref{secV0E2}, we review the computation of the $E_2$-term of the elliptic spectral sequence for $tmf \wedge V(0)$. In \Cref{secV0diffext} we compute the differentials and some exotic extensions. In \Cref{secYE2} we turn to the computation of the $E_2$-term of the elliptic spectral sequence for $tmf \wedge Y$ and in  \Cref{secYdiffext} we compute the differentials and exotic extensions.

\subsection{Acknowledgements}
 
We thank Robert Bruner and John Rognes for useful discussions and their generosity in sharing some charts and chapters as well as the front matter of their book \cite{BrunerRognesbook}. 
We are extremely grateful to Hans-Werner Henn and Vesna Stojanoska for useful conversations along the way.  In particular, Henn could very well have been a co-author given the extent of interactions we had with him on this project.

Computations like these are much harder without effective drawing tools and spectral sequence programs. We are thankful to Tilman Bauer (luasseq) and Hood Chatham (spectralsequences) for their  \LaTeX \ spectral sequence programs. While the charts in this paper have mostly been re-drawn with Hood's program, early versions of our computations (before spectralsequences was written) were facilitated by Bauer's program and his kindness in helping us make it work in such large scales. Classic but not least, we thank Bruner for his Ext-program which is an ever-useful tool.

 Finally, the second and third authors also thank l'Universit\'e de Strasbourg for its support during part of the project.


\section{Background}\label{secbackground}
In this section, we review some of the key tools that will be used in the paper.
\subsection{The elliptic spectral sequence}
We begin with the elliptic spectral sequence.
Let
\[
(A, \Lambda)=(\mathbb{Z}[a_1, a_2, a_3, a_4, a_6], \mathbb{Z}[a_1, a_2, a_3, a_4, a_6, s, r, t])
\]
with
\[
 |a_i|=2i, |r|=4, |s|=2, |t|=6
\]
be the Hopf algebroid of Weierstrass elliptic curves. Then the elliptic spectral sequence has the form \cite{tbauer}
\[
E_2^{s,t-s}=\Ext^{s,t}_{\Lambda}(A,A) \Longrightarrow \pi_{t-s}tmf.
\]

Consider the map 
\[\Omega SU(4)\rightarrow \Omega SU\simeq BU\]
induced by the usual inclusion $SU(4)\rightarrow SU$.
Let $X(4)$ be the Thom spectrum of the associated virtual vector bundle over $\Omega SU(4)$. These spectra play a crucial role in the study of nilpotence and periodicity in chromatic homotopy theory, in particular, in the work of Ravenel \cite{ravloc}.  As outlined in \cite[Ch. 9]{tmfbook}, the elliptic spectral sequence is the $X(4)$-based Adams spectral sequence for $tmf$. See also  \cite{Rezk512}.

Let us spell this out. We let $R = tmf$ and $E = tmf\wedge X(4)$. 
Then 
\[E\wedge_{R}E \simeq tmf\wedge X(4)\wedge X(4).\]
Let $\overline{E}$ be the fiber of the unit map $R\rightarrow E$. For any $tmf$-module $M$, one can construct the Adams tower 
\[ \xymatrix{
 M\ar[d] &\ar[l] \ar[d]\overline{E}\wedge_{R} M &\ar[l]\ar[d] \overline{E}\wedge_{R}\overline{E}\wedge_{R}M 
 &...\ar[l] \\
 E\wedge_{R}M\ar@{.>}[ur] & E\wedge_{R}\overline{E}\wedge_{R}M\ar@{.>}[ur] & E\wedge_{R}\overline{E}\wedge_{R}\overline{E}\wedge_{R}M\ar@{.>}[ur] 
 &
}\]
by splicing together the cofiber sequences
\[\overline{E}^{\wedge_{R}(n+1)}\wedge _{R}M \rightarrow \overline{E}^{\wedge_{R}n}\wedge _{R}M\rightarrow E\wedge_{R}\overline{E}^{\wedge_{R}n}\wedge _{R}M.\]
We abbreviate
\begin{align*}
X_k&:=\overline{E}^{\wedge_{R}k}\wedge _{R}M \simeq \overline{X(4)}^{\wedge k}\wedge M, \\
I_k &:=E\wedge_{R}\overline{E}^{\wedge_{R}k}\wedge _{R}M \simeq X(4)\wedge \overline{X(4)}^{\wedge k}\wedge M
\end{align*}
where $\overline{X(4)}$ is the fiber of the unit map $S^0\rightarrow X(4)$. As a consequence, the associated spectral sequence is identified with the $X(4)$-based Adams spectral sequence for $M$.

However, we have that the Hopf algebroid 
\[(\pi_*(E), \pi_*(E\wedge_{R}E)) = (\pi_{*}(tmf\wedge X(4)), \pi_*(tmf\wedge X(4)\wedge X(4)))\] 
is isomorphic to $(A, \Lambda)$. 
In particular, it is flat. Therefore, the $E_2$-term of the associated spectral sequence is identified with 
\[E_2^{s,t}(M)\cong \mathrm{Ext}^{s,t}_{\Lambda}(A, \pi_*(E\wedge_{R} M)). \]
See \cite{BakerLazarev}.
When $M=tmf$, this is precisely the elliptic spectral sequence, and more generally, this is the elliptic spectral sequence for the $tmf$-module $M$.

According to Bousfield \cite[Theorem 6.5]{Bousfield}, since $X(4)$ is connected and $\pi_0(X(4))\cong \Z$, if $M$ is connective, then $L_{X(4)}M\simeq M$ and the spectral sequence converges to $\pi_*(M)$. In particular, if $F$ is a finite spectrum, then the elliptic spectral sequence for $tmf\wedge F$ reads as
\[E_2^{s,t-s} = \mathrm{Ext}^{s,t}_{\Lambda}(A, \pi_*(tmf\wedge X(4)\wedge F)) \Longrightarrow \pi_{t-s}(tmf\wedge F). \]
 To simplify the notation, we put
 \[ \mathcal{F}_*(F):= \pi_*(tmf\wedge X(4)\wedge F)\]
 noting that this is a $\Lambda$-comodule.

 \subsection{(co)Truncated spectral sequences} \label{sectrunss}

   We will use the (co-)truncation of the spectral sequence associated to a tower of cofibrations. We will recall the constructions and their basic properties.
 Let 
 \[\xymatrix{
 X_0\ar[d] &\ar[l] \ar[d]X_1 &\ar[l]\ar[d] X_2 &\ar[l]\ar[d] X_3 &\ar[l] \ldots & \ar[l]\ar[d] X_{n-1} &\ar[l]  X_n& ... \ar[l]   \\
 I_0\ar@{.>}[ur] & I_1 \ar@{.>}[ur]& I_2\ar@{.>}[ur] & I_3 \ar@{.>}[ur] &  & I_{n-1} \ar@{.>}[ur]&
}\]
 be a tower of cofibrations of spectra. Let $(E_{r}^{*,*}, d_r)_{r\geq 1}$ be the associated spectral sequence.

 Let $X_i/X_n$ be the cofiber of the evident map $X_n\to X_i$. 
For any $n\in \N$, there is a tower of fibrations, which we call the \emph{$n$-truncated tower}:
\[\xymatrix{
 X_0/X_{n}\ar[d] &\ar[l] \ar[d]X_1/X_{n} &\ar[l]\ar[d] X_2/X_{n} 
 &\ar[l] \ldots & \ar[l]\ar[d] X_{n-1}/X_{n} &\ar[l]  \pt  \\
 I_0\ar@{.>}[ur] & I_1 \ar@{.>}[ur]& I_2\ar@{.>}[ur] 
 &  & I_{n-1} \ar@{.>}[ur]&
}\]
 We denote the terms of the resulting spectral sequence by $E_{r,< n}^{s,t}$. This spectral sequence computes the homotopy groups of
 \[ \mathrm{sk}_{n-1}X_0 := X_0/X_n.\]
 There is a natural map from the original tower to the $n$-truncated tower. Let
 \[T_{r}^{s,t} \colon {E}_r^{s,t} \to {E}_{r,< n}^{s,t}   \] 
be the induced map between the respective ${E}_r$-terms. Then ${E}_{2,< n}^{s,t} = 0$ for $s\geq n$, while
 $T_{2}^{s,t}$ is an isomorphism if $s< n-1$ and an injection if $s = n-1$. More generally, we have:
\begin{Lemma}\label{TruncSS} For every $r\geq 2$, the map $T_{r}^{s,t}$ has the following properties:
\begin{itemize}	
	\item[(i)] $T^{s,t}_{r}$ is injective for $s\leq n-1$, and
	\item[(ii)] $T^{s,t}_{r}$ is bijective for $s\leq n-1-(r-1)$.
\end{itemize}
\end{Lemma}

\begin{proof}
We prove this by induction on the $r$. From the above discussion, (i) and (ii) hold for $r=2$. Suppose both hold for some $r\geq 2$. 

We prove that (i) holds at $E_{r+1}$. Let $[x]\in E_{r+1}^{s,t}$ be represented by an element $x\in E_{r}^{s,t}$ such that $s\leq n-1$ and $T_{r+1}^{s,t}([x]) = 0$. So $T_{r}^{s,t}(x)$ is the target of a $d_r$-differential. That is, there exists $y\in E^{s-r,t-r-1}_{r,< n}$ such that $d_r(y) = T_{r}^{s,t}(x)$. Since $s-r\leq n-r$, $T_{r}^{s-r,*}$ is bijective by the induction hypothesis. It follows that there exists $\overline{y}\in E^{s-r,t-r-1}_{r}$ such that $T_{r}^{s-r,t-r-1}(\overline{y}) = y$. So, by naturality and the hypothesis that $T_{r}^{s,t}$ is injective, $d_r(\overline{y}) = x$. This means that $[x] = 0$, and hence $T_{r+1}^{s,t}$ is injective when $s\geq n-1$.

Now, we prove that (ii) holds at $E_{r+1}$. Let $[x]\in E_{r+1,< n}^{s,t}$ with $s\leq n-r-1$. We need to show that $[x]$ is in the image of $T_{r+1}^{s,t}$. By the induction hypothesis, there is a class $\overline{x}\in E_{r}^{s,t}$ such that $T_{r}^{s,t}(\overline{x}) = x$. It suffices to prove that $\overline{x}$ is a $d_r$-cycle. By naturality, 
\[T_{r}^{s+r,t+r-1}(d_r(\overline{x})) = d_r(T_{r}^{s,t}(\overline{x})) = d_r(x) = 0.\] 
Since $d_r(x)\in E_{r,< n}^{s+r,t+r-1}$ and $s+r\leq n-1$, the induction hypothesis implies that $d_r(\overline{x}) = 0$. 
\end{proof}

Next, we look at the co-truncated spectral sequence.
Consider the following tower of fibrations, which we call the \emph{$n$-co-truncated tower},
\[\xymatrix{
 Y_0\ar[d] &\ar[l]_{\id} \ar[d]Y_1 &\ar[l]_{\id}  \ldots &\ar[l]_{\id} \ar[d] Y_{n}=X_n &\ar[l]\ar[d] X_{n+1} & \ar[l]\ar[d] X_{n+2}
  & \ldots \ar[l]  \\
 J_0\ar@{.>}[ur] & J_1 \ar@{.>}[ur]&& J_n=I_n \ar@{.>}[ur] &  I_{n+1} \ar@{.>}[ur]
 &
 I_{n+2} \ar@{.>}[ur]
}\]
where $Y_0 = \ldots = Y_n = X_n$ and $J_0 = \ldots = J_{n-1} =pt $.
We denote by $E_{r,\geq n}^{s,t}$ the $r$-term of the spectral sequence associated to this tower. There is an obvious map from the $n$-co-truncated tower to the original one. This map induces a map of spectral sequences:
\[cT_r^{s,t} \colon E_{r, \geq n}^{s,t}\rightarrow E_r^{s,t}.\]
We observe that $E_{r, \geq n}^{s,t}=0$ for $s<n$, and that $cT_2^{s,*}$ is a bijection for $s\geq n+1$ and a surjection for $s=n$.
The following lemma is proved as in \Cref{TruncSS}.

\begin{lemma}\label{cTruncSS} For every $r\geq 2$, the map $cT_{r}^{s,t}$ has the following properties:
\begin{itemize}	
	\item[(i)] $cT^{s,t}_{r}$ is surjective for $s \geq n$, and
	\item[(ii)] $cT^{s,t}_{r}$ is bijective for $s \geq n+r-1$.
\end{itemize}
\end{lemma}

We will be applying this technology to $2$-local spectra. 
As described in \cite[Section 7]{tbauer}, one can simplify the computation of the cohomology of the Weierstrass Hopf algebroid 
\[(A_{(2)}, \Lambda_{(2)}) \cong (A\otimes \Z_{(2)} , \Lambda\otimes \Z_{(2)}) \] 
as follows.
 Let $A'$ denote $\Z_{(2)}[a_1,a_3]$ and $f \colon A \rightarrow A'$ the evident projection. Let $\Lambda'$ denote $A'\otimes_A \Lambda \otimes_A A',$ which is isomorphic to $A'[s,t]/{\sim}$, where the 
relations ${\sim}$ are generated by
\[s^4-6st+a_1s^3-3a_1t-3a_3s=0\]
\[s^6-27t^2+3a_1s^5-9a_1s^2t+3a_1^2s^4-9a_1^2st+a_1^3s^3-27a_3t=0.\] 
The map between Hopf algebroids
 \[ f \colon (A_{(2)}, \Lambda_{(2)})  \to (A', \Lambda')\]
  induces an equivalence of the associated categories of comodules \cite[Sections 2 \& 7]{tbauer}, where
 \[ N \mapsto A'\otimes_{A_{(2)}}N \]
 for an $(A_{(2)}, \Lambda_{(2)})$-comodule $N$.
  When $F$ is the $2$-localization of a finite spectrum, the ${E}_2$-term of the elliptic spectral sequence for 
   \[tmf \wedge F \simeq tmf_{(2)} \wedge F\] 
is isomorphic to
  \[
E_2^{s,t}(tmf \smsh F) \cong \Ext^{s,t}_{\Lambda'}(A', A'\otimes_A\mathcal{F}_*(F)).
 \]

\begin{rem}\label{rem:mapss}
 The spectrum $tmf\wedge X(4)$ is a complex oriented ring spectrum (e.g., $A = \pi_*(tmf \wedge X(4))$ is concentrated in even degrees).
Let us denote by 
\[H \colon MU\rightarrow tmf\wedge X(4)\] 
 the map of ring spectra inducing the complex orientation of $tmf\wedge X(4)$ given by the completion of the universal Weierstrass curve at the origin. 
Then $H$ induces a homomorphism of Hopf algebroids 
\[H_* \colon (MU_*, MU_*MU)\rightarrow ((tmf\wedge X(4))_*, (tmf\wedge X(4)\wedge X(4))_*) = (A,\Lambda).\]
Recall that $MU_* \cong \mathbb{Z}[x_1, x_2,\ldots]$ with $|x_i|=2i$ and $MU_*MU \cong MU_*[m_1, m_2, \ldots]$ with $|m_i|=2i$. We note that $H_*(x_1)= \pm a_1$. 
This is discussed in \cite[(3.2)]{tbauer}.

For any finite spectrum $F$, $H$ also induces a map from the Adams--Novikov spectral sequence for $\pi_*(F)$ to the elliptic spectral sequence for $\pi_*(tmf\wedge F)$, which  converges to the Hurewitz map 
$h \colon \pi_*(F) \to \pi_*( tmf\wedge F)$. Moreover, the induced map at the $E_2$ is induced by $H_*$.
 \end{rem}

\subsection{Duality}

In this section, we discuss Brown-Comenetz duality for $tmf$. This will be used for determining some of the exotic extensions in our spectral sequences. First, we introduce the following notation. 

\begin{notation}
Let $A$ be a graded module over a graded commutative ring $S$ and $x\in S$. We let $\Sigma^r A$ be the module determined by $(\Sigma^r A)_t = A_{t-r}$. We denote by $\Gamma_x A$ the $x$-power torsion of $A$, i.e.,
\[
\Gamma_x A = \{m\in A \mid x^{i} m = 0, i \gg 0\},
\]
and by $A/(x^{\infty})$ the module that fits into the exact sequence of $S$-modules
\[ A\rightarrow A\left[\frac{1}{x}\right]\rightarrow A/(x^{\infty}) \rightarrow 0.
\]
We will also denote by $A^{\vee}$ the Pontryagin dual $A$, i.e.,
\[
(A^{\vee})_* = \Hom((A)_{-*}, \Q/\Z)
\] with the $S$-module structure given by $(r.f)(m) = (-1)^{|r||f|}  f(rm)$ for every $r\in S_{|r|}$, $f\in  (A^{\vee})_{|f|}$ and $m\in A_{|m|}$.

Now suppose that $R$ is a commutative ring spectrum (e.g., $R=tmf$) and $M$ is a $R$-module. For any $x\in \pi_*(R)$, we define $M\left[\frac{1}{x}\right]$ to be 
\[M\left[\frac{1}{x}\right]= \hocolim(M\xrightarrow{x} \Sigma^{-|x|}M\xrightarrow{x}\Sigma^{-2|x|} M\xrightarrow{x}...).\]
We define $M/(x^{\infty})$ to be the cofiber of the natural map $M\rightarrow M\left[\frac{1}{x}\right]$.
Inductively, if $(x_1, x_2, ..., x_n)$ is a sequence of element of $\pi_*R$, then we define 
\[M/(x_1^{\infty}, x_2^{\infty}, ..., x_n^{\infty}) = (M/(x_1^{\infty}, x_2^{\infty}, ...,x_{n-1}^{\infty}))/(x_n^{\infty}).\]

With this notation, using the long exact sequence on homotopy groups, we see that the cofiber sequence 
\[ M\rightarrow M\left[\frac{1}{x}\right] \rightarrow M/(x^{\infty})
\]
gives rise to the short exact sequence of $\pi_*(R)$-modules
\[ 0\rightarrow \pi_*(M)/(x^{\infty})\rightarrow \pi_*(M/(x^{\infty}))\rightarrow \Gamma_{x}(\pi_{*-1}(M))\rightarrow 0.
\]

\end{notation}

Let $I_{\Q/\Z}$ be the spectrum representing the Pontryagin dual of stable homotopy groups, so that for a spectrum $X$,
\[   I_{\Q/\Z}^{q}(X): = \Hom(\pi_{q}X, \Q/\Z) .\]
Then the the Brown-Comenetz dual of a spectrum $X$ is defined to be 
\[I_{\Q/\Z}(X) =F(X,I_{\Q/\Z}).\]

The literature contains a variety of references and methods for studying dualities of $tmf$ and related spectra. To name a few, we note work of Mahowald--Rezk \cite{MRBCdual}, of Stojanoska \cite{StojanoskaDuality, Stojanoska-Descent} and of Greenlees \cite{MR3423453}. While the identification of $I_{\Q/\Z}(tmf)$ is known to experts, there is no direct reference in the literature. (The work of Greenlees and Stojanoska \cite{GreenleesStojanoska} describes the relationship between various forms of duality, but this work does not directly apply to $tmf$.)
Upcoming work of Bruner--Rognes \cite[Chapter 10]{BrunerRognesbook} and Bobkova--Stojanoska will soon fill this gap and provide a reference for the following result.

\begin{thm} \label{thm:duality}
There is an equivalence of $tmf$-modules 
\[I_{\mathbb{Q}/{\mathbb{Z}}} (tmf/(2^{\infty}, c_4^{\infty}, \Delta^{\infty})) \simeq \Sigma^{20} tmf.\]
\end{thm}

\begin{rem}
Here and below, ``$-/\Delta^{\infty}$'', we really mean $-/(\Delta^8)^{\infty}$ as $\Delta$ is an element of the $E_2$-term of the elliptic spectral sequence but it does not survive to the $E_{\infty}$-term. However, $\Delta^8$ survives and detects a class in  $\pi_{192}tmf$.
Note also that the class $c_4 \in \pi_8tmf$ reduces to $v_1^4 \in tmf\wedge V(0)$ and so $c_4$-power torsion is the same as $v_1$-power torsion when the latter makes sense.
\end{rem}

\begin{cor}\label{cor:duality}
There are equivalences of $tmf$-modules 
\begin{enumerate}
\item
$I_{\mathbb{Q}/{\mathbb{Z}}} (tmf\wedge V(0)/(2^{\infty}, c_4^{\infty}, \Delta^{\infty})) \simeq \Sigma^{19} tmf\wedge V(0)$, and
\item $I_{\mathbb{Q}/{\mathbb{Z}}} (tmf\wedge Y/(2^{\infty}, c_4^{\infty}, \Delta^{\infty})) \simeq \Sigma^{17} tmf\wedge Y$.
\end{enumerate}
\end{cor}

In the proof of the result below, we use the following lemma.
\begin{lem}\label{lemc4div}
For $\mathcal{X} = tmf \wedge V(0)$ or $tmf\wedge Y$ and $a\in \pi_*\mathcal{X}$, $c_4a$ is divisible by $\Delta^8$ if and only if $a$ is divisible by $\Delta^8$.
\end{lem}

\begin{rem}
The proof makes use of the structure of the $E_{\infty}$-terms of the elliptic spectral sequences as a module over $\F_2[c_4, \Delta^8]$. So this is a bit premature but we want to have this result here to gather all our techniques in one place. We note that the logic of the argument is not circular as the determination of the $E_{\infty}$-terms do not require this lemma. 
\end{rem}
\begin{proof}
Let $\mathcal{X}$ be $tmf\wedge V(0)$. The homotopy groups of $\mathcal{X}$ decompose as 
\[ 0 \to T_{c_4} \to \pi_*\mathcal{X} \to F_{c_4} \to 0
\]
Here, $T_{c_4}$ is the subgroup of $c_4$-torsion elements, and $F_{c_4} =  \pi_*(\mathcal{X})/T_{c_4}$. 
By the calculation of the $E_{\infty}$-term of the elliptic spectral sequence for $\mathcal{X}$, multiplication by $\Delta^8$ induces a bijective endormorphism of $T_{c_{4}}$ in every stem and an injective endormorphism of $F_{c_{4}}$. Furthermore, there are no non-trivial $c_4$-torsion elements in the stems between $176$ and $191$, and hence $T_{c_4}$ satisfies the conclusion of the lemma. Any element that maps non-trivially to $F_{c_4}$ is detected in filtrations less than or equal to 2 of the $E_\infty$-term of the elliptic spectral sequence. This part of the $E_{\infty}$-term is free as a module over $\F_2[v_1^4,\Delta^8]$, and hence satisfies the conclusion of the lemma. (Note that in the elliptic spectral sequence $c_4$ is detected by $v_1^4$.) Now, suppose we have $a \in \pi_*\mathcal{X}$ such that $a\not\in T_{c_4}$ and $c_4a$ is divisible by $\Delta^8$. Then by our remarks on $F_{c_4}$, $a = \Delta^8b+c$ for some element $c\in T_{c_4}$. Since $\Delta^8$ is surjective on $T_{c_4}$, we see that $c$ is in the image of $\Delta^8$ and so the claim holds.

For $\mathcal{X} = tmf\wedge Y$, a similar argument applies.
\end{proof}

\begin{cor} \label{lem:duality}
We have the following isomorphisms of $\pi_*tmf$-modules 
\begin{enumerate}[(1)]
\item $\Gamma_{c_4}(\pi_{*}(tmf\wedge V(0))/(\Delta^{\infty}))^{\vee}\cong \Gamma_{c_4}(\pi_{*-21}(tmf\wedge V(0)))$, and
\item $\Gamma_{c_4}(\pi_{*}(tmf\wedge Y)/(\Delta^{\infty}))^{\vee}\cong \Gamma_{c_4}(\pi_{*-19}(tmf\wedge Y))$.
\end{enumerate}
\end{cor}
\begin{proof}
In this proof, we let $\mathcal{X}=tmf\wedge V(0)$.
Since $\pi_*\mathcal{X}$ is $2$-power torsion, we have
$\mathcal{X}[1/2] \simeq *$.
Thus, 
\begin{equation}\label{Step1}\mathcal{X}/(2^{\infty})\simeq \Sigma \mathcal{X}.
\end{equation}
The long exact sequence in homotopy associated to the cofiber sequence 
\[\mathcal{X}/(2^{\infty})\rightarrow \mathcal{X}/(2^{\infty})\left[\frac{1}{c_4}\right]\rightarrow \mathcal{X}/(2^{\infty}, c_4^{\infty}),\] 
gives an exact sequence
\begin{equation}\label{Step2}0\rightarrow (\pi_*\mathcal{X}/(2^{\infty}))/(c_4^{\infty})\rightarrow  \pi_*(\mathcal{X}/(2^{\infty}, c_4^{\infty}))\rightarrow \Gamma_{c_4}\pi_{*-1}(\mathcal{X}/(2^{\infty}))\rightarrow 0.
\end{equation}
By (\ref{Step1}), we have that 
		\[(\pi_*(\mathcal{X}/(2^{\infty})))/(c_4^{\infty}) \cong (\pi_{*-1}\mathcal{X})/(c_4^{\infty})\]
	and that 
		\[ \Gamma_{c_4}(\pi_{*-1}(\mathcal{X}/(2^{\infty})))\cong \Gamma_{c_4}(\pi_{*-2} \mathcal{X}).\]
		Since $\Delta^{8}$ acts injectively on $\pi_*\mathcal{X}$, it also acts injectively on $\Gamma_{c_4}(\pi_{*-2}\mathcal{X})$. Moreover, $\Delta^8$ acts injectively on $(\pi_*\mathcal{X})/(c_4^{\infty})$ by \Cref{lemc4div}.
		The short exact sequence \eqref{Step2} then shows that $\Delta^{8}$ acts injectively on $\pi_*(\mathcal{X}/(2^{\infty}, c_4^{\infty}))$.
		Therefore, we have that 
		\[\pi_*(\mathcal{X}/(2^{\infty}, c_4^{\infty}, \Delta^{\infty})) \cong (\pi_*\mathcal{X}/(2^{\infty}, c_4^{\infty}))/(\Delta^{\infty}).\] 
		The 9-lemma then implies that the following is a short exact sequence of $\pi_*tmf$-modules:
	\begin{equation}\label{Step3}0\rightarrow (\pi_{*-1}\mathcal{X})/(c_4^{\infty},\Delta^{\infty})\rightarrow  \pi_*(\mathcal{X}/(2^{\infty}, c_4^{\infty},\Delta^{\infty}))\rightarrow \Gamma_{c_4}(\pi_{*-2}\mathcal{X})/(\Delta^{\infty})\rightarrow 0.
\end{equation}
 By applying $\Hom(-, \Q/\Z)$ to this exact sequence, we obtain that 
 \[0\to (\Gamma_{c_4}(\pi_{*-2}\mathcal{X})/(\Delta^{\infty}))^{\vee}\rightarrow \pi_*(\mathcal{X}/(2^{\infty}, c_4^{\infty},\Delta^{\infty}))^{\vee} \rightarrow ((\pi_{*-1}\mathcal{X})/(c_4^{\infty},\Delta^{\infty}))^{\vee}\rightarrow 0,\]
 is an exact sequence of $\pi_*tmf$-modules.
 
 We see that the right most term is $c_4$-free and the left most term is $c_4$-torsion. In particular, it follows that 
 \begin{align*}\label{Step4}
  (\Gamma_{c_4}(\pi_{*-2}\mathcal{X})/(\Delta^{\infty}))^{\vee}&\cong \Gamma_{c_4}(\pi_*(\mathcal{X}/(2^{\infty}, c_4^{\infty},\Delta^{\infty}))^{\vee})\\ 
  &\cong \Gamma_{c_4}(\pi_{*}I_{\Q/\Z} (\mathcal{X}/(2^{\infty}, c_4^{\infty}, \Delta^{\infty}) )),
 \end{align*}
where the second isomorphism comes from the definition of the Brown-Comenetz dual $I_{\Q/\Z}(\mathcal{X}/(2^{\infty}, c_4^{\infty}, \Delta^{\infty}) )$.
 Together with \Cref{cor:duality}, we obtain an isomorphism of $\pi_*tmf$-modules
 \begin{align*}
 (\Gamma_{c_4}(\pi_*\mathcal{X})/(\Delta^{\infty}))^{\vee}&\cong \Sigma^{2} (\Gamma_{c_4}(\pi_{*-2}\mathcal{X})/(\Delta^{\infty}))^{\vee}\\
 &\cong \Sigma^{2} \Gamma_{c_4} \pi_*(I_{\Q/\Z}(\mathcal{X}/(2^{\infty}, c_4^{\infty}, \Delta^{\infty}) ))\\
 &\cong \Sigma^{2}\Sigma^{19}\Gamma_{c_4}(\pi_*\mathcal{X})\\
 &\cong \Sigma^{21}\Gamma_{c_4}(\pi_*\mathcal{X}). 
\end{align*}

Substituting $\mathcal{X}$ for $tmf\wedge Y$ and this last $19$ with $17$ gives the result for $Y$. 
\end{proof}

\begin{rem}\label{rem:usingduality}\label{lem:duality-for-extensions}
We will explain how to use \Cref{lem:duality} to compute extensions. Continue to let $\mathcal{X}=tmf\wedge V(0)$.
Let $K$ denote the kernel of the homomorphism induced by multiplication by $\Delta^8$ on $\Gamma_{c_4}(\pi_*\mathcal{X})/(\Delta^{\infty})$. Since multiplication by $\Delta^{8}$ induces an isomorphism 
\begin{equation}
\label{deltaisomorphism2}
\Gamma_{c_4}(\pi_*\mathcal{X}) \xrightarrow{\cong} \Gamma_{c_4}(\pi_{*+192}\mathcal{X})
\end{equation}
for $*\geq 0$, we see that, for $-192\leq t<0$,
\[ K_t \cong \Gamma_{c_4}(\pi_*\mathcal{X})/(\Delta^{\infty})_t.
\]
The Snake Lemma applied to the following diagram
\[\xymatrix{ 
&0 \ar[d]\ar[r] & \Gamma_{c_4}(\pi_*\mathcal{X})\left[\frac{1}{\Delta^8}\right] \ar[r]^{\Delta^8} \ar[d] &\Gamma_{c_4}(\pi_*\mathcal{X})\left[\frac{1}{\Delta^8}\right] \ar[r] \ar[d] &0\\
0\ar[r]&K \ar[r]& \Gamma_{c_4}(\pi_*\mathcal{X})/(\Delta^{\infty})\ar[r]^{\Delta^8} & \Gamma_{c_4}(\pi_*\mathcal{X})/(\Delta^{\infty})&
}\]
gives rise to the exact sequence
\[0\rightarrow \Gamma_{c_4}(\pi_*\mathcal{X})\xrightarrow{\Delta^{8}} \Gamma_{c_4}(\pi_{*+192}\mathcal{X})\rightarrow K\rightarrow 0 .
\]
Using \eqref{deltaisomorphism2} again, the homomorphism $\Gamma_{c_4}(\pi_{*+192}\mathcal{X})\rightarrow K$ in the above short exact sequence induces an isomorphism
\[\Gamma_{c_4}(\pi_*\mathcal{X})_t \rightarrow K_{t-192} \cong \Gamma_{c_4}\pi_*(\mathcal{X}/(\Delta^{\infty}))_{t-192}
\]
for $0\leq t<192$. 

Now let $r$ be an element of $\pi_{l}(tmf)$. If $0\leq k<192-l$, multiplication by $r$ induces a commutative diagram
\[\xymatrix{\Gamma_{c_4}(\pi_{*}\mathcal{X})_k \ar[r]^-{\cong}\ar[d]^{r} & K_{k-192} \cong \Gamma_{c_4}(\pi_*\mathcal{X})/(\Delta^{\infty})_{k-192}\ar[d]^{r}\\
\Gamma_{c_4}(\pi_*\mathcal{X})_{k+l} \ar[r]^-{\cong} & K_{k+l-192}\cong \Gamma_{c_4}(\pi_*\mathcal{X})/(\Delta^{\infty})_{k+l-192}.
}\]
By applying the Pontryagin dual to this commutative diagram, together with \Cref{lem:duality}, we obtain the commutative diagram
\[\xymatrix{\Hom(\Gamma_{c_4}(\pi_{*}\mathcal{X})_k,\Q/\Z)  &  \Gamma_{c_4}(\pi_*\mathcal{X})_{171-k} \ar[l]_-{\cong}\\
\Hom(\Gamma_{c_4}(\pi_*\mathcal{X})_{k+l}, \Q/\Z)\ar[u]^-{r^{\vee}} &  \Gamma_{c_4}(\pi_*\mathcal{X})_{171-k-l}\ar[u]^-{r} \ar[l]_-{\cong}.
}\]
As a consequence, the cardinality of the image of 
\[r \colon\Gamma_{c_4}(\pi_{*}\mathcal{X})_k \rightarrow \Gamma_{c_4}(\pi_*\mathcal{X})_{k+l}\] is the same as that of 
\[r \colon \Gamma_{c_4}(\pi_{*}\mathcal{X})_{171-k-l} \rightarrow \Gamma_{c_4}(\pi_*\mathcal{X})_{171-k}.\] 
In particular, this means that a non-trivial multiplication by $r$ on stem $k$ forces a non-trivial multiplication by $r$ on stem $171-k-l$.

Similarly, for $tmf\wedge Y$ we obtain that a non-trivial multiplication by $r$ on stem $k$ forces a non-trivial multiplication by $r$ on stem $173-k-l$.
\end{rem}

\subsection{The Geometric Boundary Theorem}
We also make use of the following result, due to Bruner  \cite{Brunergeo}. A standard reference is Theorem 2.3.4 of \cite{ravgreen}. We apply this Theorem 2.3.4 to the $X(4)$-based Adams-Novikov spectral sequence and the cofiber sequence
\begin{align*}
tmf \wedge S^0 \xrightarrow{2} tmf  \wedge S^0 \xrightarrow{i} tmf \wedge V(0) \xrightarrow{p} tmf \wedge S^1. \end{align*}
Using $X(4)_*tmf \cong A$ and $X(4)_*(tmf\wedge V(0)) \cong A/2$, we have $X(4)_*p=0$ and hence a short exact sequence
\begin{align}\label{coftmfv0}
0 \to A  \xrightarrow{2} A \to A/2 \to 0 .
\end{align}

\begin{thm}[Geometric Boundary Theorem]\label{lem:half-the-extensions-pre}\label{lem:half-the-extensions}\label{geobound}
There are maps 
\[\delta_r \colon E_r^{s,t}(V(0)) \to E_r^{s+1,t}(S^0)\] such that
\[\delta_2 =\delta \colon E_2^{s,t}(V(0) ) \to E_2^{s+1,t}(S^0)\]
is the connecting homomorphism arising from \eqref{coftmfv0}.
 For all $r$, 
\[\delta_r d_r = d_r \delta_r\]
and $\delta_{r+1}$ is induced by $\delta_r$. Furthermore, $\delta_{\infty}$ is a filtered form of 
\[p_* \colon \pi_*tmf\wedge V(0) \to \pi_{*+1}tmf.\]
\end{thm}

\subsection{Further observations on extensions}
Here,  we collect a few classical but useful extension results. Note that, in this paper, we use Definition 2.10 of \cite{morestems} as our definition of an \emph{exotic extension}. See Section 2.1 of that reference for a detailed discussion. However, briefly, we have
\begin{defn}[{Definition 2.10 \cite{morestems}}]\label{defnexoext}
 Let $\alpha \in \pi_*tmf$ be an element detected by $a$ on the $E_{\infty}$-term of the elliptic spectral sequence for $tmf$.
An \emph{exotic extension by $\alpha$} is a pair of elements $b$ and $c$ on the $E_{\infty}$-term of the elliptic spectral sequence for $M$ (where $M$ is a $tmf$-module) such that
\begin{enumerate}[(1)]
\item $ab =0$ on the $E_{\infty}$-term,
\item there is an element $\beta$ detected by $b$ such that $\alpha \beta$ is detected by $c$,
\item if an element $\beta'$ detected by $b'$ is such that $\alpha \beta'$ is detected by $c$, then the filtration of $b'$ is less than or equal to that of $b$.
\end{enumerate}
\end{defn}
Note that this implies that if both $\alpha \beta$ and $\alpha \beta'$ are detected by $c$ as in \Cref{exoext}, there is no exotic extension from $b'$ to $c$.
\begin{figure}[h]
\includegraphics[width=0.4\textwidth]{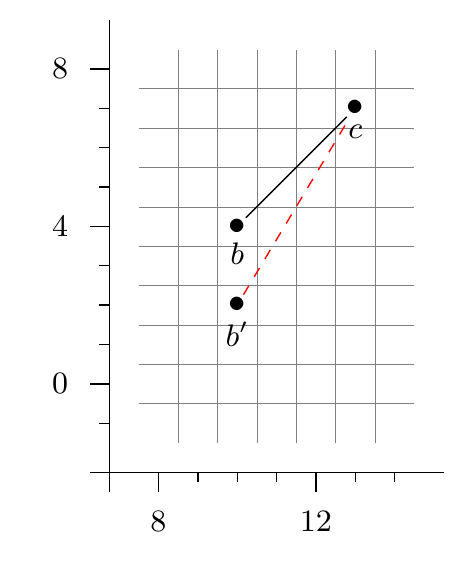}
\caption{  Here there is no exotic extensions from $b'$ to $c$, and so the dashed line would not be drawn.}
\label{exoext}
\end{figure}

\begin{lemma}\label{lem2exttrick}
Let $X$ be a spectrum. Consider the long exact sequence in homotopy
\[
\ldots \to \pi_{n} X \xra{i} \pi_n( X \wedge V(0)) \xra{p} \pi_{n-1} X \xrightarrow{2} \ldots
\]
associated to the cofiber sequence $X \xra{2} X \to X \wedge V(0)$.
Let $a \in \pi_{n-1}X$ be an element of order $2$. If $a' \in \pi_{n}( X \wedge V(0))$ is such that $p_*(a')=a$, then 
\[2a'=i_*(\eta a) \in \pi_{n}X \wedge V(0).\] 
\end{lemma} 
\begin{proof}
This is a classical result. See, for example, \cite[Lemma 3.1.5.]{BGH}. 
\end{proof}

\begin{rem}
\Cref{lem2exttrick} will be used with $X=tmf$ and $tmf\wedge C_{\eta}$ where $C_{\eta}$ is the cofiber of the Hopf map $\eta \colon S^1 \to S^0$.
This gives all exotic 2-extensions in the elliptic spectral sequences for $tmf \wedge V(0)$ and $tmf\wedge Y$, since $Y \simeq C_{\eta}\wedge V(0)$. 
\end{rem}

Finally, we have the following classical result which is an analogue of \Cref{lem2exttrick}.
\begin{lem}\label{propextnutrick}
Let $b \in \pi_nX$ be such that $\eta b=0$. If $b' \in \pi_{n+2}(C_{\eta} \wedge X)$ is such that $p_*b' =b$ in the long exact sequence on homotopy groups associated to 
\[ \Sigma X \xrightarrow{\eta} X \xrightarrow{i} X\wedge C_{\eta}  \xrightarrow{p} \Sigma^2 X , \]
then
$\eta b' = i_*(\nu b)$.
\end{lem}
\begin{proof}
First, consider $b = \iota \in \pi_0C_{\eta}$ given by the inclusion $S^0 \to C_{\eta}$ of the bottom cell.  We have a cofiber sequence
\[ C_{\eta} \xrightarrow{i} C_{\eta} \wedge C_{\eta} \xrightarrow{p} \Sigma^2 C_{\eta}\]
which is not split because of the non-triviality of $Sq^4$ in $H^*(C_{\eta}\wedge C_{\eta}, \Z/2)$.
We get a diagram
\[\xymatrix{ 
\pi_{2}C_{\eta}  \ar[r]^-{i_*} \ar[d]^-\eta& \pi_2C_{\eta} \wedge C_{\eta} \ar[r]^-{p_*}  \ar[d]^-{\eta} & \pi_2\Sigma^2C_{\eta} \ar[r] & 0 \\
\pi_{3}C_{\eta}  \ar[r]^-{i_*} & \pi_{3} (C_{\eta} \wedge C_{\eta})  & &  }\]
For any $b' \in \pi_2(C_{\eta}\wedge C_{\eta})$ such that $p_*b'=\iota$, we must have $\eta b' \neq 0$, else we could split the above cofiber sequence. Since $\eta \iota=0$, $\eta b' \in i_*(\pi_3C_{\eta})$, where $\pi_3C_{\eta} \cong \Z/4\{\nu \iota\}$. Now, in $\pi_*C_{\eta}$, we can form the bracket
\[ \langle \iota, \eta, 2\rangle \in \pi_2C_{\eta}\] 
with indeterminacy 
\[2\pi_2 C_{\eta} + \iota \pi_2 S^0  = 2 \pi_2C_{\eta} \cong 2 \Z.\] 
So, $\langle \iota, \eta, 2\rangle \eta $ contains a unique element.
In $\pi_*S^0$, we also have $2\nu \in \langle \eta, 2, \eta\rangle $ with indeterminacy $\eta\pi_2S^0$. 
It follows that
\[ \langle \iota, \eta, 2\rangle \eta     = \iota  \langle \eta, 2, \eta\rangle = \iota 2\nu \neq 0 \in \pi_3C_{\eta}.  \]
So $i_*(2\nu)=0$ and $\eta b' = i_*(\nu \iota)$. 

For the general case, note that any class $b \colon S^n \to X$ such that $\eta b=0$ can be extended to a map $\bar{b} \colon \Sigma^n C_{\eta} \to X$. The claim then follows from the commutativity of the following diagram
\[\xymatrix{\Sigma^n C_{\eta} \ar[r] \ar[d]^-{\bar b}& \Sigma^n C_{\eta}\wedge C_{\eta} \ar[r] \ar[d]^-{ \bar b \wedge C_{\eta}}  & \Sigma^{n+2} C_{\eta} \ar[d]^-{\Sigma^2 \bar b} \\
X \ar[r]^-{i} & X \wedge C_{\eta} \ar[r]^-{p} & \Sigma^2 X } \] 
Then $b' = (\bar b \wedge C_\eta)_*\iota$ satisfies $\eta b' = i_*(\nu b)$. Now, suppose that $p_*\tilde{b}' = b$. Then $\tilde{b}' -b' \in \ker p_* =\im i_*$. Therefore, $\eta (\tilde{b}' -b' ) =0$ so, $\eta \tilde{b}' = i_*(\nu b)$ as well.
\end{proof}

\subsection{Self-maps and their cofiber}\label{v1selfmaps}

It is well-known that $V(0)$ admits $v_1^4$ self-maps, i.e., maps $\Sigma^8 V(0) \to V(0)$ which induce multiplication by $v_1^4$  in $K(1)$-homology for $K(1)$ the first Morava $K$-theory. The map on $MU$-homology is given by multiplication by $x_1^4\in MU_8$. Under the map from the Adams-Novikov spectral sequence of $V(0)$ to that of the ellpitic spectral sequence of $tmf \wedge V(0)$, $x_1$ maps to $v_1$ on the $E_2$-term. See the discussion surrounding \eqref{eqx1a1v1}. Any $v_1^4$ self-map is detected by the same-named element. The spectral sequence inherits an action of $v_1^4$ and the differentials are $v_1^4$-linear.

Recall that we let $Y$ be the spectrum $V(0) \wedge C_\eta$. In \cite{DM}, Davis and Mahowald show that there exist $v_1$ self-maps of $Y$, i.e., maps $\Sigma^2 Y\rightarrow Y$ which induce multiplication by $v_1$ in $K(1)_*Y$. Any of these is detected by the element  $v_1$ on the $E_2$-term of elliptic spectral sequence for $tmf \wedge Y$ and the differentials are $v_1$-linear.

In \Cref{lemv1ext}, we will be studying the $v_1$-multiplication in $tmf_*Y$. Some of the answers will \emph{depend on the choice of $v_1$-self map}, so we give a bit of background here on this subject. This material can be found in \cite{DM}.

In \cite{DM}, the authors show that there are in fact $8$ $v_1$-self maps of $Y$.
They also show that a $v_1$-self map of $Y$ is detected in the Adams spectral sequence by an element of $\Ext^{1,3}_{\mathcal{A}}(H^*(Y), H^*(Y))$, where $\mathcal{A}$ denotes the Steenrod algebra at the prime $2$. 

A class of $\Ext^{1,3}_{\mathcal{A}}(H^*(Y), H^*(Y))$ is represented by a short sequence of $\mathcal{A}$-modules:
\[ 0 \rightarrow \Sigma^2 H^*(Y) \rightarrow M \rightarrow H^*(Y)\rightarrow 0.
\]
Let $\mathcal{A}(1)$ be the sub-algebra of the Steenrod algebra generated by $Sq^1$ and $Sq^2$.
We know that $\Ext^{1,3}_{\mathcal{A}(1)}(H^*(Y), H^*(Y)) \cong \F_2$ and its unique non-trivial class is represented by the short exact sequence of $\mathcal{A}(1)$-module
\[ 0 \rightarrow \Sigma^2 H^*(Y) \rightarrow A(1) \rightarrow H^*(Y)\rightarrow 0,
\]
where $A(1)$ is isomorphic to $\mathcal{A}(1)$ as an $\mathcal{A}(1)$-module, thus the notation. 
Davis and Mahowald showed that a class of $\Ext^{1,3}_{\mathcal{A}}(H^*(Y), H^*(Y))$ which detects a $v_1$-self map of $Y$ is sent to the unique non-trivial class of $\Ext^{1,3}_{\mathcal{A}(1)}(H^*(Y), H^*(Y))$ (via the map induced by the inclusion $\mathcal{A}(1)\subset \mathcal{A}$). 

To put an $\mathcal{A}$-module structure on $A(1)$, it suffices to specify the $Sq^4$ action. Indeed, the action of $Sq^{k}$, for $k\geq 8$ on $A(1)$ is trivial for degree reasons. By the Adem relations, there must be a non-trivial $Sq^4$ on the class of degree one of $A(1)$. There are possibilities for a non-trivial action of $Sq^4$ on the classes of degrees zero and two, giving rise to four different $\mathcal{A}$-module structures on $A(1)$. This implies, in particular, that 
\[ \Ext^{1,3}_{\mathcal{A}}(H^*(Y), H^*(Y)) \cong \F_2^{\oplus 3}.
\]
Computing the first three stems of $\Ext^{s,t}_{\mathcal{A}}(H^*(Y), H^*(Y))$, we see that 
\[ \Ext^{s,s+2}_{\mathcal{A}}(H^*(Y), H^*(Y)) \cong \begin{cases}
												\F_2 & \mbox{if} \ s= 2\\
												0 & \mbox{otherwise}.  
											\end{cases} 	
\]
We deduce that there are eight  homotopy classes of maps $\Sigma^2 Y \rightarrow Y$ detected in $\Ext^{1,3}_{\mathcal{A}}(H^*(Y), H^*(Y))$ and mapping non-trivially to $\Ext^{1,3}_{\mathcal{A}(1)}(H^*(Y), H^*(Y))$. These are the $v_1$ self-maps of $Y$. 

The singular cohomology of the cofiber of each of the $v_1$-self map is isomorphic to one of the four $A(1)$s as an $\mathcal{A}$-module. We denote the four choices by $A_1[ij]$, with $i,j \in \{0,1\}$. Here, $A_1[ij]$ means that the cohomology has a non-trivial $Sq^4$ on the class of degree $0$, respectively $2$ if $i=1$, respectively if $j=1$.

It is somewhat surprising that out of eight $v_1$-self-maps, there are only four homotopy types which are distinguished by their cohomology, as is shown \cite{DM}. We use the notation $A_1$, for short, when we mean any or all of the four models.


\section{$tmf_*V(0)$: The $E_2$-page}\label{secV0E2}

From now on, we will be working exclusively with $2$-local spectra. We will write $tmf$ for $tmf_{(2)}$ to simplify the notation. Furthermore, we will be considering only elliptic spectral sequences for $M = tmf \wedge F $ for $F$ a finite spectrum and so shorten our notation even more to
  \[
E_2^{s,t}(F):=  \Ext^{s,t}_{\Lambda'}(A', A'\otimes_A\mathcal{F}_*(F)).
 \]

The map $S^0\xrightarrow{\times 2} S^0$ induces multiplication by $2$ on $\mathcal{F}_*(S^0)\cong A$, which is injective. Thus the cofiber sequence 
\[S^0\xrightarrow{2}S^0\rightarrow V(0)\] gives rise to 
a short exact sequence of $\Lambda'$-comodules 
\begin{align}\label{eqsss}
0\rightarrow A' \xrightarrow{\times 2} A' \rightarrow A'\otimes_{A}\mathcal{F}_*(V(0))\rightarrow 0.
\end{align}
It follows that $A'\otimes_{A}\mathcal{F}_*(V(0))$ is isomorphic to $A'/(2)$ as a $\Lambda'$-comodule. Since $(2)\subseteq A'$ is a $\Lambda'$-invariant ideal, we have that
 \[\Ext^{s,t}_{\Lambda'}(A', A'/(2))\cong \Ext^{s,t}_{\Lambda'/(2)}(A'/(2), A'/(2)).\]
 See, for example, \cite[Proposition A1.2.16]{ravgreen}. So, we have a spectral sequence
  \begin{equation}\label{eq:ss-V0}
 E_2^{s,t}(V(0))=\Ext^{s,t}_{\Lambda'/(2)}(A'/(2), A'/(2)) \Longrightarrow   \pi_*tmf\wedge V(0).
 \end{equation}

 A computation of the cohomology of $(A'/(2), \Lambda'/(2))$ 
is originally due to Hopkins and Mahowald and can be found in \cite[Chapter 15, Section 7]{tmfbook} and  \cite[Section 7]{tbauer}. 
Let us describe the answer here and introduce some notation.

Classical computations of modular forms yield
\begin{align*}
\Ext^{0,*}_{\Lambda'}(A', A') \cong \ZZ_{(2)}[c_4, c_6, \Delta]/{(c_4^3 -c_6^2-(12)^3 \Delta )} 
\end{align*}
where
\begin{align*}
c_4 &= a_1^4-24a_1a_3   \\
c_6 &= -a_1^6+36a_1^3a_3-216a_3^2 \\
\Delta &=a_1^3a_3^3-27a_3^4
\end{align*} 
as well as
\begin{align*}
\Ext^{0,*}_{\Lambda'/(2)}(A'/(2), A'/(2))\cong \ZZ/2[a_1, \Delta].
\end{align*}
See, for example \cite{tbauer} and \cite[III.1]{Silverman}. The map on $\Ext^{0,*}$ induced by the mod 2 reduction $(A', \Lambda') \to (A'/(2), \Lambda'/(2))$ 
sends  $c_4 \mapsto a_1^4$ and
 $c_6 \mapsto a_1^6$.

There are also maps of Adams--Novikov Spectral Sequences, where $H$ and $h$ are as in \Cref{rem:mapss}:
\[
\xymatrix{
\Ext^{*,*}_{BP_*BP}(BP_*, BP_*V(0)) \ar[r] &\pi_*V(0)\\
\Ext^{*,*}_{MU_*MU}(MU_*, MU_*V(0)) \ar[u]^-{\cong}\ar[d]^-H \ar[r] &\pi_*V(0)\ar[u]^-{\cong}\ar[d]^-h \\
\Ext^{*,*}_{\Lambda'/(2)}(A'/(2), A'/(2)) \ar[r] &\pi_*tmf\wedge V(0)
 }
\]
Further,
 \[
\Ext^{0,*}_{BP_*BP}(BP_*, BP_*V(0)) \cong \mathbb{F}_2[v_1];
\]
see \cite[Thm 4.3.2]{ravgreen}. 

So, we have $a_1 \in \Ext^{0,2}_{\Lambda'/(2)}(A'/(2), A'/(2))$, $v_1 \in \Ext_{BP_*BP}^{0,2}(BP_*,BP_*V(0))$ and $x_1 \in \Ext_{MU_*MU}^{0,2}(MU_*,MU_*V(0))$, and 
\begin{align}\label{eqx1a1v1}
v_1 \mapsfrom x_1   \mapsto a_1.
\end{align}
Note that $v_1$ detects either of the two classes in $\pi_2V(0) \cong \Z/4$ which map to $\eta \in \pi_1 V(0)$ under the homomorphism $\pi_2 V(0) \to \pi_1 S^0$ in the long exact sequence in homotopy. We fix a choice and call it $v_1\in \pi_2V(0)$.
 It follows that $a_1$ survives to detect the image of $v_1 \in \pi_2V(0)$ in $\pi_2tmf\wedge V(0)$. From now on, in mod $2$ computations, we  abuse notation and denote all classes we have named $a_1$ by $v_1$.

 Now we will present the $E_2$ page of \eqref{eq:ss-V0} as computed in \cite[p. 270]{tmfbook}, \cite[Fig. 5]{Stojanoska-Descent} and \cite[p.26]{tbauer}. See \Cref{tmf-V0-E2}.
Even if the elliptic spectral sequence for $V(0)$ is not multiplicative, $E_2(V(0))$ is a ring and we can completely describe the algebraic relations (which also follow from \cite{tbauer}). The ring structure will be used in our computation of $E_2(Y)$ below.

Recall that $\delta=\delta_2$ was defined in \Cref{lem:half-the-extensions-pre}.
In the theorem below,  $\kappa \in E_2^{2,16}(S^0)$ is the unique non-zero element.
\begin{thm}[\Cref{tmf-V0-E2}]\label{thmE2V0all}
The ring $E_2(V(0))$ is isomorphic to
\[\F_2[v_1, \Delta, \kappabar, \eta, \nu, x,y]/(\sim)\]
for elements 
\[\eta \in \Ext^{1,2}, \ \ \  \nu \in \Ext^{1,3}, \ \ \  \kappabar \in \Ext^{0,24}, \ \ \ \Delta \in \Ext^{0,24}\] 
in the image of $E_2(S^0) \to E_2(V(0))$, as well as elements
\[v_1 \in \Ext^{0,2}, \ \ \  x\in \Ext^{1,8}, \ \ \  y\in \Ext^{1,15} \] 
in the image of $\delta_2 \colon  E_2(V(0))  \to E_2(S^0)$ where
\[\delta_2(v_1) =\eta, \ \ \delta_2(x) = \nu^2,  \ \ \delta_2(y)=\kappa.\]
 The relations $(\sim)$ is the ideal generated by
\begin{align*}
(s=1) & &  &v_1 \nu & & v_1^2 x  & & v_1y \\
 (s=2) & & & \nu \eta  & & \nu x-v_1\eta x & &   \eta y-v_1x^2  & & xy & & y^2-\nu^2\Delta  \\
 (s=3) &  & & \eta^2x-\nu^3  & &   x^3-\nu^2y  \\
 (s=4)&     & &\eta^4\Delta-v_1^4\kappabar \ . 
\end{align*}
Furthermore, we have $\kappa=x^2$ and $\delta_2(\nu^2 y)=4\kappabar$.
\end{thm}
\begin{rem}
The algebraic structure in \Cref{thmE2V0all} can also be deduced from the appendix of \cite{BeaudryMoore}.
\end{rem}

 \begin{rem}
The element $\Delta$ is detected by $v_2^4$ in the Bockstein spectral sequence computation of \cite[II.2.7]{tmfbook}.
\end{rem}

\begin{figure}[h]\label{fig:E2-V0-full}
\includegraphics[page=1, width=\textwidth]{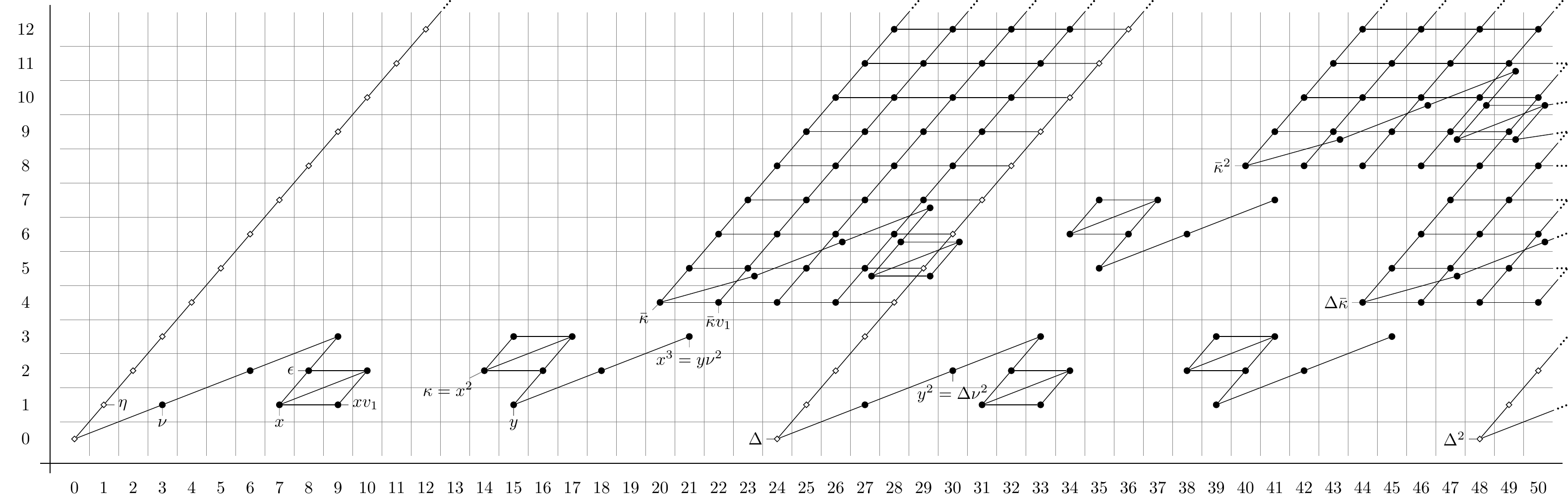}
\caption{The $E_2$-term of the elliptic spectral sequence for $tmf\wedge V(0)$. A bullet $\bullet$ denotes $\mathbb{F}_2$ and a diamond $\diamond$ denotes a copy of $\mathbb{F}_2[v_1]$. The lines of slope 1 denote multiplication by $\eta$, and the lines of slope 1/3 denote multiplication by $\nu$. Horizontal lines are $v_1$-multiplications.}
\label{tmf-V0-E2}
\end{figure}
 
 \begin{rem}
Let $P$ denote the following pattern: 
\begin{center}
  \includegraphics[page=1, width=0.8\textwidth]{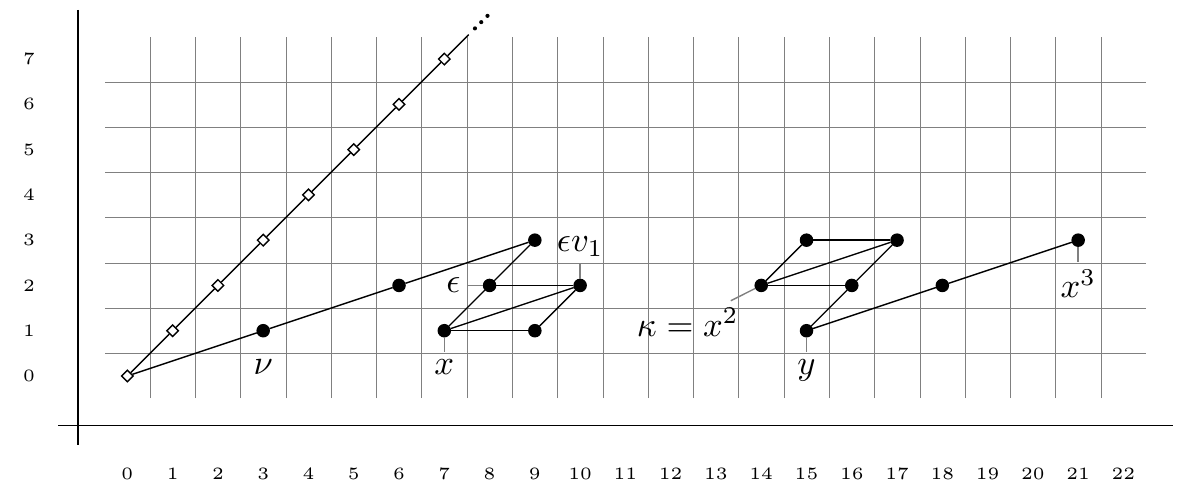}
\end{center}
Then $E_2^{*,*}(V(0))$ can be summarized additively as
\[
E_2^{*,*}(V(0))=P[\kappabar, \Delta]/(\Delta\eta^4-\kappabar v_1^4).
\]
\end{rem}


\section{$tmf_*V(0)$: The differentials and extensions}\label{secV0diffext}
We begin with an observation that $V(0)$ has a $v_1^4$ self map, hence all differentials $d_r$ for $r\geq 3$ are $v_1^4$ linear. Since $\eta$, $\nu$, $\kappabar$ and $\Delta^8$ are permanent cycles, all differentials are linear with respect to multiplication by these elements. Note that there are no even length differentials due to sparseness. 

We will use the following methods when computing differentials in this section.
\begin{enumerate}
 \item The map of spectral sequences induced by the map of spectra 
 \[tmf \to tmf \wedge V(0)\] allows us to import a differential $d_r(a)=b$ from the spectral sequence for $tmf$ if the images of $a$ and $b$ are both non-trivial on the $E_r$ page of the spectral sequence for $tmf \wedge V(0)$. Note also that the elliptic spectral for $tmf \smsh V(0)$ is a module over the elliptic spectral sequence for $tmf$.
 \item The long exact sequence in homotopy groups associated to the fiber sequence 
 \[
tmf \xrightarrow{2} tmf \to tmf \wedge V(0) 
 \]
 gives short exact sequences
 \[ 0 \to (\pi_itmf)/2 \to \pi_i (tmf \wedge V(0)) \to   \ker_2(\pi_{i-1}tmf) \to 0\]
 where $  \ker_2(\pi_{i-1}tmf) $ is the subgroup of elements of order $2$.
This allows us to compute the rank of $\pi_i (tmf \wedge V(0))$ and forces certain differentials by various dimension count arguments.
\item The Geometric Boundary Theorem, stated in \Cref{lem:half-the-extensions-pre}.
\end{enumerate}

\subsection{The $d_3$-differentials}
\begin{lem}[\Cref{V0-d3}]\label{lem:d3V0}
 The $d_3$-differentials are  $\Delta$ and $v_1^4$-linear. They are determined by this linearity, the differentials
 \[d_3(v_1^2)=\eta^3; \quad d_3(v_1^3)=v_1\eta^3,\]
 and the module structure over the elliptic spectral sequence for $tmf$.
\end{lem}
\begin{proof}
The two listed $d_3$-differentials occur in the Adams-Novikov spectral sequence computing $\pi_*V(0)$ so happen here also by naturality. 
See, for example, \cite[Theorem 5.13 (a)]{ravnovice}. Since $\Delta$ is a $d_3$-cycle in the elliptic spectral sequence computing $\pi_*tmf$ and the elliptic spectral sequence for $V(0)$ is a module over this spectral sequence, the $d_3$-differentials are $\Delta$-linear. For degree reasons (making use of $\Delta$ and $\kappabar$-linearity), these determine all $d_3$-differentials.
\end{proof}

\begin{figure}
\includegraphics[page=1, width=\textwidth]{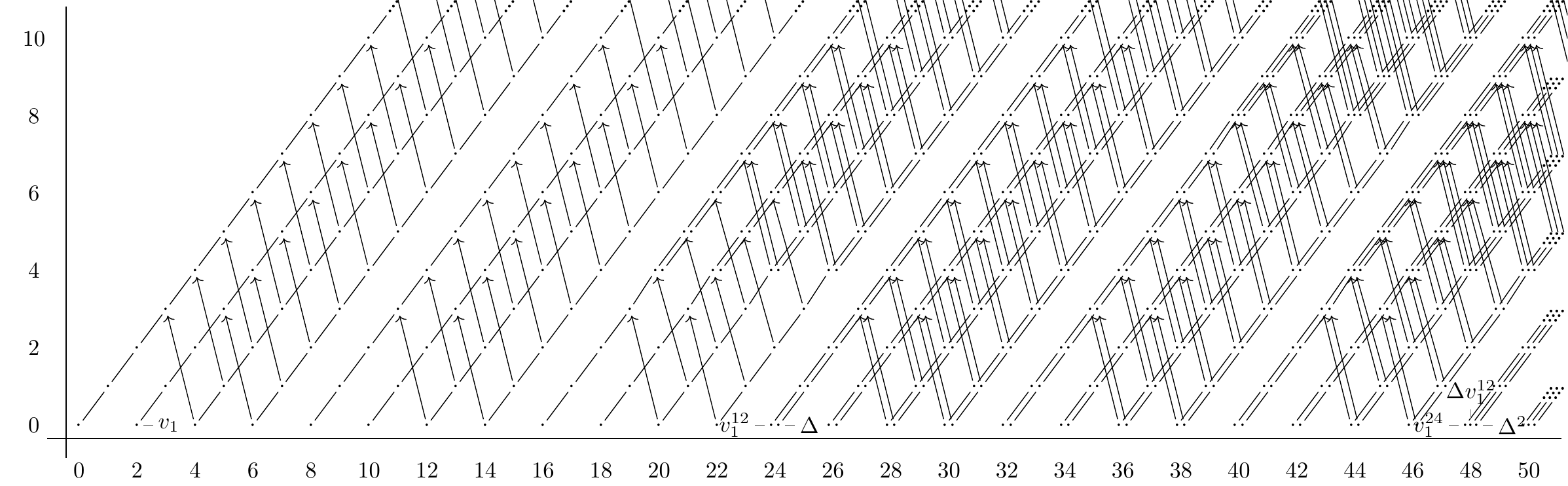}
\caption{The $d_3$-differentials}
\label{V0-d3}
\end{figure}

\begin{figure}
\includegraphics[page=1, width=\textwidth]{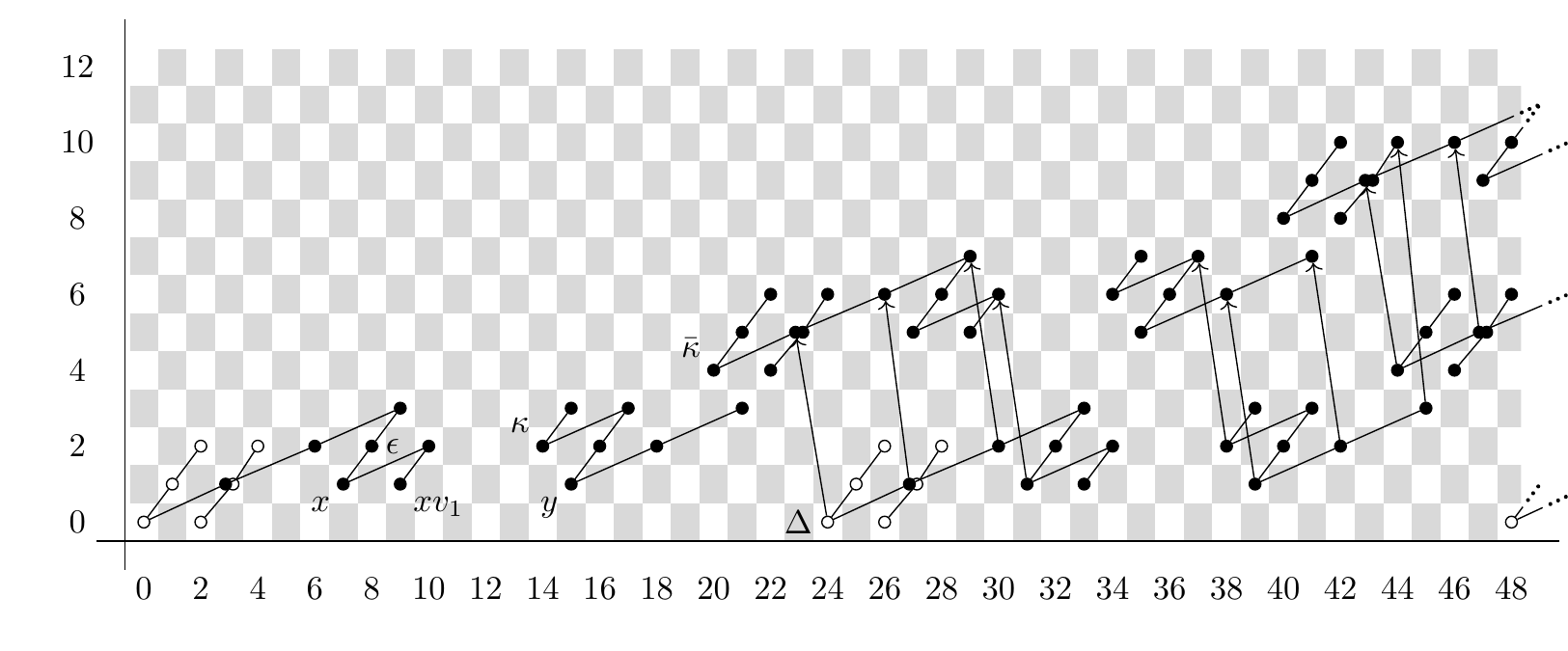}
\caption{$d_5$ and $d_7$-differentials in stems 0-48. A $\circ$ denotes $\F_2[v_1^4]$.}
\label{V0-0-50}
\end{figure}

\begin{rem}
On the $E_5$-page, all classes in filtrations $s\geq 3$ are $v_1^4$-torsion. The $v_1^4$-free classes are concentrated in stems $t-s \not\equiv 5,6,7 \mod 8$.
\end{rem}

\subsection{The $d_5$-differentials}
\begin{lem}[\Cref{V0-0-50}]
The $d_5$-differentials are $\Delta^2$-linear. They are determined by this linearity, the differential
\[
 d_5(\Delta)=\kappabar\nu ,
\]
 and the module structure over the elliptic spectral sequence for $tmf$.
\end{lem}
\begin{proof}
The same differential occurs in the spectral sequence for $\pi_*tmf$. The rest of the argument is as in the proof of \Cref{lem:d3V0}.
\end{proof}

   \begin{figure}
  \includegraphics[width=\textwidth]{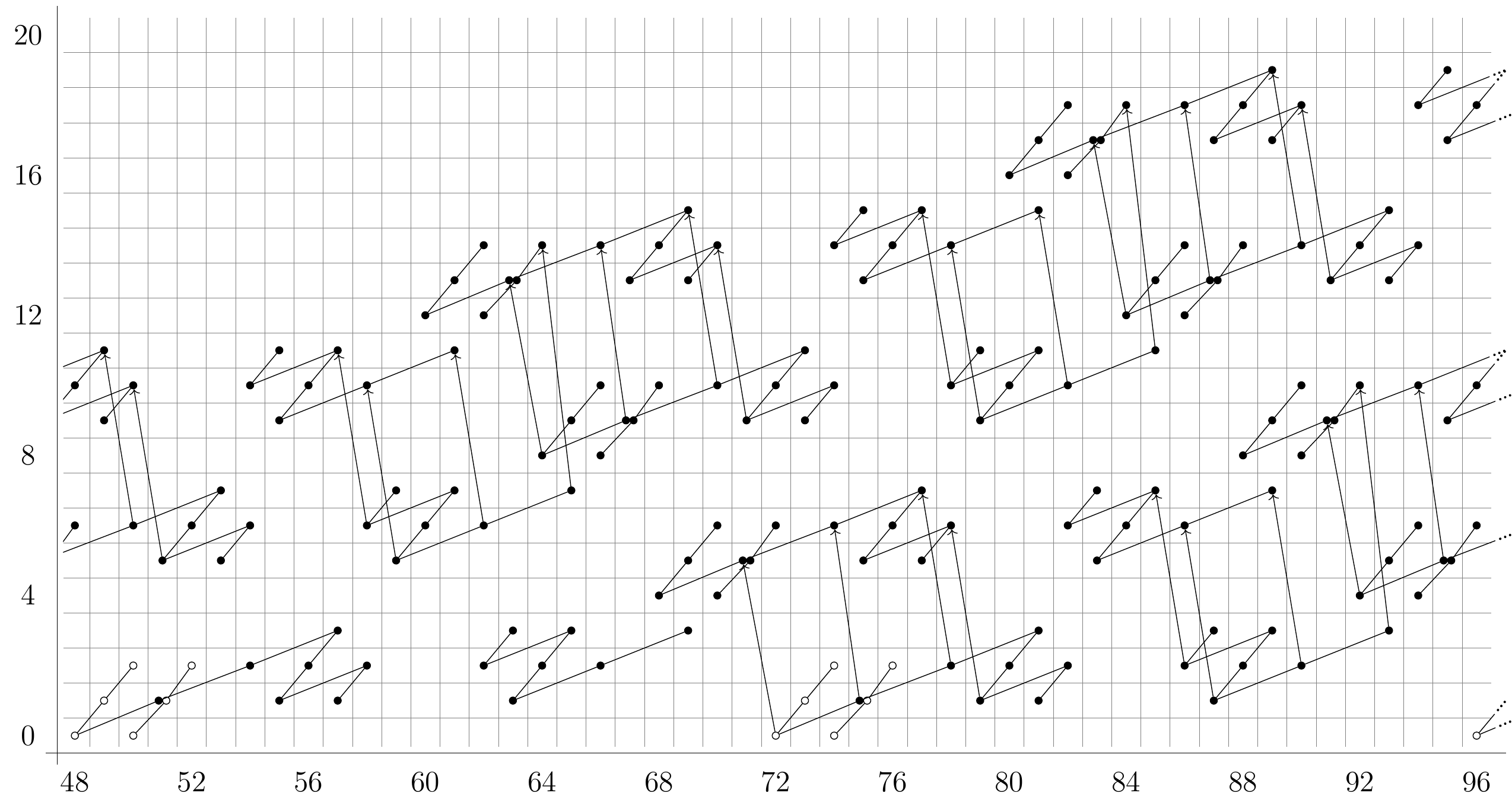}
  
  \vspace{0.1in}
  
    \includegraphics[width=\textwidth]{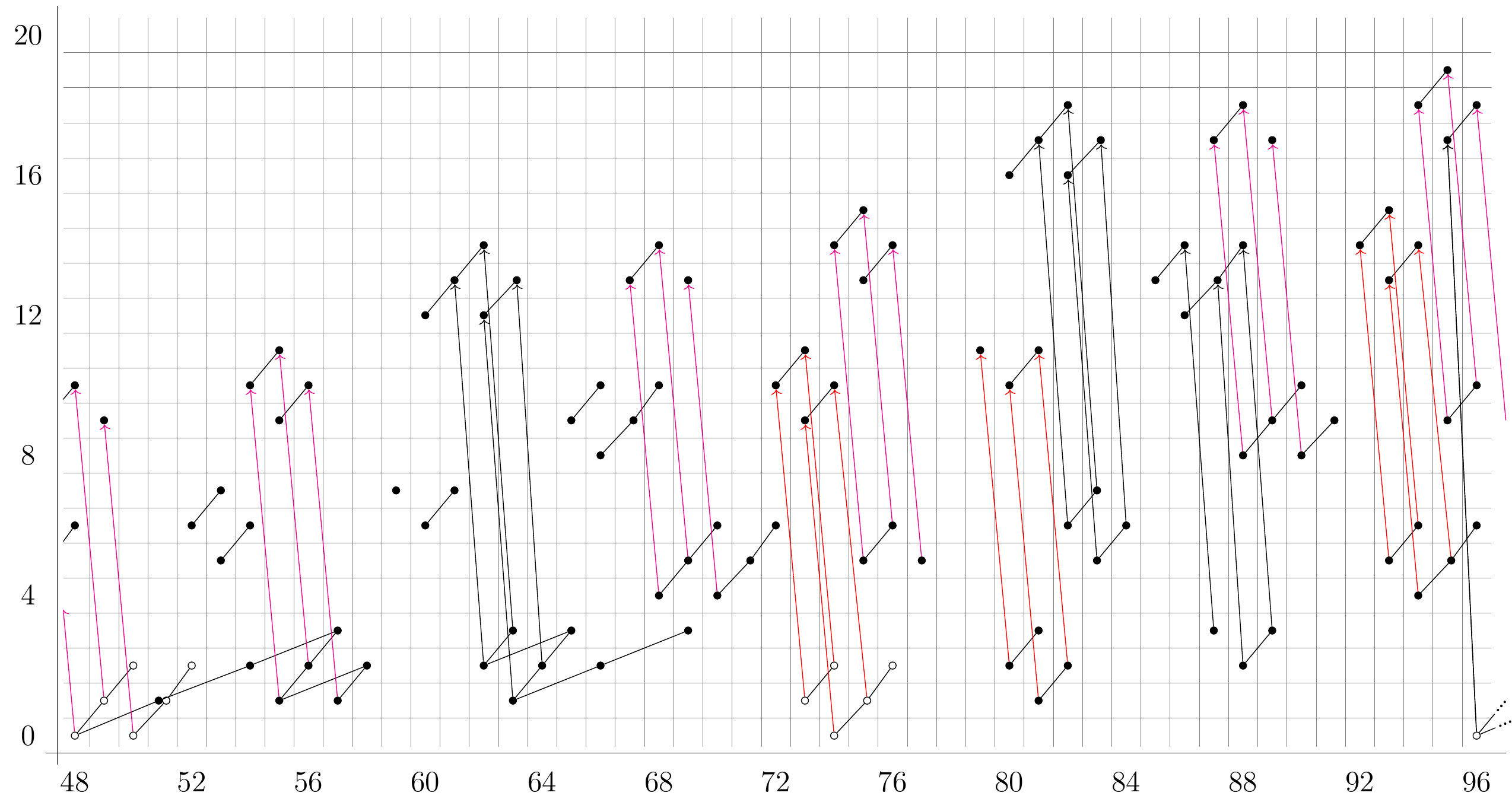}
  \caption{Differentials in stems $48$ to $96$}
  \label{figured5d7}
 \end{figure}

\subsection{Higher differentials}
Since all the classes in filtrations 4 and above are in the ideal generated by $\kappabar$, the differentials that have sources in filtrations 0-3 generate the other differentials with respect to the module structure over the elliptic spectral sequence for $tmf$ (denoted $E_r^{*,*}(S^0)$). We focus on these differentials in the narrative. See Figures~\ref{V0-0-50}, \ref{figured5d7}, \ref{figd9andthensome} and \ref{figd9andthensomemore}.
\begin{lem}
The $d_7$-differentials are $\Delta^4$-linear and determined by
\begin{align*}
d_7(\Delta \nu^2 y)&=\kappabar^2 \eta^2 v_1 , \\
d_7(\Delta^{3} \nu^2 y)&=\Delta^{2}\kappabar^2 \eta^2 v_1  \ ,
\end{align*}
 and the module structure over the elliptic spectral sequence for $tmf$.
\end{lem}
\begin{proof}
First, note that $d_7(\Delta^4)=\eta^2 \kappabar = 4\nu \kappabar$ in the spectral sequence for $tmf$. Therefore, for any $a\in E_{7}(V(0))$
\[d_7(\Delta^4a) = 4\nu \kappabar a+ \Delta^4 d_7(a).\]
Since $4E_{7}(V(0))=0$, we get $\Delta^4$-linearity.

We give a proof for the differential $d_7(\Delta \nu^2 y)=\kappabar^2 \eta^2 v_1 $. The proof for the other differential is similar.
In the spectral sequence for $tmf$, we have
\[d_7( \Delta 4 \kappabar )= \eta^3 \kappabar^2.\]
But, for $\delta_2 \colon E_2^{s,t}(V(0)) \to E_2^{s+1,t}(S^0)$ the connecting homomorphism, we have
\[ \delta_2(\Delta  \nu^2 y)  = \Delta 4 \kappabar\] and 
\[\delta_2(\kappabar^2 \eta^2 v_1) =  \kappabar^2 \eta^3.\] The differential when $i=0$ then follows from \Cref{geobound}.

Making use of the module structure over the spectral sequence for $tmf$, the only other possible $d_7$-differential for degree reasons is on $\Delta^2 \nu^2 y$. But this class is in fact a $d_7$-cycle since $\Delta^2y$ is a $d_7$-cycle by sparseness. 
\end{proof}

\begin{lem}\label{lem:d9-first}
Using the module structure over the elliptic spectral sequence for $tmf$, the $d_9$-differentials are determined by the following differentials with $i=0,1$:
\begin{enumerate}
\item
$d_9(\Delta^{2+4i})=\Delta^{4i}\kappabar^2 \aseven$,  
\item
$d_9(\Delta^{2+4i} x) =
\Delta^{4i}\kappabar^2\kappa$,
\item $d_9(\Delta^{3+4i}\eta)=\Delta^{1+4i}\kappabar^2\epsilon$
\item $d_9(\Delta^{3+4i}\epsilon)=\Delta^{1+4i}\kappa \kappabar^2\eta$
\item
$d_9(\Delta^{2+4i} v_1)=
\Delta^{4i}\kappabar^2  v_1\aseven
$
\item
$d_9(\Delta^{2+4i} v_1\aseven)=\Delta^{4i}\kappabar^2 \eta \afifteen$
\item
$d_9(\Delta^{3+4i}v_1)=\Delta^{1+4i}\kappabar^2  v_1\aseven$
\item 
$d_9(\Delta^{3+4i}v_1\aseven)=\Delta^{1+4i} \kappabar^2 \eta \afifteen$
\end{enumerate}
\end{lem}
\begin{proof}
We prove the claim for $i=0$. To prove $i=1$, one uses exactly the same arguments in later stems.

In order to show (1), note that $\Delta^2$ cannot support any $d_r$ for $r<9$ by sparseness. Then we have the differential from the elliptic spectral sequence for $tmf$ 
\[
d_9(\Delta^2\eta)=\kappabar^2\epsilon
\]
and  this differential becomes $\eta$ divisible in the spectral sequence for $tmf\wedge V(0)$. For (2), we use the same argument with the differential $d_9(\Delta^{2} \epsilon) =
\Delta \kappabar^2\kappa\eta$ from the elliptic spectral sequence for $tmf$.

The differentials (3) and (4) are the images of the same differentials in the  elliptic spectral sequence for $tmf$.
The differentials (5)--(8) are proved using \Cref{geobound}. For example, the differential $d_9(\Delta^2\eta)=\kappabar^2\epsilon$ and the facts that $\delta(v_1)=\eta$ and $\delta(v_1 \aseven)=\epsilon$ together imply (5). The others are similar.

It remains to argue that there are no other generating $d_9$-differentials. As noted above, it suffices to determine this on classes in filtration less than four. Combining a comparison with the spectral sequence for $tmf$ and sparseness, we see that the only question is whether or not the classes $\Delta^4 \aseven$ and  $\Delta^4 v_1\aseven$ support non-trivial $d_9$s. However, the possible targets are the sources of $\kappabar$-multiples of the $d_{11}$-differentials (1) and (3) of \Cref{d11} shown below, which settles the question.
\end{proof}

    \begin{figure}
  \includegraphics[width=\textwidth]{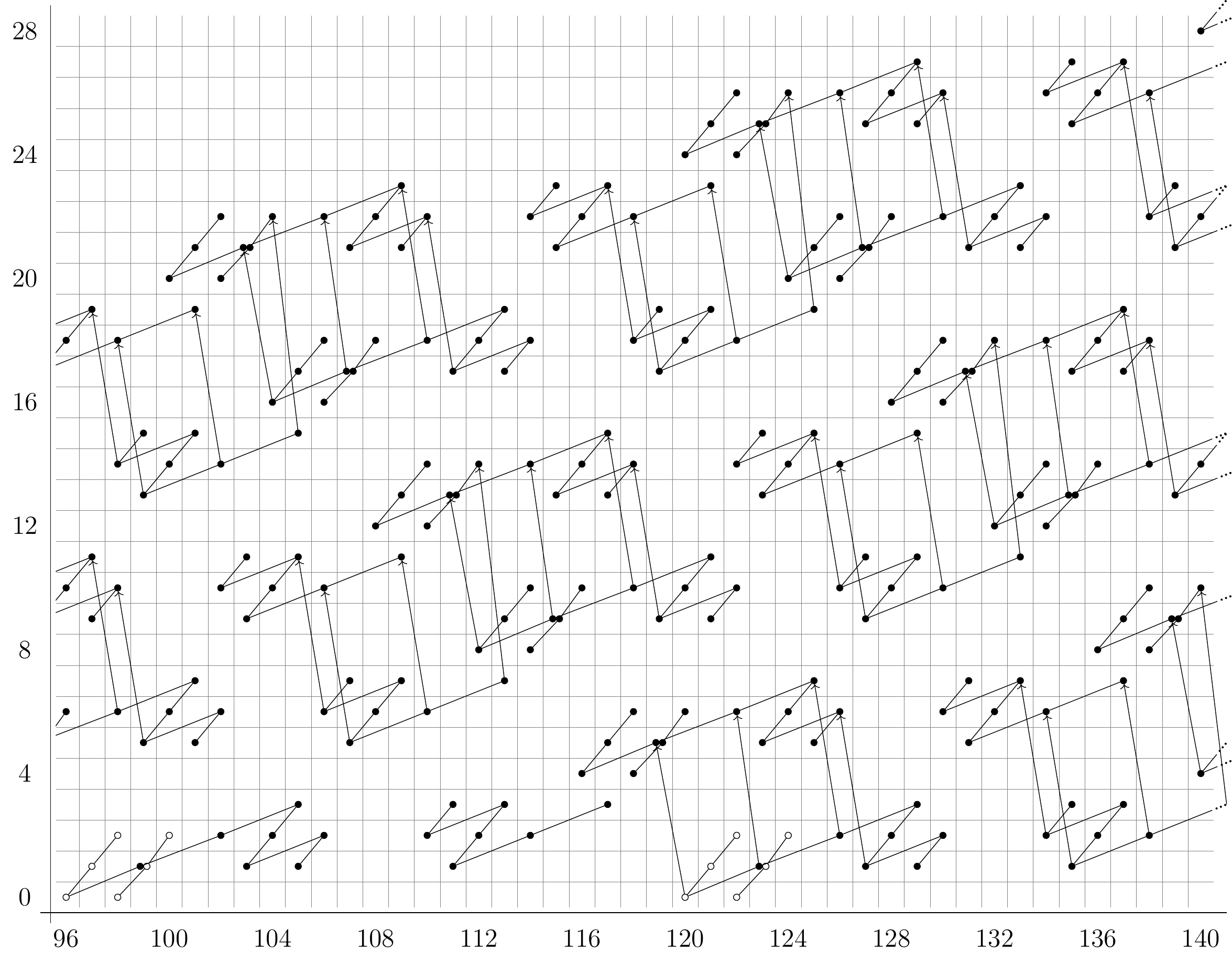}
  
  \vspace{0.1in}
  
    \includegraphics[width=\textwidth]{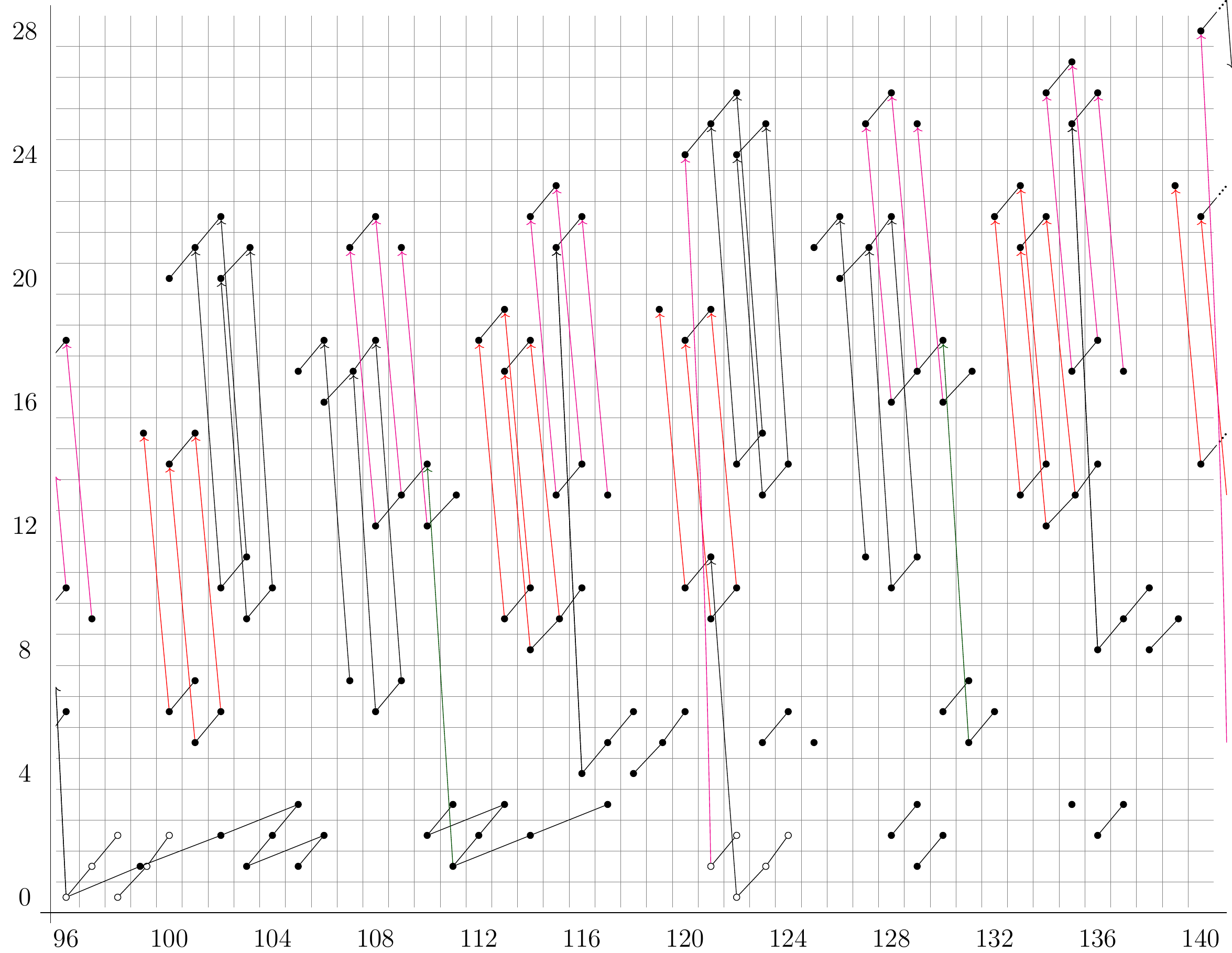}
  
  \caption{Differentials in stems $96$ to $140$}
  \label{figd9andthensome}
 \end{figure}
 
     \begin{figure}
  \includegraphics[width=\textwidth]{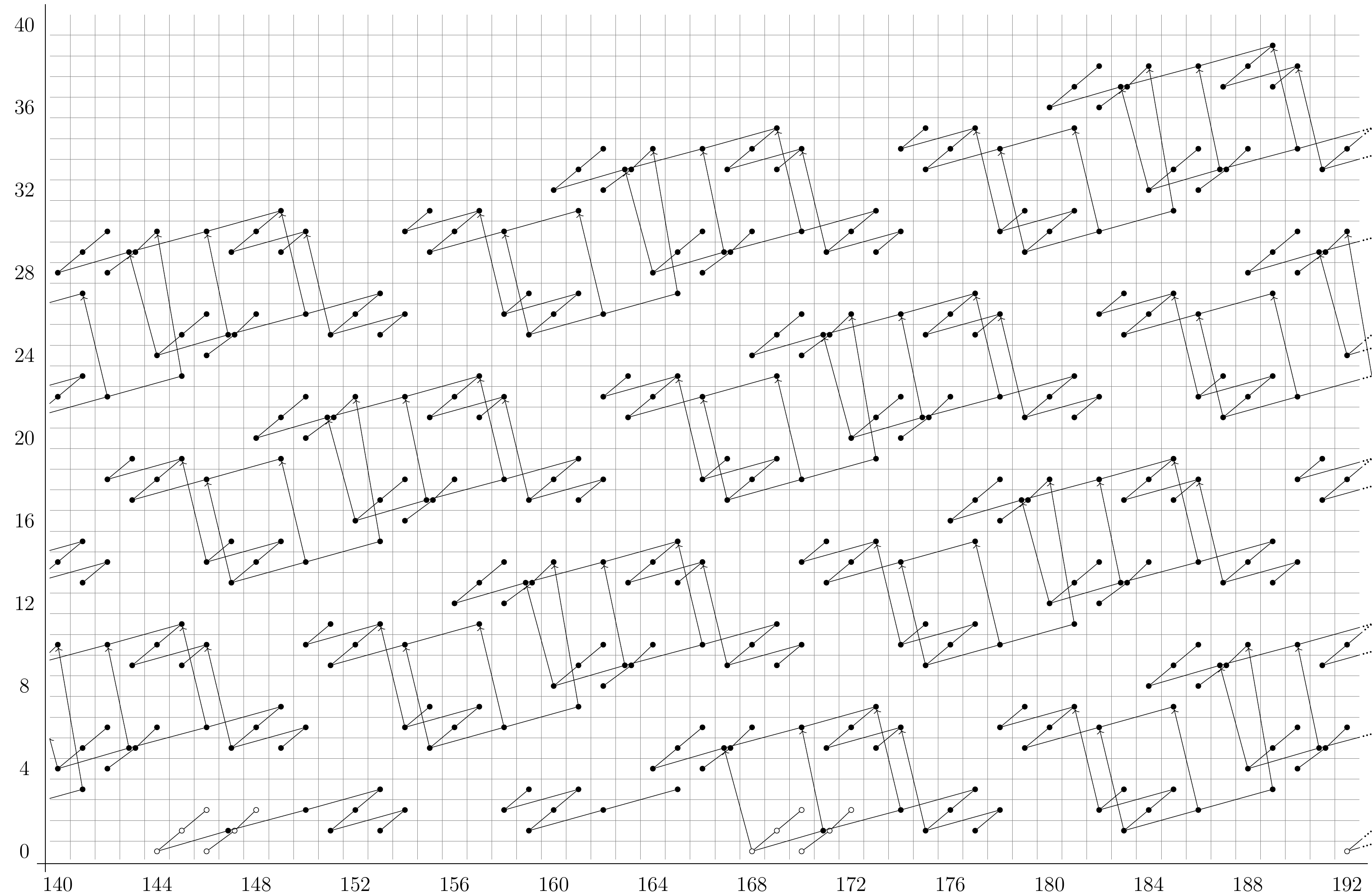}
  
    \vspace{0.3in}
  
    \includegraphics[width=\textwidth]{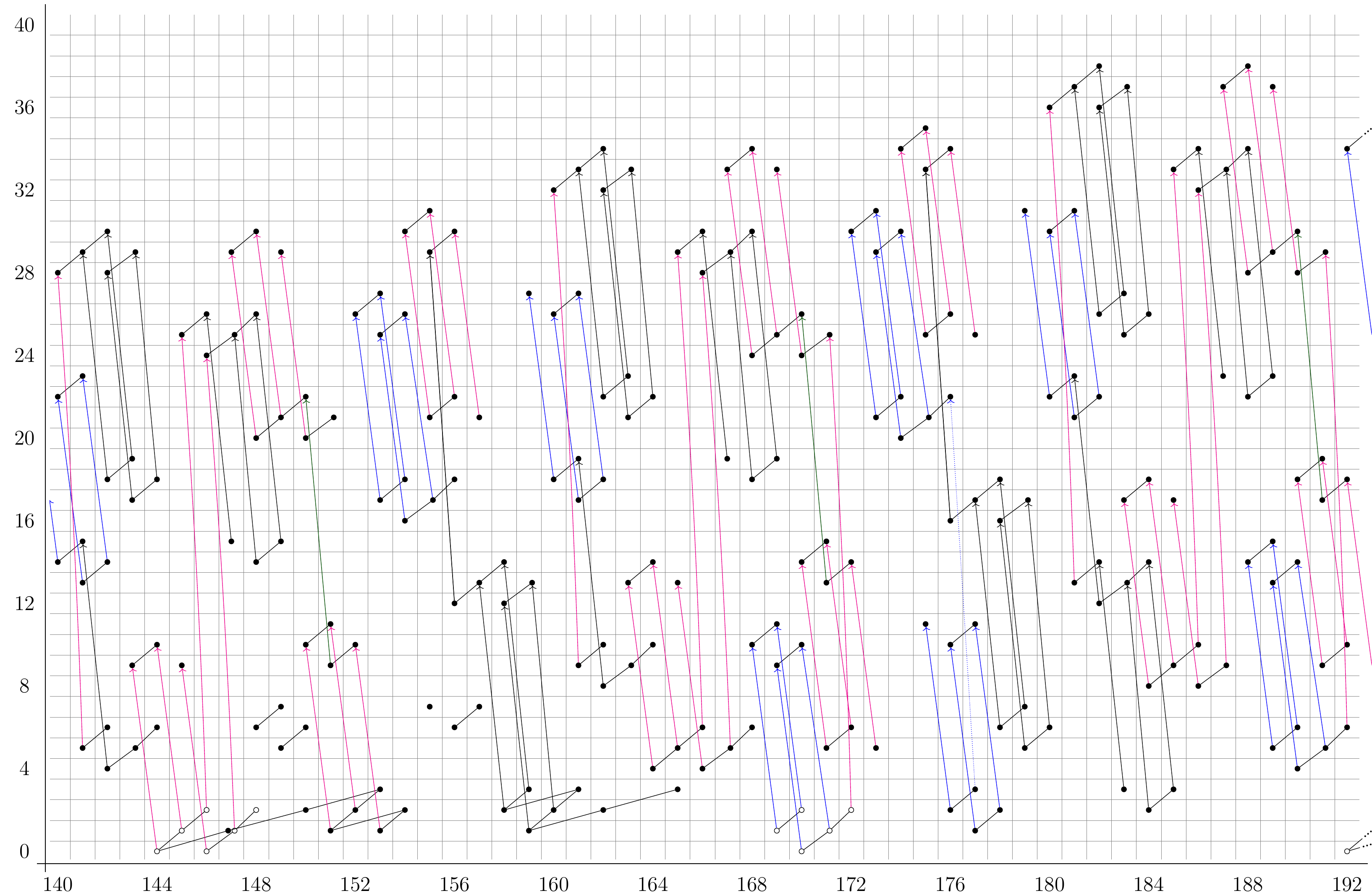}
  \caption{Differentials in stems $140$ to $192$}
  \label{figd9andthensomemore}
 \end{figure}
 
\begin{lem}\label{d11}
Using the module structure over the elliptic spectral sequence for $tmf$, the $d_{11}$-differentials are determined by the following differentials with $i=0,1$:
\begin{enumerate}
\item
$ d_{11}(\Delta^{2+4i}\kappa)=\Delta^{4i}\kappabar^3\eta$
\item 
$ d_{11}(\Delta^{3+4i}\kappa\eta)=\Delta^{1+4i}\kappabar^3\eta^2$ 
\item 
$d_{11}(\Delta^{2+4i} \afifteen)=\Delta^{4i}\kappabar^3 v_1$
\item
$
 d_{11}(\Delta^{3+4i}\kappa v_1)=\Delta^{1+4i} v_1 \kappabar^3 \eta$
\item
 $d_{11}(\Delta^5 v_1)=\Delta^{3}\kappabar^2 \nu^3 $
\end{enumerate}
\end{lem}

\begin{proof}
The differentials (1) and (2) are images of the same differentials in the spectral sequence for $tmf$. The differentials (3) and (4) follow from (1) and (2) respectively using \Cref{geobound}. The differential (5) follows from the fact that $\pi_{121}(tmf \wedge V(0))$ does not contain $v_1^4$-torsion, which can be verified by comparing with $\pi_*tmf$ using the long exact sequence on $\pi_*$.

Sparseness and multiplicative structure guarantees that these are all the generating $d_{11}$-differentials, except for a possible $d_{11}$ on $\Delta^7\eta^2v_1$. However, the possible target is the source of the $\kappabar$-multiple of the $d_{13}$ below. 
\end{proof}

\begin{lem}
The $d_{13}$-differentials are determined by
\[d_{13}(\Delta^4 \afifteen)=\Delta^2 \kappabar^3\eta^2.\]
There are no $d_{15}$-differentials and the $d_{17}$-differentials are determined by
\[d_{17}(\Delta^4)=\kappabar^4 \afifteen.\]
The $d_{19}$-differentials are determined by
\[d_{19}(\Delta^7 \nu^3)=\kappabar^5 \Delta^3 v_1 \eta^2.\]
\end{lem}
\begin{proof}
The first and second differentials follow from the facts that 
\begin{align*}
\pi_{110}(tmf\wedge V(0))=\mathbb{Z}/2 \ \ \text{and} \ \ 
\pi_{95}(tmf\wedge V(0))=0\end{align*}
respectively.
The  $d_{19}$-differential follows from the fact that the there is no $v_1^4$-torsion in $\pi_{177} (tmf\wedge V(0) )$.

There are no $d_{15}$ differentials and no other $d_{17}$ and $d_{19}$ for degree reasons. The only argument needed beyond sparseness and multiplicative structure to show that there are no other $d_{13}$-differentials is as follows. There are possible $d_{13}$s on $\Delta^3\nu^3$ and $\Delta^7\nu^3$. These classes are in the image of the $tmf$ spectral sequence. For $tmf$, $d_{13}(\Delta^3\nu^3)=2\kappabar^4$ and the target maps to zero in the spectral sequence for $tmf \wedge V(0)$ and similarly for $\Delta^7\nu^3$.
\end{proof}

\begin{warn}
The $d_{13}$ differential above is in fact equivalent to the $2$-extension in $\pi_{110}tmf$. For those familiar with names, this corresponds to $2\kappa_4 = \eta_1\kappabar^3$. For a recent detailed treatment of this extension, see \cite[Chapter 9]{BrunerRognesbook}.
\end{warn}

\begin{lem}
There are no $d_{21}$-differentials.  
The $d_{23}$-differentials are determined by:
\begin{enumerate}
\item $d_{23}(\Delta^5 \eta)=\kappabar^6$
\item
$d_{23}(\Delta^6\eta^2)=\kappabar^6 \Delta \eta$
\item
$d_{23}(\Delta^6  \eta v_1)= \kappabar^6 \Delta v_1$
\item
$d_{23}(\Delta^7  \eta^2 v_1)=\kappabar^6 \Delta^2  \eta v_1$ 
\end{enumerate}
\end{lem}
\begin{proof}
The differentials (1) and (2) occur in the elliptic spectral sequence for $tmf$. The differential (3) is the geometric boundary of (2) as in \Cref{geobound}. The last differential is forced by the fact that the $v_1^4$-torsion in  $\pi_{171} (tmf \wedge V(0) )$ is trivial.
There are no $d_{21}$ or other $d_{23}$-differentials for degree reasons.
\end{proof}

The following is now immediate.
\begin{lem}
The spectral sequence collapses at $E_{24}$ with a horizontal vanishing line at $s=22$, i.e., $E_{\infty}^{s,t}(V(0))=0$ for $s\geq 22$.
\end{lem}

\subsection{Exotic extensions}
We list the exotic extensions that do occur. All other possibilities can be ruled out using algebraic structure and duality. We bring to the attention of the reader the precise meaning of exotic extensions given in \Cref{defnexoext}. Note also that all exotic $2$-extensions are deduced from \Cref{lem2exttrick}. We do not discuss $2$-extensions further but include them in our figures.

\begin{figure}
\includegraphics[width=\textwidth]{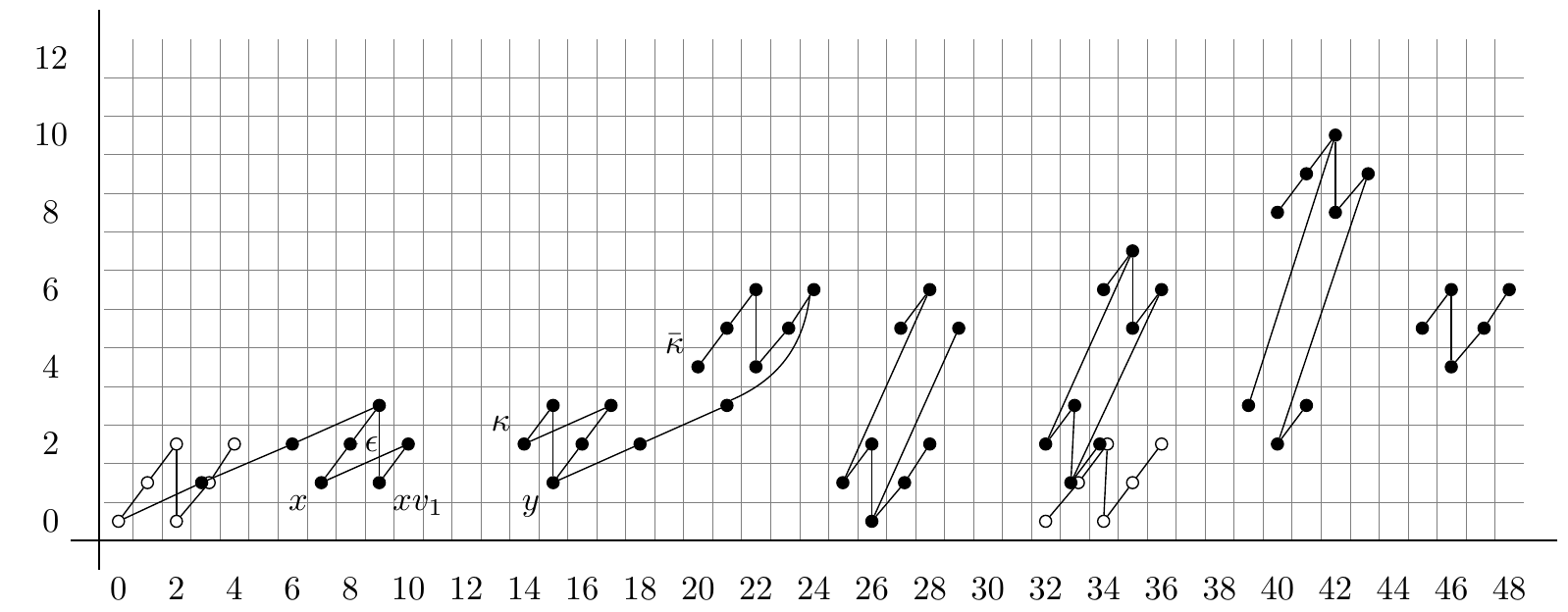}

\vspace{0.1in}

\includegraphics[width=\textwidth]{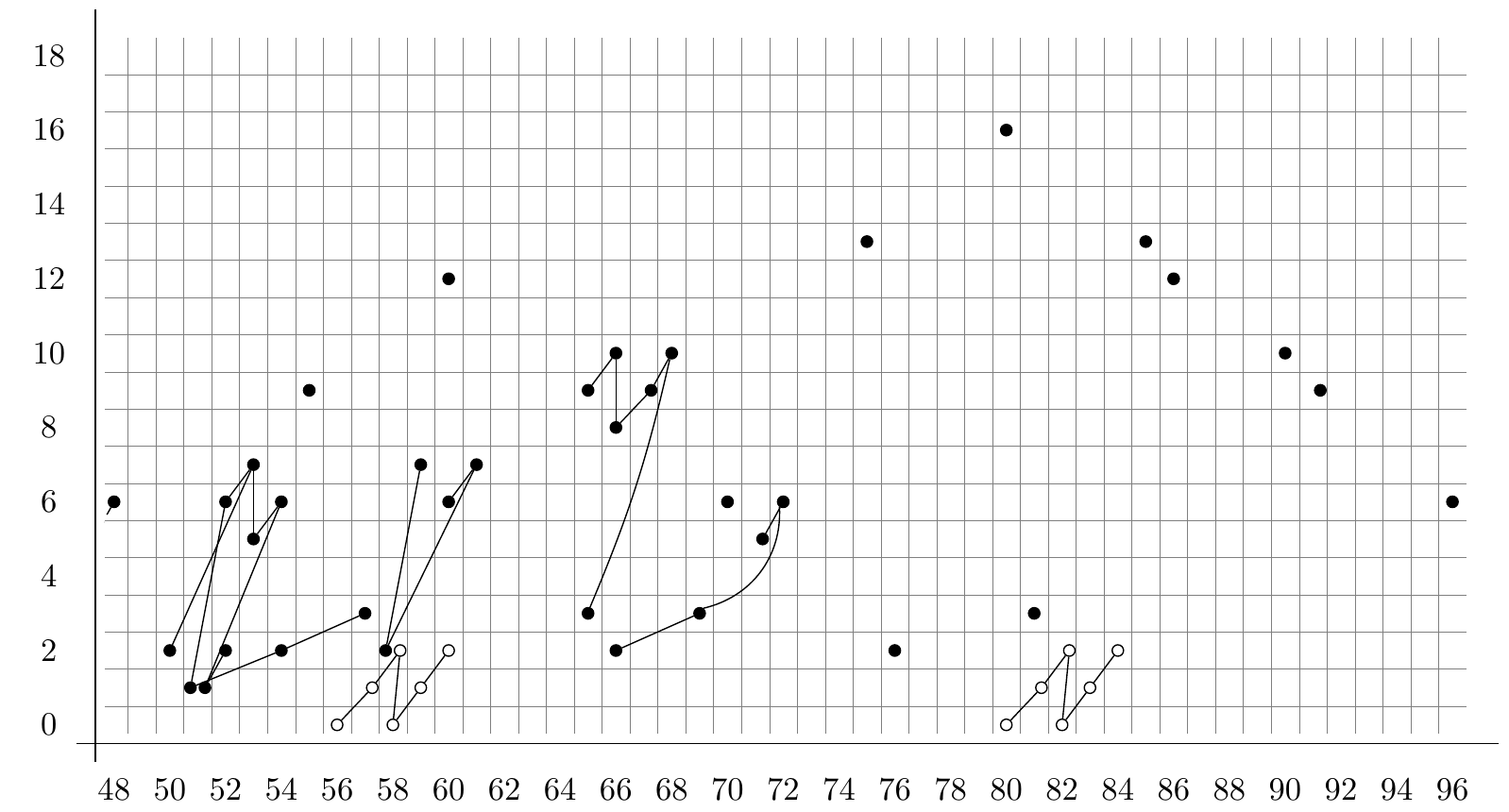}
\caption{Exotic extensions in the elliptic spectral sequence for $tmf\wedge V(0)$ in stems $0$ to $96$. This records $tmf_*V(0) \cong \widetilde{tmf}_{*+1}\R P^{2}$.}
\label{V0-48-96-ext}\label{V0-0-48-ext}
\end{figure}

\begin{lem}[\Cref{V0-0-48-ext}]
In stems $0$ to $45$, there are exotic extensions:
\begin{enumerate}
\item
 $[\Delta\eta]\nu=
 \kappabar\epsilon$
 \item
 $[\Delta\epsilon]\nu=\kappa\kappabar\eta$
 \item
 $[\Delta \kappa\eta]\nu=\kappabar^2\eta^2$
 \item
  $[\Delta v_1]\nu=
\kappabar  v_1 \aseven$
   \item
  $[\Delta v_1 \aseven ]\nu=\kappa\kappabar v_1$
  \item
 $[\Delta \kappa v_1 ]\nu=\kappabar^2 \eta v_1 $
 \item $[y \nu^2]\nu=\kappabar v_1 \eta^2$ 
 \end{enumerate}
\end{lem}

\begin{proof}
The first three extensions are between elements from $\pi_*tmf$, see \cite{tbauer}. 
The next three are forced by the fact that the connecting homomorphism in the long exact sequence on homotopy groups is a map of $\pi_*S^0$-modules, the geometric boundary theorem, and the fact that under the map
\[\delta \colon E_2^{s,t}( V(0) ) \to E_2^{s+1,t}(S^0)\]
we have $\delta(v_1)=\eta$ (and so $\delta(\aseven v_1) =\epsilon$, $\delta(\kappa v_1)=\eta \kappa$, etc.).

The last extension follows from duality and the fact that there is a $\nu$ multiplication between stems $147$ and $150$ (already present on the $E_2$-page).
\end{proof}

\begin{lem}[\Cref{V0-48-96-ext}]
In stems $46$ to $96$, there are exotic extensions:
\begin{enumerate}
\item \label{uuu1}
$[\Delta^2\eta^2]\nu=\Delta\kappabar\nu^3$
\item\label{uuu2}
 $[\Delta^2\nu]\eta=\Delta\kappabar\epsilon$
 \item \label{uuu3}
 $[\Delta^2 v_1\eta]\nu =\Delta\kappabar \aseven \nu$
  \item  \label{uuu4}
 $[\Delta^2 \aseven \nu]\eta =\Delta \kappabar \kappa \eta$
  \item\label{uuu5}
 $[\Delta^2 \aseven \nu] \nu = \Delta \kappabar \kappa \eta v_1$
  \item \label{uuu6}
 $[\Delta^2 \kappa \nu] \nu =\Delta \kappabar^2 \eta^2 v_1 $
 \item  \label{uuu7} $ [\Delta^2 y \nu^2]\nu=\Delta^2\kappabar v_1 \eta^2$ 
\end{enumerate}
\end{lem}

\begin{proof}
The first two extensions \eqref{uuu1} and \eqref{uuu2} are multiplicative relations that hold in $\pi_* tmf$. 
Extension \eqref{uuu3} follows from \eqref{uuu1} and \Cref{lem:half-the-extensions}. Extension \eqref{uuu4} is dual to the algebraic $\eta$ multiplication from stem $112$ to $113$, and similarly for \eqref{uuu5}. The extension \eqref{uuu6} involves classes in the image of $i_*$ and this extension happens in $tmf_*$. Finally, \eqref{uuu7} is dual to the algebraic $\nu$ multiplication from stem $99$ to $102$.
\end{proof}

\begin{rem}
Looking at the charts in \cite{tbauer}, one might have expected extensions  $[\Delta^2\kappa \nu] \eta =\Delta \kappabar^2  \eta^2$
and, by the Geometric Boundary Theorem, $[\Delta^2 \afifteen \nu] \eta =\Delta \kappabar^2 \eta v_1 $. However, these are not exotic extensions according to \Cref{defnexoext}.

We also note that $ [\Delta^2 c_4]\nu \neq [\Delta \kappabar\kappa \eta ]$ and $ [\Delta^3c_4 v_1]\nu \neq [\Delta\kappabar^3 \eta  ] $. The first comes from the fact in $\pi_*tmf$, there is no such extension. (This can be seen, for example, from the Adams Spectral Sequence of $tmf$.) The second follows from the fact that the target has a non-trivial $\kappabar$-multiple and $\kappabar\nu=0$.
 \end{rem}

\begin{figure}
\includegraphics[width=\textwidth]{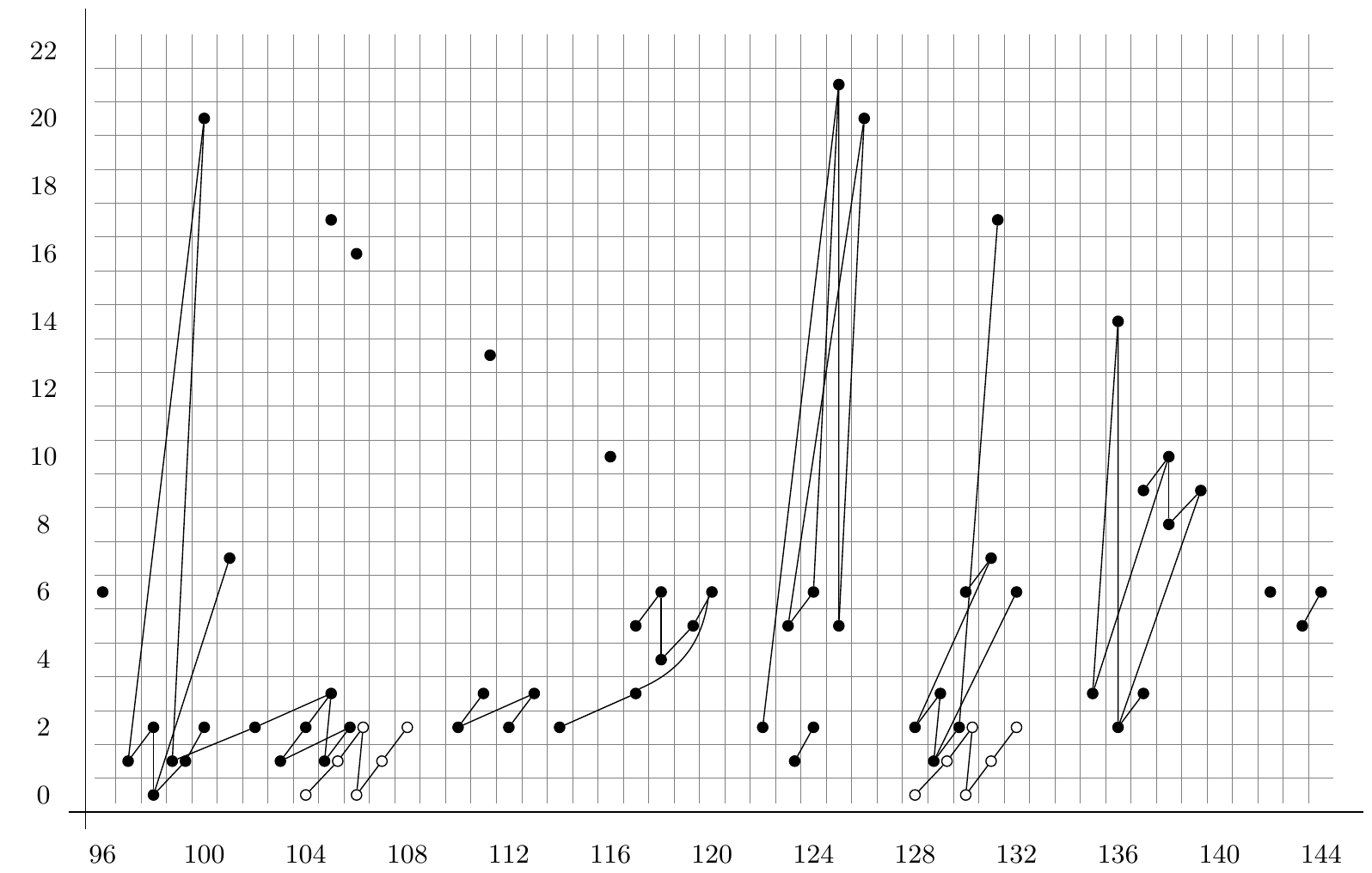}

\vspace{0.1in}

\includegraphics[width=\textwidth]{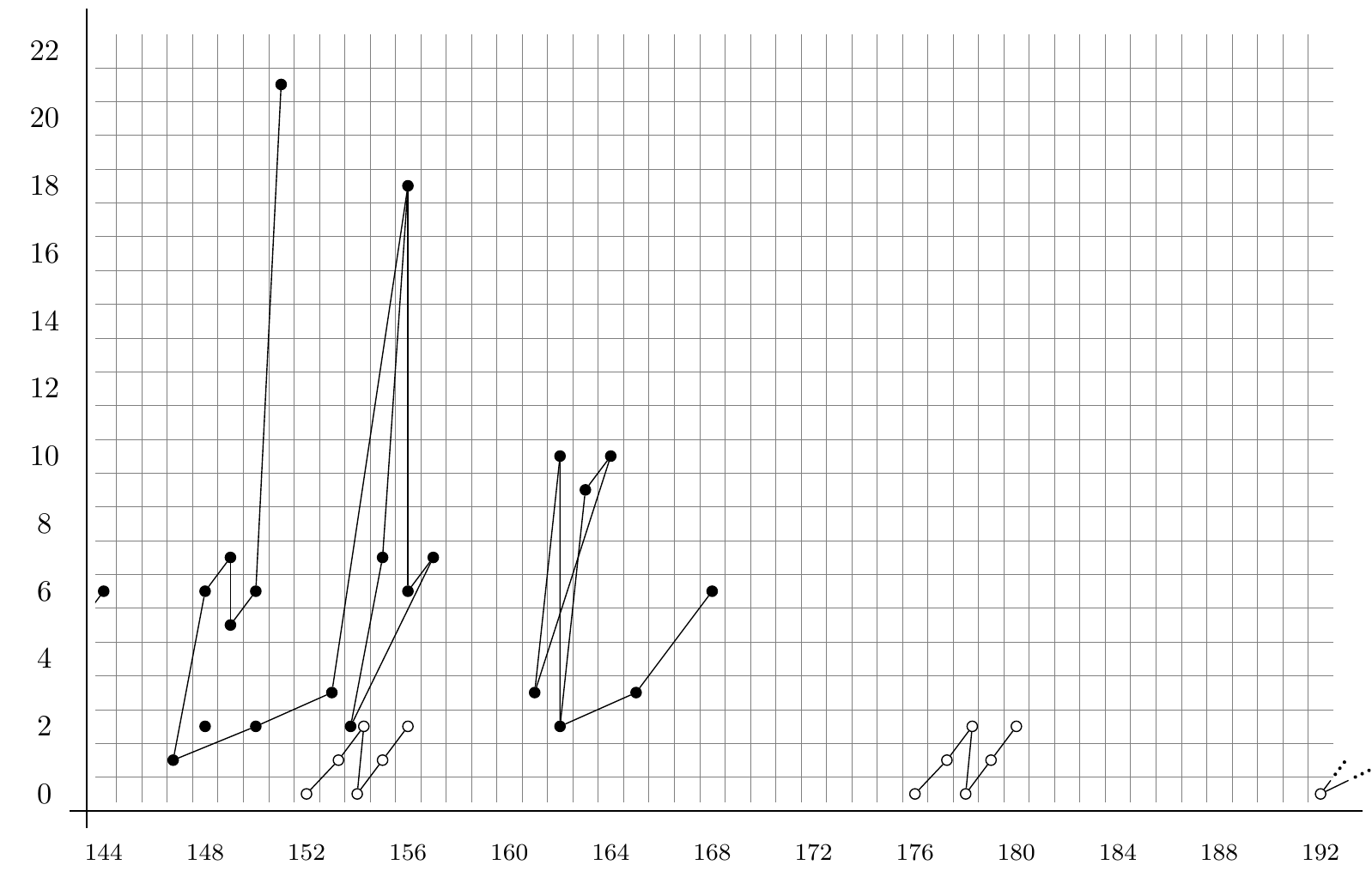}
\caption{Exotic extensions in the elliptic spectral sequence of $tmf\wedge V(0)$ in stems 96 to 192, recording $tmf_*V(0) \cong \widetilde{tmf}_{*+1}\R P^{2}$.}
\label{V0-144-192-ext}\label{V0-96-144-ext}
\end{figure}

\begin{lem}[\Cref{V0-96-144-ext}]
In stems $97$ to $144$, there are exotic extensions:
\begin{enumerate}
\item \label{uuuuu1}
$[\Delta^4\eta]\nu=\kappabar^5$
\item \label{uuuuu2}
$[\Delta^4\nu]\eta=\kappabar^5$
\item \label{uuuuu3}
$[\Delta^4\kappabar\epsilon]\eta=\Delta\kappabar^5\eta$
\item  \label{uuuuu4}
 $[\Delta^5\eta^2]\nu=\Delta \kappabar^5\eta$
\item \label{uuuuu5}
$[\Delta^5\epsilon]\nu=\Delta^4\kappa\kappabar\eta$
\item   \label{uuuuu6}
$[\Delta^5\kappa\eta]\nu=\Delta^4\kappabar^2\eta^2$
\item \label{uuuuu3p}
  $[\Delta^4\kappabar xv_1]\eta=\Delta\kappabar^5 v_1$ (from $(125,5)$ to $(126,20)$)
  \item  \label{uuuuu5p}
 $[\Delta^5 x v_1]\nu=\Delta^4 \kappa\kappabar v_1 $ (from $(129,1)$ to $(132,6)$)
   \item   \label{uuuuu6p}
   $[\Delta^5\kappa v_1]\nu=\Delta^4\kappabar^2 \eta v_1$  (from $(136,2)$ to $(139,9)$)
 \item \label{item:A}
  $[\Delta^5\epsilon v_1]\eta =\Delta^2 \kappabar^4 v_1 \eta$  (from $(130,2)$ to $(131,17)$)
   \item \label{item:B}
  $[\Delta^4 v_1]\nu=\Delta^3\kappabar\nu^3$ (from $(98,0)$ to $( 101,7)$)
 \item \label{item:C}
    $[\Delta^5\kappa \eta]\eta=\Delta^3 \kappabar^3 \eta^2 v_1$ (from $(135,3)$ to $(136,14)$)
 \item \label{item:D} $[\Delta^4 y \nu^2]\nu=[\Delta^4 \kappabar v_1 \eta^2]$ (from $(117,3)$ to $(120,6)$)
 \end{enumerate}
\end{lem}

\begin{proof}
Extensions \eqref{uuuuu1}--\eqref{uuuuu6} follow from studying $i_* \colon tmf_* \to tmf_*V(0)$. 
Note that \eqref{uuuuu4} is missing from the \cite{tbauer} charts, but can be obtained from the classical Adams Spectral Sequence for $tmf$. See \cite[Chapter 13]{tmfbook} or \cite[Chapter 9]{BrunerRognesbook}.
Extensions \eqref{uuuuu3p}, \eqref{uuuuu5p}  and \eqref{uuuuu6p} follow from \eqref{uuuuu3},  \eqref{uuuuu5} and  \eqref{uuuuu6}, respectively, using \Cref{lem:half-the-extensions}.

For \eqref{item:B}, note that by  \Cref{lem:half-the-extensions}, $[\Delta^4v_1]$ has geometric boundary $[\Delta^4\eta]$. Since $[\Delta^4\eta]\nu\neq 0$,  $[\Delta^4v_1]\nu\neq 0$ and this extension is the only choice.
For \eqref{item:C},
use \Cref{lem:duality-for-extensions} and the  algebraic $\eta$ multiplication between $\pi_{35} tmf \wedge V(0)$ and $\pi_{36}tmf \wedge V(0)$. A similar argument applies for \eqref{item:D}.
\end{proof}

\begin{rem}
There is no exotic $\nu$-extension on $[\Delta^5c_4]$ since the potential target is not annihilated by $\kappabar$.
\end{rem}

\begin{lem}[\Cref{V0-144-192-ext}]
In stems $145$ to $191$, there are exotic extensions:
\begin{enumerate}
\item
 $ [\Delta^6\nu]\eta=[\Delta^5\kappabar\epsilon]$   (from $(147,1)$ to $(148,6)$)
 \item\label{uuuuuuu2}
 $ [\Delta^6\kappa\nu]\eta=[\Delta^5\kappabar^2\eta^2]$   (from $(161,3)$ to $(162,10)$)
\item \label{uuuuuuu3}
$ [\Delta^5 \kappabar  \kappa \eta ]\eta=[\Delta^3 \kappabar^4 \eta^2v_1 ]$ (from $(155,7)$ to $(156,18)$)

\item \label{uuuuuuu2p} 
 $ [\Delta^6 y \nu] \eta=[\Delta^5 \kappabar^2 v_1 \eta]$   (from $(162,2)$ to $(163,9)$)
 \item \label{item:2F}
$ [  \Delta^5 \kappabar\epsilon v_1]\eta=[\Delta^2 \kappabar^5 \eta v_1 ]$ (from $(150,6)$ to $(151,21)$)

 \item \label{item:2A} $[\Delta^6\nu^3]\nu=\Delta^3 \kappabar^4 v_1 \eta^2$    (from $(153,3)$ to $(156,18)$)

 \item \label{item:2B} $[\Delta^6\epsilon v_1]\eta=\Delta^5  \kappabar \kappa\eta$ (from $(154,2)$ to $(155,7)$)
  \item \label{item:2C}
 $[\Delta^6 \epsilon v_1]\nu= \Delta^5 \kappabar \kappa \nu$  (from $(154,2)$ to $(157,7)$)
\item \label{item:2D}
 $ [\Delta^6\kappa\nu]\nu=\Delta^5\kappabar^2 v_1\eta^2$   (from $(161,3)$ to $(164,10)$)
\item \label{item:2E}
$ [\Delta^6 y \nu^2]\nu=[\Delta^6 \kappabar v_1 \eta^2]$ (from $(165,3)$ to $(168,6)$)
\end{enumerate}
\end{lem}

\begin{proof}
The first two extensions occur in $tmf_*$. The third is also an extension in $tmf_*$, namely $ [\Delta^5 \kappabar  \kappa \eta ]\eta = [\Delta^42\kappabar^3] $, but the image of the class $[\Delta^42\kappabar^3] $ is detected by $[\Delta^3 \kappabar^4 \eta^2v_1 ]$ in $tmf_*V(0)$.
The extension \eqref{uuuuuuu2p}  follow from (\ref{uuuuuuu2}) and \Cref{lem:half-the-extensions}. This result also implies  \eqref{item:2F} from the extensions $[\Delta^5 \kappabar \nu^3]\eta = [\Delta^2 \kappabar^5 \eta^2]$ in $tmf_*$. All the extensions \eqref{item:2A}--\eqref{item:2E} follow from \Cref{lem:duality} and \Cref{lem:duality-for-extensions} and the data for algebraic multiplications in the range $3 \leq t-s \leq 20$.  
\end{proof}


\section{$tmf_*Y$: The $E_{2}$-page}\label{secYE2}
Let $\Ceta$ be the cofiber of the Hopf map $\eta$, so that there is an exact triangle
\begin{equation}\label{eqcofeta}
S^1 \xrightarrow{\eta}S^0 \to C_{\eta} \to S^2.\end{equation} 
We define
$Y \simeq V(0) \wedge \Ceta$
and study its elliptic spectral sequence. Recall that if $F$ is a finite spectrum, then we abbreviate 
\[\mathcal{F}_*(F) :=\pi_*(tmf\wedge X(4)\wedge F) .\]

We first describe $\mathcal{F}_*(\Ceta)$. 
Since 
$\pi_*(tmf\wedge X(4))\cong A$
is concentrated in even degrees, the cofiber sequence \eqref{eqcofeta}
induces a short exact sequence on $tmf\wedge X(4)$-homology
\[0\rightarrow A \rightarrow \mathcal{F}_*( \Ceta) \rightarrow \Sigma^2 A\rightarrow 0.\]
This splits as a sequence of $A$-modules so that
\[\mathcal{F}_*( \Ceta)\cong A\oplus \Sigma^2 A.\]
Multiplication by $2$ on $\Ceta$ induces multiplication by $2$ on $tmf\wedge X(4)$-homology, which is injective because $\mathcal{F}_*( \Ceta)$ is torsion-free. Thus the cofiber sequence 
\begin{equation*}
\Ceta\xrightarrow{2}\Ceta\rightarrow Y
\end{equation*}
 induces a short exact sequence in $tmf\wedge X(4)$-homology
\[0\rightarrow \mathcal{F}_*( \Ceta)\rightarrow \mathcal{F}_*( \Ceta)\rightarrow \mathcal{F}_*(Y)\rightarrow 0,\]
and it follows that 
\begin{equation}\label{Mod-F(Y)}
\mathcal{F}_*( Y)\cong A/(2)\oplus \Sigma^2 A/(2)
\end{equation} as an $A/(2)$-module.

Likewise, since $\mathcal{F}_*( V(0))$ is concentrated in even degrees, the induced map on $tmf\wedge X(4)$-homology of the cofiber sequence 
\begin{equation*}
\Sigma V(0)\xrightarrow{\eta} V(0)\rightarrow Y
\end{equation*}
is trivial. It follows that there is a short exact sequence of $\Lambda$-comodules 
\[
0\rightarrow A/(2) \rightarrow \mathcal{F}_*( Y)\rightarrow \Sigma^2 A/(2)\rightarrow 0.
\]
This short exact sequence of $A$-modules splits because of \eqref{Mod-F(Y)}.
Tensoring it with $A'$ over $A$, we obtain a short exact sequence of $\Lambda'$-comodules, which splits as a sequence of $A'$-modules
\begin{equation}
\label{tmfY}
0\rightarrow A'/(2) \rightarrow A'\otimes_{A} \mathcal{F}_*( Y)\rightarrow \Sigma^2 A'/(2)\rightarrow 0.
\end{equation}
As $\mathcal{F}_*(Y)$ is $2$-torsion, \eqref{tmfY} is a short exact sequence of $A'/(2)$-module, and hence splits as such.
Therefore, applying $\Ext^{*,*}_{\Lambda'}(A', -)$ to \eqref{tmfY}, we get a long exact sequence of $\Ext^{*,*}_{\Lambda'}(A', A'/(2))$-modules. See, for example, \cite[p.110, (3.3)]{Brown}. Its connecting homomorphism 
\begin{equation}\label{conn_hom1}
\delta : \Ext^{s,t}_{\Lambda'}(A', A'/(2))\rightarrow \mathrm{Ext}^{s+1,t+2}_{\Lambda'}(A', A'/(2))
\end{equation}
is given by multiplication with $\eta\in \Ext^{1,2}_{\Lambda'}(A', A'/(2))$. Here, as is often the case, we denote by $\eta$ the class in $\Ext$ which detects the same-named homotopy class.

We present the effect of the connecting homomorphism separately for the $v_1$-power torsion and for the $v_1$-free classes of $E_2(V(0))$ in \Cref{d-one-Y}  and \Cref{d-one-Y-bo}, respectively.

\begin{figure}[h]
 \includegraphics[width=\textwidth]{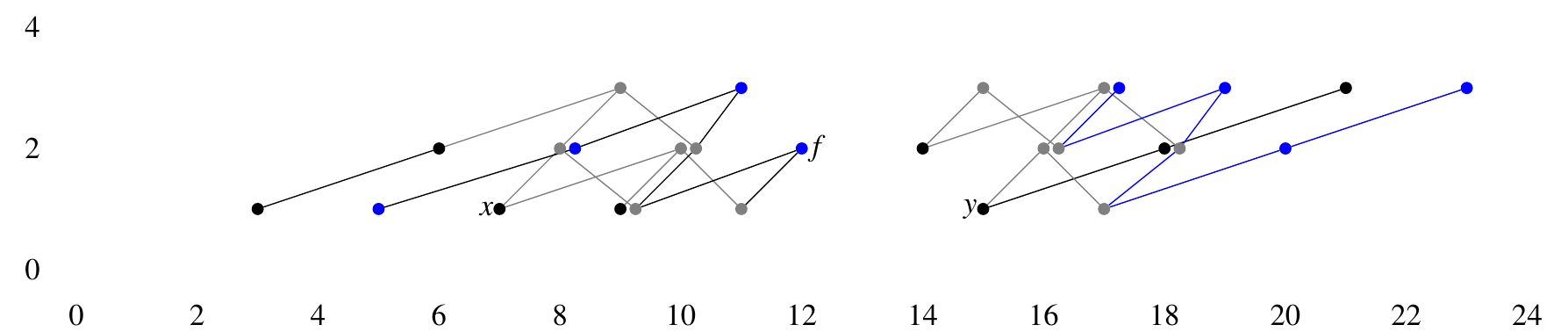}
 \caption{The connecting homomorphism \eqref{conn_hom1} for the $v_1$-power torsion classes}
 \label{d-one-Y}
\end{figure}

\begin{figure}[h]
 \includegraphics[width=\textwidth]{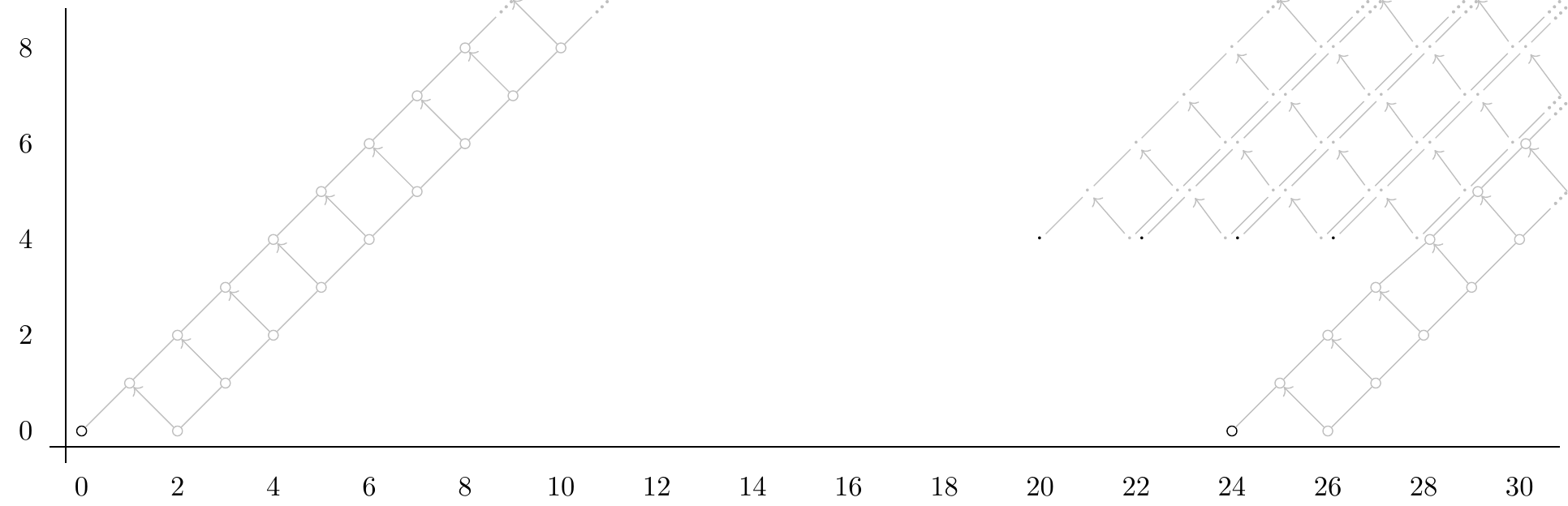}
 \caption{The connecting homomorphism \eqref{conn_hom1} for the $v_1$-free classes}
 \label{d-one-Y-bo}
\end{figure}
In  \Cref{d-one-Y-bo} a $\circ$ denotes a copy of $\FF_2[v_1]$, and a line of slope 1 denotes, as usual, multiplication by $\eta$. Note that we have $\kappabar v_1^4=\Delta \eta^4$, hence $\kappabar v_1^4=0$ in $E_2(Y)$, while $v_1$ itself is not nilpotent and $\Delta^i$ is not $v_1$ torsion.
For our purposes, we need to determine completely the action of $v_1$ on $E_2(Y)$. The class $v_1\in \Ext_{\Lambda'}^{0}(A', A'/(2))$ is detected by the primitive $a_1\in A'/(2)$ (with respect to the $\Lambda'$-comodule structure).

Since $a_1\in A'/(2)$ is a primitive, multiplication by $a_1\in A'/(2)$ induces the following diagram of $\Lambda'$-comodules 

\begin{align}
\label{diagram}\xymatrix{
& 0\ar[d] &0\ar[d]&0\ar[d] &\\
0 \ar[r] & \Sigma^2A'/(2)\ar[d]^{\times a_1} \ar[r]& \Sigma^2 A'/(2)\otimes_{A/(2)}\mathcal{F}_*( Y)\ar[d]^{\times a_1} \ar[r]& \Sigma^4 A'/(2) \ar[d]^{\times a_1}\ar[r]& 0\\
0 \ar[r] & A'/(2)\ar[d] \ar[r]&  A'/(2)\otimes_{A/(2)}\mathcal{F}_*( Y)\ar[d] \ar[r]& \Sigma^2 A'/(2)\ar[d] \ar[r]& 0\\
0 \ar[r] & A'/(2,a_1)\ar[d] \ar[r]& A'/(2,a_1)\otimes_{A/(2)} \mathcal{F}_*( Y)\ar[d] \ar[r]& \Sigma^2 A'/(2,a_1)\ar[d] \ar[r]& 0 \\
& 0 &0&0 &
}
\end{align}
We let
\[\mathcal{M}:=A'/(2,a_1)\otimes_{A/(2)}\mathcal{F}_*( Y)\]
The middle vertical short exact sequence induces a long exact sequence 
\begin{equation}\label{v1-action} 
\xymatrix@C=1.5pc{ \ldots \ar[r] & E^{s,t}_2(Y) \ar[r]^-{v_1} & E^{s,t+2}_2(Y)\ar[r]  
 &  \Ext^{s,t+2}_{\Lambda'}(A', \mathcal{M}) \ar[r]  
  & E_2^{s+1,t+2}(Y) \ar[r]^-{v_1}  & \ldots}
\end{equation}
so that we can determine the action of $v_1$ on $E_2(Y)$ by computing 
\[\Ext^{*,*}_{\Lambda'}(A', \mathcal{M})\cong \Ext^{*,*}_{\Lambda'/(2,a_1)}(A'/(2,a_1), A'/(2,a_1)\otimes_{A/(2)} \mathcal{F}_*( Y)) .\]
The cohomology ring $\Ext_{\Lambda'/(2,a_1)}^{*,*}(A'/(2,a_1), A'/(2,a_1))$ is computed in \cite[Section 7]{tbauer}. With our notation,
\[\Ext_{\Lambda'/(2,a_1)}^{*,*}(A'/(2,a_1), A'/(2,a_1))\cong \F_2[\eta,\nu,\kappabar,v_2]/(v_2\eta^3 -\nu^3,\eta\nu).\]
The bottom short exact sequence of the above diagram \eqref{diagram} splits as a sequence of $A'/(2,a_1)$-modules. However, it does not split as a one of $\Lambda'/(2,a_1)$-comodules, rather it represents the element $\eta\in \Ext_{\Lambda'/(2,a_1)}^{1,2}(A'/(2,a_1), A'/(2,a_1))$. Therefore, the connecting homomorphism 
\begin{equation}\label{connect_hom} \Ext_{\Lambda'/(2,a_1)}^{s,t}(A'/(2,a_1), A'/(2,a_1))\rightarrow \Ext_{\Lambda'/(2,a_1)}^{s+1,t+2}(A'/(2,a_1), A'/(2,a_1))
\end{equation}
of the induced long exact sequence in $\Ext_{\Lambda'/(2,a_1)}^{*,*}(A'/(2,a_1), - )$ is given by multiplication by $\eta$. We obtain:
\begin{lemma}\label{eta-extens-alg} As a module over the ring $\F_2[\eta, \nu, \kappabar, v_2]/(v_2\eta^3-\nu^3,\eta\nu)$, the cohomology group 
\[\Ext^{*,*}_{\Lambda'/(2,a_1)}(A'/(2,a_1), A'/(2,a_1)\otimes_{A/(2)}\mathcal{F}_*( Y))\] is generated by $a[0,0]\in \Ext^{0,0}$ and $a[5,1]\in \Ext^{1,6}$ with the relations
\begin{align*}
\eta a[0,0] &= 0, & \eta a[5,1] = \nu^2a[0,0].
\end{align*}
\end{lemma}
\begin{proof} By the description of the connecting homomorphism \eqref{connect_hom}, we see that 
\[\Ext^{*,*}_{\Lambda'/(2,a_1)}(A'/(2,a_1), A'/(2,a_1)\otimes_{A/(2)}\mathcal{F}_*( Y))\cong \F_2[\nu, \kappabar, v_2]/(\nu^3)\{a[0,0],a[5,0]\}\]
as an $\F_2[\nu, \kappabar, v_2]/(\nu^3)$-module.
Next, we determine the action of $\eta$. We see easily that $\eta a[0,0]=0$. To calculate $\eta a[5,1]$, we remark that 
\begin{align*}
\nu^2a[0,0] &= \langle \eta,\nu,\eta\rangle  a[0,0] = \eta\langle\nu,\eta, a[0,0]\rangle,\end{align*}
where the first equality comes from the 
Massey product $\nu^2 = \langle \eta,\nu,\eta\rangle$ and the second is a shuffle.
As $\nu^2 a[0,0]\ne 0$, $\langle\nu,\eta, a[0,0]\rangle$ is not trivial and must be equal to $a[5,1]$ by sparseness. Hence,
$\nu^2 a[0,0] = \eta a[5,1]$.
\end{proof}

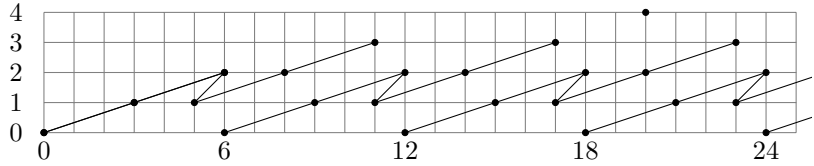
\begin{figure}[h!]
\begin{tikzpicture}[scale=0.4]
\clip(-1.5,-1.5) rectangle (25.5,4.5);
\draw[color=gray] (0,0) grid [step=1] (25,4);
\foreach \n in {0,6,...,24}
{
\def\nn{\n-0}
\node[below] at (\nn,0) {$\n$};
}
\foreach \s in {0,1,...,4}
{\def\ss{\s-0};
\node [left] at (-0.4,\ss,0){$\s$};
}
\draw [fill] ( 0.00, 0.00) circle [radius=0.1];
\draw [fill] (3,1) circle [radius=0.1];
\draw [fill] (6,2) circle [radius=0.1];
\draw [-] (0,0)--(3,1);
\draw [-] (3,1)--(6,2);
\foreach \t in {0}
\foreach \s in {0,6,...,24}
{\def\ss{\s-0};
\draw [fill]  (\s-\t,\t) circle [radius=0.1];
\draw [fill] (\s+3-\t,1+\t) circle [radius=0.1];
\draw [-] (\s-\t,\t)--(\s+3-\t,1+\t);
\draw[fill] (\s+5-\t,1+\t) circle [radius=0.1];
\draw[fill] (\s+8-\t,2+\t) circle [radius=0.1];
\draw[-] (\s+5-\t,1+\t)--(\s+8-\t,2+\t); 
\draw [fill] (\s-\t+6,2+\t) circle [radius=0.1];
\draw [-] (\s+3-\t,1+\t)--(\s+6-\t,2+\t);
\draw[fill] (\s+11-\t,3+\t) circle [radius=0.1];
\draw[-] (\s+8-\t,2+\t)--(\s+11-\t,3+\t); 
\draw [-] (\s+5-\t,1+\t)--(\s-\t+6,2+\t);
}
\draw [fill] ( 20, 4) circle [radius=0.1];
\end{tikzpicture}
\caption{ $\Ext^{s,t}_{\Lambda'/(2,a_1)}(A'/(2,a_1), A'/(2,a_1)\otimes_{A/(2)}\mathcal{F}_*( Y))$ depicted in the coordinate $(t-s,s)$.}
\label{fig:Amod2a1}
\end{figure}

\begin{proposition}[\Cref{figE2Y}]\label{propE2Y}
As a module over $E_2(V(0))$, $E_2(Y)$ is generated by classes
\begin{align*}
a[0,0],   a[5,1],  a[17,3]
\end{align*}
The submodule generated by $a[0,0]$ is isomorphic to $E_2(V(0))/\eta$. There are Massey products
\[a[5,1]=\langle \nu,\eta, a[0,0] \rangle , \ \ a[17,3] =\langle \eta x^2,\eta, a[0,0] \rangle \] 
and these classes are subject to the following relations. On the new classes, we have $v_1$ mulitiplications
\begin{align*}
  v_1 a[5,1]= xa[0,0]  \ \ \ \   v_1a[17,3]= x^2a[5,1], 
   \end{align*}
   $\eta$ and $\nu$ multiplications
   \begin{align*}
 \eta a[5,1] = \nu^2a[0,0], \ \ \  \eta a[17,3]= \nu a[17,3]= y a[17,3]= 0 
 \end{align*}
as well as
 \begin{align*}
\nu^2ya[5,1]= v_1^3 \kappabar a[0,0] \ .
\end{align*}
\end{proposition}
\begin{proof}
Using the description of $E_2(V(0))$, the effect of the connecting homomorphism $\delta$ of \eqref{conn_hom1} is straightforward to compute. The cokernel is simply $E_2V(0)/\eta$ as an $E_2(V(0))$-module. Using the multiplication on $E_2(V(0))$, the kernel is generated by classes $a[5,1]$ and $a[17,3]$ defined as
\begin{align*}
a[5,1] = p_*(\nu) \ \ \ a[17,3] = p_*(\eta x^2),
\end{align*}
where $p_*$ is induced by the map $A'\otimes_{A} \mathcal{F}_*( Y)\rightarrow \Sigma^2 A'/(2)$ of (\ref{tmfY}).

 Inspecting the long exact sequence \eqref{v1-action} and the structure of 
\[\Ext^{*,*}_{\Lambda'/(2,a_1)}(A'/(2,a_1), A'/(2,a_1)\otimes_{A/(2)}\mathcal{F}_*( Y)),\]
we see that $v_1 a[5,1] = xa[0,0]$ (else the latter  $\Ext^{1,8}$-term would be nonzero and contain the image of $xa[0,0]$). That $v_1 a[17,3]=x^2a[5,1]$ follows from the fact that $v_1\eta x^2 = x^2 \nu$ in $E_2(V(0))$ and the definition these classes as images of $\delta$.

By the same argument used in \Cref{eta-extens-alg}, we deduce the $\eta$-multiplication on $a[5,1]$. The relations $\eta a[17,3]= \nu a[17,3]= y a[17,3]= 0$ follows for degree reasons.

It remains to verify that $\nu^2ya[5,1]= v_1^3 \kappabar a[0,0]$.
A juggling of Massey products gives
 \[\afifteen\nu^2 \langle \nu, \eta, a[0,0]\rangle = \langle \afifteen\nu^2, \nu,\eta\rangle a[0,0].\] 
The relation $\nu^2ya[5,1]= v_1^3 \kappabar a[0,0]$  then follows by \Cref{lemkv13} below and the fact that $\eta a[0,0]=0$.
\end{proof}

\begin{figure}[h]
\includegraphics[page=1, width=\textwidth]{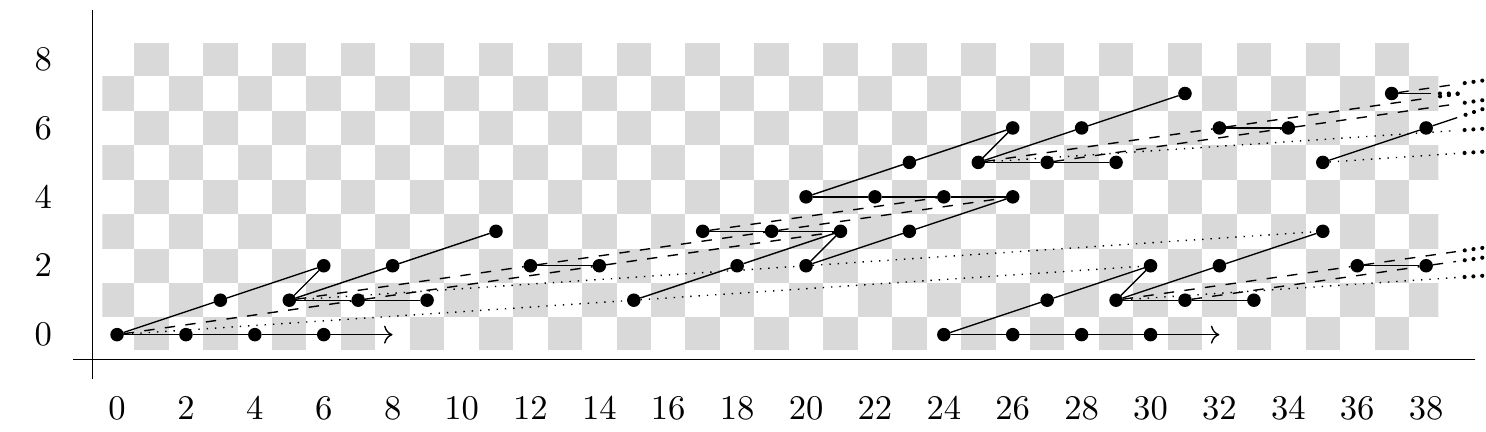}
\caption{$E_2(Y)$ as a module over $E_2(V(0))$. The dashed lines are $x$-multiplications and dotted lines $y$-multiplications. Other structure lines are as in \Cref{tmf-V0-E2}.}
\label{figE2Y}
\end{figure}

\begin{lemma}\label{lemkv13} In $\Ext_{\Lambda'}^{*,*}(A',A'/(2))$, there is the following Massey product
\[\kappabar v_1^3\in \langle \afifteen\nu^2, \nu,\eta\rangle \]
with indeterminacy 
\[ \eta \Ext^{3,28}_{\Lambda'}(A',A'/(2)) + y\nu^2\Ext^{1,6}_{\Lambda'}(A',A'/(2)).\]
\end{lemma}
\begin{proof}
By \cite[Formula 7.9]{tbauer}, 
$\kappabar v_1^2 = \langle \eta, \kappa \eta, \aseven\rangle$
and
\[\kappabar v_1^3 = v_1 \langle \eta, \kappa \eta, \aseven\rangle \subset \langle v_1 \eta, \kappa \eta, \aseven\rangle \subset \langle \eta, v_1 \kappa \eta, \aseven\rangle = \langle \eta, \eta^2 \afifteen, \aseven \rangle\]
and the indeterminacy 
$\eta \Ext^{3,28}_{\Lambda'}(A',A'/(2))$ 
does not contain $\kappabar v_1^3$. 
Here, we used the relation $v_1 \kappa \eta = \eta^2 \afifteen$.
So
$\kappabar v_1^3 = \langle \eta, \eta^2 \afifteen,\aseven \rangle $
and it follows that 
\begin{align*}
\kappabar v_1^4 & = v_1 \langle\eta, \eta^2\afifteen, \aseven \rangle  
= \langle v_1 \eta, \afifteen, \eta^2 \aseven \rangle  
= \langle v_1 \eta, \afifteen, \nu^3\rangle 
 \subset \langle v_1\eta, \afifteen\nu,\nu^2\rangle\end{align*}
and  the indeterminacy $\eta v_1 \Ext^{3,28}$ does not contain $\kappabar v_1^4$. So 
\[\kappabar v_1^4 = \langle v_1\eta, \afifteen\nu,\nu^2\rangle
 = \langle v_1\eta, \nu,\nu^2\afifteen\rangle  = v_1 \langle \eta, \nu,\afifteen\nu^2\rangle .\]
As $v_1$ acts injectively on $\Ext_{\Lambda'}^{4, 30}(A', A'/(2))$, so
$\kappabar v_1^3 = \langle \eta, \nu,\afifteen\nu^2\rangle = \langle \afifteen\nu^2,\nu,\eta\rangle$.
\end{proof}

\begin{rem}
In $E_2^{s,t}(Y)$, there is at most one non-zero element in any bi-degree $(s,t)$ with filtration $s>0$. There is also a unique non-zero element in bi-degree $(0,0)$. So, for $s>0$ or $(s,t)=(0,0)$, we often denote by $a[t-s,s] \in E_r^{s,t}$ the non-zero element, if it exists. Furthermore, when $s=0$ and $t>0$, we let $a[t,0]$ denote the element of $E_2^{0,t}(Y)$ which is divisible by the largest power of $\Delta$. For example, $E_2^{0,52}(Y) \cong \F_2\{v_1^{26}a[0,0],v_1^{14}\Delta a[0,0], v_1^2\Delta^2 a[0,0] \}$ and $a[52,0]=v_1^2\Delta^2 a[0,0]$.
\end{rem}

Although \Cref{propE2Y} gives us a very compact description of $E_2(Y)$, the elliptic spectral sequence of $tmf \wedge Y$ is not a module over the elliptic spectral sequence of $tmf \wedge V(0)$ as the latter is not even a multiplicative spectral sequence. However, the elliptic spectral sequence of $tmf \wedge Y$ is a module over the elliptic spectral sequence of $tmf$. 
In fact, we get even more structure than that from the fact that $Y$ has $v_1$-self maps. As explained in \Cref{v1selfmaps}, we have:
\begin{lem}[$v_1$-linearity]\label{lem:e2-e4}
The differentials in the elliptic spectral sequence for $tmf \wedge Y$ are $v_1$-linear. 
\end{lem}

We state the following ``intermediate'' result for convenience of reference in the computations below. The module structure of the elliptic spectral sequence spectral sequence of $tmf\wedge Y$ over that of $tmf$ is richer than what is stated here but that information can be read off of \Cref{propE2Y}. 
\begin{cor}\label{cormodstructureE2Y}
As a module over
\[\F_2[v_1, \nu, \kappabar, \Delta]/( v_1\nu, \nu^3, v_1^4 \kappabar)\]
$E_2(Y)$ is generated by
\[a[0,0], \ a[5,1],  \ a[12,2], \ a[15,1], \ a[17,3],  \ a[20,2] \]
subject to the relations generated by
\begin{align*}
v_1^3a[5,1]=v_1^2a[12,2]=v_1a[15,1]=\nu a[12,2] =\nu a[17,3]  =0
\end{align*}
and
\begin{align*}
\nu^2a[15,1]=v_1^2a[17,3] ,  \ \  \ \ \nu^2 a[20,2]=v_1^3 \kappabar a[0,0] \ .
\end{align*}
Furthermore, the differentials are $\F_2[v_1, \nu, \kappabar, \Delta^8]/( v_1\nu, \nu^3, v_1^4 \kappabar)$-linear.
\end{cor}
\begin{proof}
This follows from the results of this section and the fact that $\Delta^8$ is a permanent cycle in the elliptic spectral sequence spectral sequence of $tmf$. 
\end{proof}

\section{$tmf_*Y$: The differentials and extensions}\label{secYdiffext}

Our approach to computing the differentials of the elliptic spectral sequence for $\pi_* (tmf\wedge Y)$ is based largely on the analysis of the action of $\kappabar$. Since $\kappabar$ is a permanent cycle in the elliptic spectral sequence for $tmf$, $\kappabar$ acts on the spectral sequence for $tmf\wedge Y$ and differentials are linear with respect to this action. 

\begin{Lemma}\label{StructureSS} The $E_r$-term of the elliptic spectral sequence for $Y$ has the following properties:
\begin{enumerate}
\item All classes in filtration greater than $(r-1)$ are $\kappabar$-free.
\item All classes in filtration greater than or equal to 4 are divisible by $\kappabar$.
\end{enumerate}
\end{Lemma}
\begin{proof} We prove these two properties by induction on $r\geq 2$. For $r=2$, this follows from \Cref{propE2Y}. Suppose now that $r>2$. Let $a$ be a $d_{r-1}$-cycle and $[a] \in E_r^{s,t}$ the corresponding class. Suppose that $a$ lives in filtration $s$ with $s > (r-1)$. We have that $\kappabar[a] = 0$ if and only if there exists $b\in E_{r-1}$ such that $d_{r-1}(b)  = \kappabar a$. Then, $b$ must live in filtration $(4+s) - (r-1) > 4$. By the second property, $b$ is divisible by $\kappabar$, i.e., there exists $c\in E_{r-1}$ such that $\kappabar c = b$. As a consequence of the $\kappabar$-linearity, $\kappabar d_{r-1}(c) = d_{r-1}(b) = \kappabar a$, and so $\kappabar(d_{r-1}(c) -a) = 0$. Since $(d_{r-1}(c) -a)\in E_{r-1}$ lives in filtration $s$ greater than $r-2$, it is $\kappabar$-free by the second property. It follows that $d_{r-1}(c) = a$, and so $[a]=0$. Therefore, the $E_r$-term has the first property. 

For the second property, suppose that $a$ lives in filtration greater than or equal to 4. By the second property for $E_{r-1}$, there exists $b\in E_{r-1}$ such that $\kappabar b =a$. It suffices to prove that $b$ is a $d_{r-1}$-cycle. Suppose that $d_{r-1}(b) = c$. The latter implies that $c$ lives in filtration greater that $(r-2)$, hence is $\kappabar$-free by the first property. Since $a$ is a $d_{r-1}$-cycle by assumption, we have, by $\kappabar$-linearity, that 
\[0 = d_{r-1}(a) = d_{r-1}(\kappabar b) = \kappabar c.\]
This means that $c=0$ and so $b$ is a $d_{r-1}$-cycle, as required. 
\end{proof}

\begin{term}For convenience, we will call all $\kappabar$-multiples of a class which has filtration less than four the $\kappabar$-family of that class. By part $(2)$ of the above lemma, at any term of the spectral sequence, every class belongs to some $\kappabar$-family. The following corollary tells us how these $\kappabar$-families are organized.
\end{term}

\begin{Corollary} \label{StructureKappa}
\begin{enumerate}
\item
At any term of the spectral sequence, all non-zero $\kappabar$-power torsion classes survive to the $E_{\infty}$-term.
\item
 Every $\kappabar$-free family consisting of permanent cycles is truncated by one and only one other $\kappabar$-free family. 
\end{enumerate}
\end{Corollary}
\begin{proof}
For part $(1)$, let $a\in E_r$ be a non-zero $\kappabar$-power torsion class. By part (1) of \Cref{StructureSS}, $a$ is in filtration less than or equal to $r-1$. It follows that $a$ cannot be hit by any differential from the $E_r$-term onwards. Moreover, by part (1) of \Cref{StructureSS}  again, the possible targets of $d_{r'}(a)$, $r'\geq r$ are $\kappabar$-free classes. Since $a\in E_{r}$ is $\kappabar$-power torsion, it is a permanent cycle, by $\kappabar$-linearity. Therefore, $a$ persists to the $E_\infty$-term.
 
For part $(2)$, let $a$ be a permanent cycle of filtration striclty less than four which is $\kappabar$-free at the $E_2$-term. Then the $\kappabar$-family of $a$ consists of permanent cycles.
 Since $\kappabar$ is nilpotent at the $E_{\infty}$-term of the elliptic spectral sequence for $tmf$, some $\kappabar$-multiple of $a$ must be hit by a differential. Suppose that $a$ is $\kappabar$-free at the $E_r$-term and that $\kappabar^la$ is the smallest $\kappabar$-multiple of $a$ that is hit by a differential, say $d_r(b) = \kappabar^l a$. Since $a$ is $\kappabar$-free at the $E_r$-term, so is  $b$. It follows that the $\kappabar$-multiples of $b$ truncate those of $\kappabar^{l}a$ by differentials $d_r$, i.e., $d_r(\kappabar^n b) = \kappabar^{l+n}a$. So, all the classes $\kappabar^ka$ for $k\leq l-1$ are non-zero $\kappabar$-power torsion classes on the $E_{r+1}$-term, hence are essential by part $(1)$.

Finally, we claim that $b$ has filtration less than four so that the $\kappabar$-family of $b$ truncates the $\kappabar$-family of $a$.
If $b$ had filtration greater than or equal to $4$, then $b$ would be divisible by $\kappabar$, i.e., there would exist $c\in E_r$ such that $\kappabar c = b$, by \Cref{StructureSS} part $(2)$. By $\kappabar$-linearity, we have that $\kappabar^l a = d_r(b) = \kappabar d_r(c)$, and so $\kappabar (\kappabar^{l-1}a-d_r(c)) = 0$. This means that $d_r(c) = \kappa^{l-1}a$ because $d_r(c)-\kappabar^{l-1}a$ has filtration at least $r$ so that it is $\kappabar$-free, by \Cref{StructureSS} part $(1)$. This contradicts the minimality of $\ell$, so $b$ has filtration less than four.  
\end{proof}

\begin{slogan}\label{slogankappabar}
The $\kappabar$-free families at the $E_r$-page come in pairs. The first member of the pair is a family consisting of permanent cycles. The second member is a family which eventually supports differentials (i.e., possibly at a later page) to truncate the first family.
\end{slogan}

\begin{Corollary}\label{StructureDelta} At the $E_r$-term, we have:
\begin{enumerate} 
\item The homomorphism $E_r^{s,t}\rightarrow E_r^{s,t+192}$ induced by multiplication by $\Delta^8$ is an injection for all $s$ and $t$,
\item If $a$ is a class of the $E_2$-term such that $\Delta^8a$ is a $d_r$-cycle, then $a$ is also a $d_r$-cycle. 
\end{enumerate} 
\end{Corollary}
\begin{proof}
We prove part $(1)$ by induction on $r\geq 2$. For $r=2$, this can be seen from the explicit structure of the $E_2$-term. Suppose the $E_{r'}$-term has these properties for $r'< r$. Let us prove part (1) for $E_r$. Let $a\in E_{r-1}$ represent a class of $E_{r}$. If $\Delta^8 [a]=0 \in E_{r}.$ This means that there exists $b\in E_{r-1}$ such that $d_{r-1}(b) = \Delta^8a$. 
It is obvious that $b$ lives in stem at least $192$, hence there exists $c\in E_{r-1}$ such that $b= \Delta^8 c$, by the induction hypothesis. It follows that $\Delta^8(d_{r-1}(c)-a) = 0$, and so $d_{r-1}(c) = a$ because of part (1) of the induction hypothesis. Thus $[a] = 0\in E_{r}$, as needed.

For part (2), by induction, suppose that $a$ is a $d_{r-1}$-cycle. We need to prove that $a$ is a $d_{r}$-cycle.  In effect, if $d_{r}(a) = b$, then 
\[0= d_{r}(\Delta^8a) = \Delta^8d_r(a) = \Delta^8 b.\] 
By part (1), $b=0$, and so $d_{r}(a) = 0$, as needed.
\end{proof}

Finally, we will also use the following result to establish the differentials.

\begin{lem}[Vanishing line]\label{VanishingLemma}
The spectral sequence for $\pi_* tmf \wedge Y$ degenerates at the $E_{24}$-term and has a  horizontal vanishing line at $s=24$, i.e., $E_{24}^{s,t} = E_{\infty}^{s,t}=0$ for $s\geq 24$.
\end{lem}

\begin{proof}
We know that $\kappabar^6$ is hit by a differential $d_{23}$ in the elliptic spectral sequence for $tmf$, see \cite{tbauer}.
 This means that at the $E_{24}$-term of the elliptic spectral sequence for $tmf\wedge Y$, all the classes are annihilated by $\kappabar^6$, hence are $\kappabar$-power torsion. Therefore, by \Cref{StructureSS}, all the classes in the $E_{24}$-term are in filtrations less than $24$, meaning that the spectral sequence has the horizontal vanishing line at $s= 24$, i.e., $E^{s,t}_{r} = 0$ for $s \geq 24$ and $r\geq 24$.
\end{proof}

\begin{rem}
The cofiber sequence
\[V(0) \xrightarrow{i} Y \xrightarrow{p} \Sigma^2 V(0) \xrightarrow{\eta} \Sigma V(0)\]
gives rise to maps of spectral sequences
\[i_* \colon E_2^{s,t}(V(0)) \to E_2^{s,t}(Y), \ \ \ \ p_* \colon E_2^{s,t}(Y) \to E_2^{s,t-2}(V(0))\]
as well as a long exact sequence 
\begin{align}
\label{LES} \ldots \to tmf_{*+1}V(0) \xrightarrow{\eta} tmf_*V(0)\xrightarrow{i_*} tmf_*Y\xrightarrow{p_*} tmf_{*-1}V(0) \to \ldots  \end{align}
\end{rem}


\subsection{The $d_3$, $d_5$ and $d_7$-differentials}
Note that for $r$ even, $E_r(Y) \cong E_{r+1}(Y)$ since the spectral sequence is concentrated in bi-degrees $(s,t)$ with $t$ even.
The differentials in this section are depicted in Figures~\ref{gensone}, \ref{d5d9one}, \ref{d5d9three} and \ref{d5d9four}.

\begin{Proposition} 
 There is no non-trivial $d_3$-differential, and so $E_3(Y) \cong E_5(Y)$.
\end{Proposition}

\begin{proof}
Since $\Delta$ is a $d_3$-cycle in the elliptic spectral sequence of $tmf$, the $d_3$-differentials are $\F_2[v_1, \nu, \kappabar, \Delta]/( v_1\nu, \nu^3, v_1^4\kappabar)$-linear. All the generators
listed in \Cref{cormodstructureE2Y} are $d_3$-cycles for degree reasons. 
\end{proof}

We then get the following result for degree reasons.
\begin{cor}
The classes in stems $t-s<24$ are permanent cycles.
\end{cor}

\begin{lem}
	The $d_5$-differentials are linear with respect to $\kappabar, \nu, v_1, \Delta^2$ and are determined by
	\begin{align*}
	d_5(\Delta)& = \nu\kappabar,  & d_5(\Delta a[5,1]) &=\nu \kappabar a[5,1] \\
	d_5(\Delta a[15,1])& = \nu\kappabar a[15,1],  & d_5(\Delta a[20,2]) &=\nu \kappabar a[20,2]
	\end{align*}
	under multiplication by elements of $\F_2[\Delta^2, \kappabar, \nu, v_1]/(v_1\nu, \nu^3, \kappabar v_1^4)$.
\end{lem}
\begin{proof}
For linearity, we only need to prove the $\Delta^2$-linearity. 
Note that $d_5(\Delta) = \nu\kappabar$ in the elliptic spectral sequence of $tmf$. 
By Leibniz rule and the fact that $E_2(Y)$ is $2$-torsion, 
	\[ d_5(\Delta^2x) = 2\Delta d_5(\Delta)x + \Delta^2 d_5(c) = \Delta^2 d_5(x). \]
	Using the module structure over the elliptic spectral sequence of $tmf$, we get 
\[d_5(\Delta a[5,1]) = d_5(\Delta)a[5,1] + \Delta d_5(a[5,1]) =  \nu\kappabar a[5,1].\]
The other arguments are similar.
\end{proof}

\begin{lem}
There are no non-trivial $d_7$-differentials.
\end{lem}
\begin{proof}
This is an immediate consequence of sparseness.
\end{proof}

The following observation will be crucial for our computation and is motivated by \Cref{slogankappabar}.

\begin{cor}[\Cref{gensone}]\label{E7gen}
The $\kappabar$-free families on the $E_9$-term of the elliptic spectral sequence of $tmf \wedge Y$ in stems $0\leq t-s <48$ are given by the following $24$ classes
\begin{align*}
&a[0,0] & &  a[2,0]= v_1a[0,0] &  & a[4,0]= v_1^2a[0,0] \\
 & a[5,1]  & & a[7,1]= v_1a[5,1] & & a[9,1]= v_1^2a[5,1]  \\
 &  a[12,2]  & &a[14,2]=v_1 a[12,2] && a[15,1]    \\
& a[17,3] & & a[19,3]=v_1a[17,3]  & &  a[20,2] \\
& a[26,0] = \Delta v_1 a[0,0] & & a[28,0]=\Delta v_1^2 a[0,0] && a[30,0]= \Delta v_1^3 a[0,0] \\ 
 & a[30,2]= \Delta \nu^2 a[0,0]  & & a[31,1] =\Delta v_1a[5,1] && a[33,1] =  \Delta v_1^2a[5,1] \\ 
 & a[35,3] = \Delta \nu^2a[5,1] & & a[36,2] = \Delta a[12,2]   & & a[38,2] = \Delta v_1 a[12,2]  \\
 & a[41,3] = \Delta a[17,3]  & & a[43,3] = \Delta v_1 a[17,3]  &  & a[45,3] =  \Delta v_1^2a[17,3]
\end{align*}
All $\kappabar$-free families at $E_9$ are given by these classes and their $\Delta^2$-multiples. All the elements in filtrations four and above are $\kappabar$-multiples of these generators. 
\end{cor}

\begin{figure}[h]
\includegraphics[page=1, width=\textwidth]{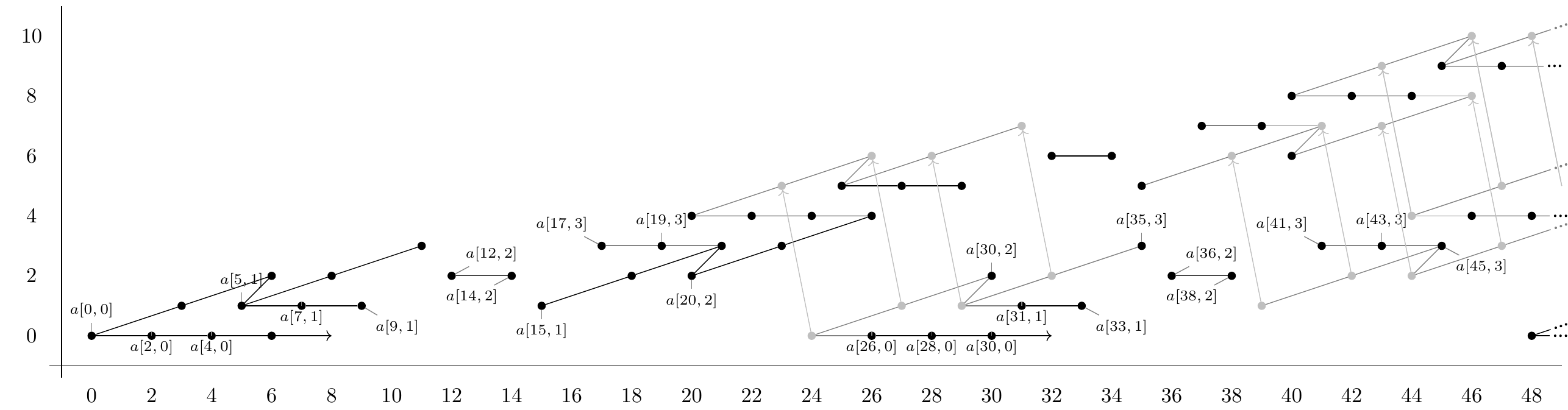}
\caption{$d_5$-differentials in stems 0 to 48 and $\kappabar$-free generators at $E_9$}
\label{gensone}
\end{figure}

The generators of the $\kappabar$-free families in stems $0\leq t-s <48$ are presented in \Cref{gensone}. 
The $\kappabar$-free generators in the  range $0\leq t-s <192$ are given by multiples of these with $\Delta^2, \Delta^4$ and $\Delta^6$ and all other $\kappabar$-free generators are multiples of these with the powers of $\Delta^8$. By  \Cref{StructureKappa}, each $\kappabar$-free family consisting of permanent cycles is truncated by exactly one other $\kappabar$-free family. Thus, using the $\Delta^8$-linearity and \Cref{StructureDelta}, we see that the $24\times 4$ $\kappabar$-free generators in the range $0\leq t-s <192$ organize themselves as follows. Exactly half of them are permanent cycles and the other half are not. The $\kappabar$-family of each non-permanent $\kappabar$-free generator supports a differential that hits the $\kappabar$-family of exactly one of the other permanent generators. Note that the truncation must begin in stems less than four by \Cref{StructureKappa}. This allows us to determine longer differentials before settling shorter ones.

All $24$ $\kappabar$-free generators in the range $0\leq t-s<48$ are permanent cycles due to sparseness and in the next section we will find their ``partners".

\subsection{The $d_9$-differentials}

To analyze the $d_9$-differentials, we make the following observation, which, in some sense, is a very small part of the geometric boundary theorem as in \cite[Appendix 4]{behrens_EHP}. 
\begin{lem}\label{lemtrickY}
Let $a \in E_r^{s,t}(Y)$ so that $p_*(a)\in E_r^{s,t-2}(V(0))$. Suppose $p_*(a)$ persists to the $E_{r'}$ term  for some $r'\geq r$ and that there is a non-trivial differential, $d_{r'}(p_*a) \neq 0$. Then $d_{r''}(a)\neq 0$ for $r'' \leq r'$.
\end{lem}
\begin{proof}
This is a straight forward application of naturality. Our assumptions imply that $a$ cannot be hit by a differential $d_{r''}$ for $r''\leq r'$ and, furthermore, that if it persists to the $E_{r'}$ term, that $d_{r'}(a)=b'$ for $b'$ such that $p_*b'=b$. 
\end{proof}

The differentials below are depicted in Figures~\ref{gensone}, \ref{d5d9one}, \ref{d5d9three} and \ref{d5d9four}.

\begin{lem}[Figures~\ref{d5d9one}, \ref{d5d9three} and \ref{d5d9four}]
There are $d_9$-differentials, for $i=0,1$,
\begin{enumerate}[(1)]
\item $d_9(\Delta^{4i+2}a[0,0])=\kappabar^2\Delta^{4i}  v_1a[5,1] $
\item $d_9(\Delta^{4i+2}a[5,1])=\kappabar^2\Delta^{4i}  a[12,2]$
\item $d_9(\Delta^{4i+3}v_1a[0,0])=\kappabar^2 \Delta^{4i+1} v_1^2a[5,1] $
\item $d_9(\Delta^{4i+2}a[17,3])=\kappabar^3 \Delta^{4i} v_1^2 a[0,0]$
\item $d_9(\Delta^{4i+3}a[17,3])=\kappabar^3\Delta^{4i+1} v_1^2 a[0,0]$
\item $d_9(\Delta^{4i+3}a[12,2])=\kappabar^2 \Delta^{4i+1} v_1a[17,3]$
\item $d_9(\Delta^{4i+2}a[12,2]) =\kappabar^2 \Delta^{4i} v_1 a[17,3]$
\item $d_9(\Delta^{4i+3}v_1a[5,1]) =\kappabar^2 \Delta^{4i+1}\ v_1 a[12,2]$
\item $d_9(\Delta^{4i+2}v_1a[0,0])=\kappabar^2\Delta^{4i}  v_1^2a[5,1] $
\item $d_9(\Delta^{4i+2}v_1a[5,1])=\kappabar^2\Delta^{4i}  v_1a[12,2]$
\item $d_9(\Delta^{4i+3}v_1a[17,3])=\kappabar^3\Delta^{4i+1} v_1^3 a[0,0].$
\item $d_9(\Delta^{4i+3}v_1a[12,2])=\kappabar^2 \Delta^{4i+1} v_1^2a[17,3]$
\end{enumerate}
\end{lem}
\begin{proof}
Let $i=0$. The differentials (1) and (3) are the image of a differential in $E_2(V(0))$ under $i_*$. The second differential (2) follows $v_1$-linearity and from the fact that $d_9(\Delta^{4i+2}\aseven)=\kappabar^2\Delta^{4i} \kappa$ in $E_2(V(0))$, $i_*(\aseven)=v_1a[5,1]$ and $i_*(\kappa)=v_1a[12,2]$. 

For (4), we use \Cref{lemtrickY}. In $E_*(V(0))$, we have $d_{11}( \Delta^2\eta \kappa) = \eta^2\kappabar^3$. Since $p_*( \Delta^2 a[17,3])= \Delta^2 \eta \kappa$, $\Delta^2 a[17,3]$ supports a differential of length at most $11$. This $d_9$ is the only choice. The argument for (5) is the same, with one more power of $\Delta$.

For (6), note that $p_*(\Delta^3 a[12,2]) = \Delta^3 v_1\eta x$. Since $d_{9}(\eta v_1  x) = \nu \kappa \kappabar^2 \Delta$, the class $\Delta^3 a[12,2]$ supports a differential of length at most $9$. This is the only choice. 
 
The arguments (1)--(6) when $i=1$ are the same as those for $i=0$.

For (7)--(8), note that 
from our computation above, $tmf_{59}Y \cong \Z/2$. This forces (7) when $i=0$. Arguing in a similar way, $tmf_{79}Y =0$, $tmf_{155}Y\cong \Z/2$ and $tmf_{175}Y=0$ imply the other $d_9$s.

The $d_9$-differentials $(9)$-$(12)$ follow from those of $(1)$, $(2)$, $(5)$, $(6)$, respectively, by $v_1$-linearity.
\end{proof}

\begin{rem}
It turns out these are all the $d_9$-differentials. For degree reasons, there can be very few other $d_9$s.
The class $\Delta^5v_1a[0,0]$ is the image of a $d_{9}$-cycle in $E_9(V(0))$ so does not support a $d_9$.

The only other possible $d_9$ differentials for degree reasons are
\begin{itemize}
\item A non-trivial $d_9$ on $\Delta^{5} a[17,3]$. This does not happen since it implies a non-trivial $d_9$ on $v_1\Delta^{5} a[17,3]= \Delta^4 a[43,3]$, but this family has already been paired: it is truncated by $\Delta^6a[36,2]$.
\item A nontrivial $d_9$ on $\Delta^{4} a[17,3]$, truncating the $\kappabar$-family of $\Delta^2a[4,0]$. We will see below that this does not happen, but at this point, we leave this undecided.
\end{itemize}

\end{rem}

\begin{figure}[h]
\includegraphics[page=1, width=\textwidth]{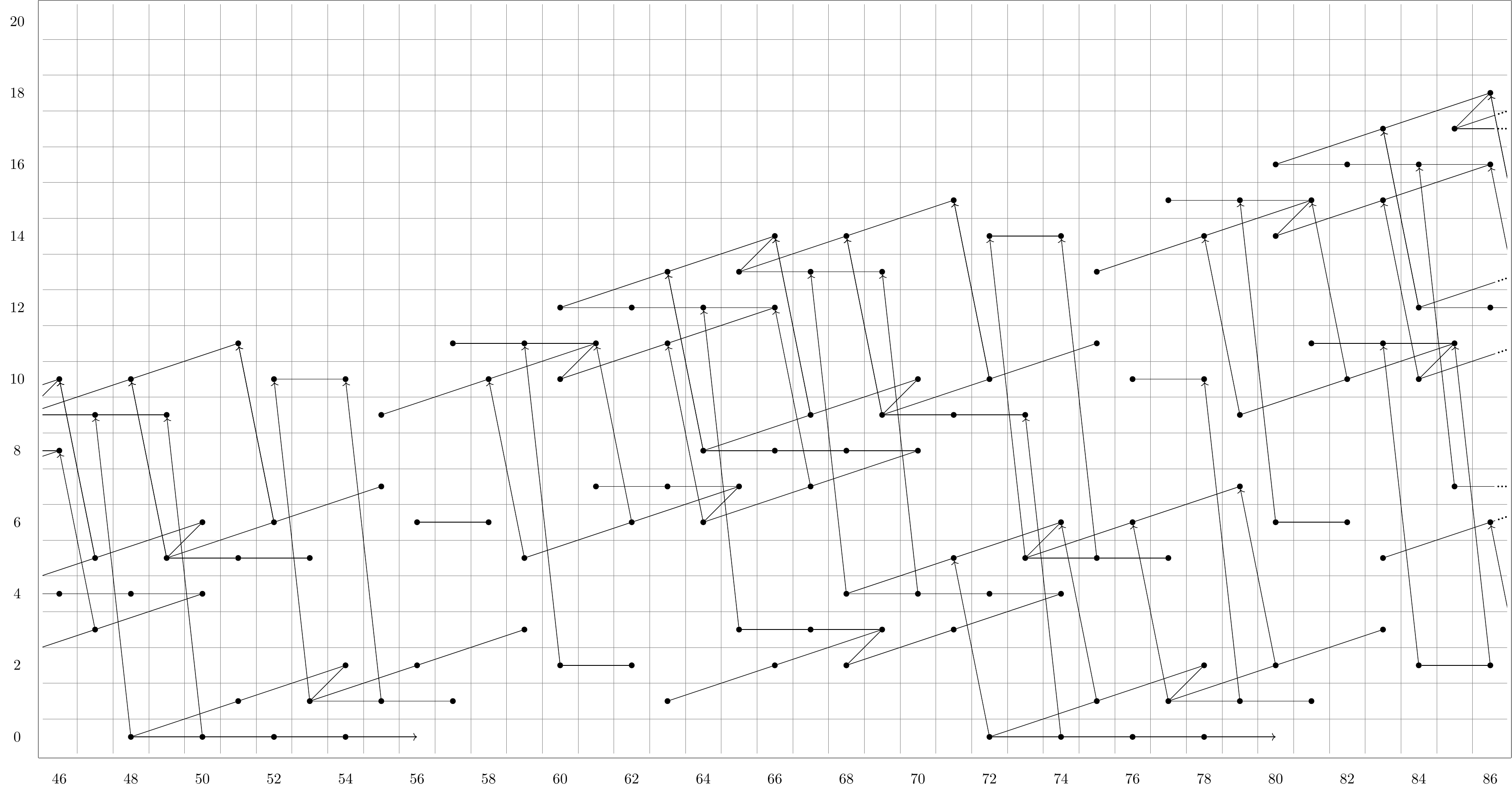}
\caption{$d_5$ and $d_9$ differentials in stems $46$ to $86$}
\label{d5d9one}
\end{figure}

\subsection{Higher differentials}
We begin our analysis using \Cref{slogankappabar}. The reader should remember that we only need to analyze the generators of the $\kappabar$-free families, which are in filtration less than four. All differentials discussed in this section are depicted in Figures~\ref{d11d23one} and \ref{d11d23four}.

\begin{lem}\label{lemhigherfirstd19}
There are differentials
\begin{enumerate}[(1)]
\item $d_{19}(\Delta^{4}a[5,1])=\kappabar^{5}a[0,0]$
\item $d_{19}(\Delta^{5}v_1a[5,1])=\kappabar^{5}\Delta v_1 a[0,0]$
\item \label{16} $d_{19}(\Delta^{4}a[36,2]=\kappabar^{5}a[31,1]$
\item $d_{19}(\Delta^{4}a[41,3])=\kappabar^{5}a[36,2]$
\item $d_{19}(\Delta^{4}a[26,0])=\kappabar^{4}a[41,3]$
\end{enumerate}
\end{lem}
\begin{proof}
For (1), since the element $\kappabar^4\in \pi_{80}(tmf\wedge V(0))$ is not divisible by $\eta$ and $\kappabar^5\in \pi_{100}(tmf\wedge V(0))$ is divisible by $\eta$, the $\kappabar$-family of $a[0,0]$ in the elliptic spectral sequence for $tmf \wedge Y$ must be truncated at $\kappabar^5a[0,0]$.
Remembering that the source has to have filtration less than four, the only possibility is this differential.

Inspection then show that the differentials (2)-(4) are the only possibilities to satisfy \Cref{slogankappabar}.
\end{proof}

\begin{lem}\label{lemhighersecondd17}
There are differentials
\begin{enumerate}[(1)]

\item \label{8} $d_{17}(\Delta^{4}a[0,0])=\kappabar^{4}a[15,1]$
\item \label{8'} $d_{17}(\Delta^{4}a[15,1])=\kappabar^{4}a[30,2]$
\end{enumerate}
\end{lem}
\begin{proof}
For (1), note that in $\pi_*(tmf\wedge V(0))$, $\kappabar^3 y$ is not divisible by $\eta$ and $\kappabar^4y= 0$. The class $y$ maps to $a[15,1]$ under $i_*$ so
it follows that the $\kappabar$-family of $a[15,1]$ is truncated at $\kappabar^4a[15,1]$. The only possibility is this differential. 

For (2), using the long exact sequence, we obtain that
$\pi_{111}(tmf \wedge Y) =\mathbb{Z}/2$.
 By part \Cref{lemhigherfirstd19} (\ref{16}), the class $\kappabar^4a[31,1]\in E_{5}^{17,128}$ survives the spectral sequence and so detects the unique non-trivial class of $\pi_{111}(tmf\wedge Y)$. This implies that the class $\Delta^4a[15,1]\in E_5^{1,112}$ must support a differential. Taking into account the $d_9$ differentials proves (2).
\end{proof}

\begin{lem}
There is a differential $d_{23}(\Delta^{4}a[30,2])=\kappabar^{6}a[5,1]$.
\end{lem}
\begin{proof}
By inspection, taking into account the $d_9$s, the only generators that can be paired with $a[5,1]$ are $\Delta^{4}a[30,2]$ and $\Delta^{4}a[30,0]$. However, it cannot be $\Delta^{4}a[30,0]$ because such a differential would have length $25$, contradicting \Cref{VanishingLemma}.
\end{proof}

\begin{lem}
For $i=0,1$, there are differentials:
\begin{enumerate}[(1)]
\item 
$d_{11}(\Delta^{4i+2}a[15,1])=\kappabar^{3} \Delta^{4i}a[2,0]$
\item $
d_{11}(\Delta^{4i+2}a[28,0])=
\kappabar^{2}\Delta^{4i}a[35,3]$ 
\end{enumerate}
\end{lem}
\begin{proof}
In (1), for both $i=0,1$, these are the image of differentials in the spectral sequence $E_*(V(0))$. Both source and targets survive to $E_{11}(Y)$ and so these two differentials occur.

For (2), the long exact sequence shows that $\pi_{75}(tmf \wedge Y)=\Z/2$. \Cref{lemhighersecondd17} (\ref{8}) implies that the class $\kappabar^3 a[15,1]\in E_{7}^{13, 88}$ survives the spectral sequence and detects the unique non-trivial element of  the $\pi_{75}(tmf \wedge Y)$. On the other hand, the class $\kappabar^2\Delta \nu^2 a[5,1]\in E_7^{11,86}$ is a permanent cycle. Thus, it must be hit by a differential and this is the  possibility. 

For $i=1$, by taking into account the $d_9$-differentials and the $d_{17}$-differential \Cref{lemhighersecondd17} (\ref{8'}), we see that $\Delta^4 a[35,3]$ is a permanent cycle, which is $\kappabar$-free at the $E_{11}$-term.  By inspection, the only class which can truncate its $\kappa$-family is $\Delta^6a[28,0]$ by the indicated $d_{11}$-differential.
\end{proof}

\begin{lem}
There are differentials:
\begin{enumerate}[(1)]
\item $d_{13}(\Delta^{2}a[30,2])=\kappabar^{3}a[17,3]$
\item $d_{13}(\Delta^{2}a[33,1])=\kappabar^{3}a[20,2]$
\end{enumerate}
\end{lem}
\begin{proof}
For (1), it follows from \eqref{LES} that
$\pi_{78}(tmf\wedge Y)\cong \Z/2$.
By sparseness, either $\Delta^{2}a[30,2]$ or $\Delta^{2}a[30,0]$ is a permanent cycle detecting the non-zero element of $\pi_{78}(tmf\wedge Y)$. Suppose that 
\[\Delta^{2}a[30,2] = \Delta^3 \nu^2a[0,0]\]
is a permanent cycle detecting a class $\alpha \in  \pi_{78}(tmf \wedge Y)$. At $E_2$, $\Delta^3 \nu^2a[0,0]$ is in the image of $i_* \colon E_2(V(0)) \to E_{2}(Y)$ and so $p_*(\Delta^3 \nu^2a[0,0])= 0$. However, since $\pi_{78}(tmf \wedge V(0))=0$, $p_*\alpha \neq 0$ in $\pi_{76}(tmf \wedge V(0))$ and so is detected by a non-zero class in filtration $s>2$, but such a class does not exist.
We conclude that $\Delta^2a[30,0]$ is a permanent cycle and that $\Delta^2a[30,2]$ supports the stated differential.
For (2), by inspection, only $\Delta^{2}a[33,1]$ and $\Delta^{4}a[5,1]$ can support differentials truncating the $\kappabar$-family of $a[20,2]$. But $\Delta^{4}a[5,1]$ is already paired with $a[0,0]$. 
\end{proof}

\begin{figure}[H]
\includegraphics[page=1, width=\textwidth]{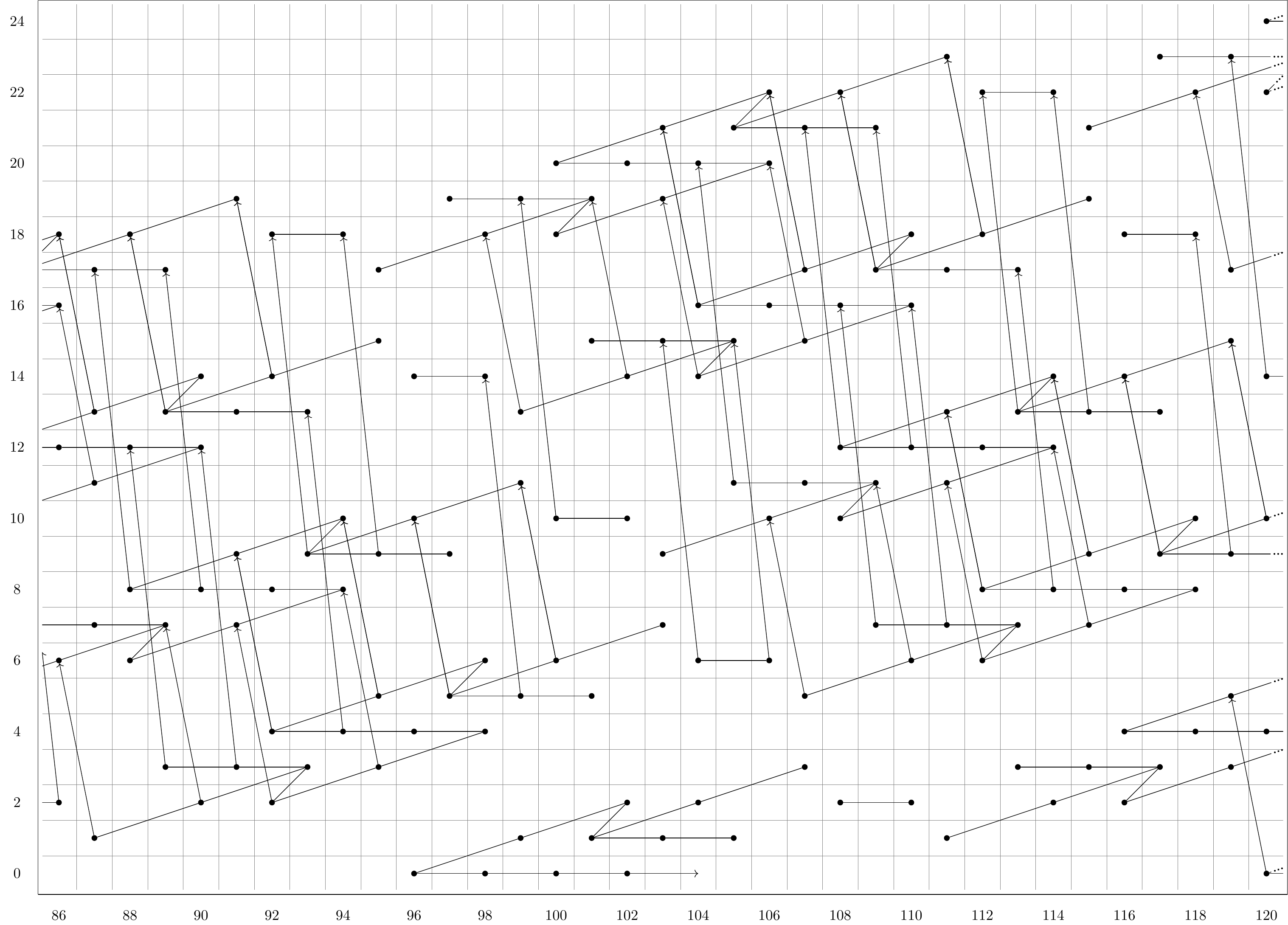}

\vspace{0.1in}

\includegraphics[page=1, width=\textwidth]{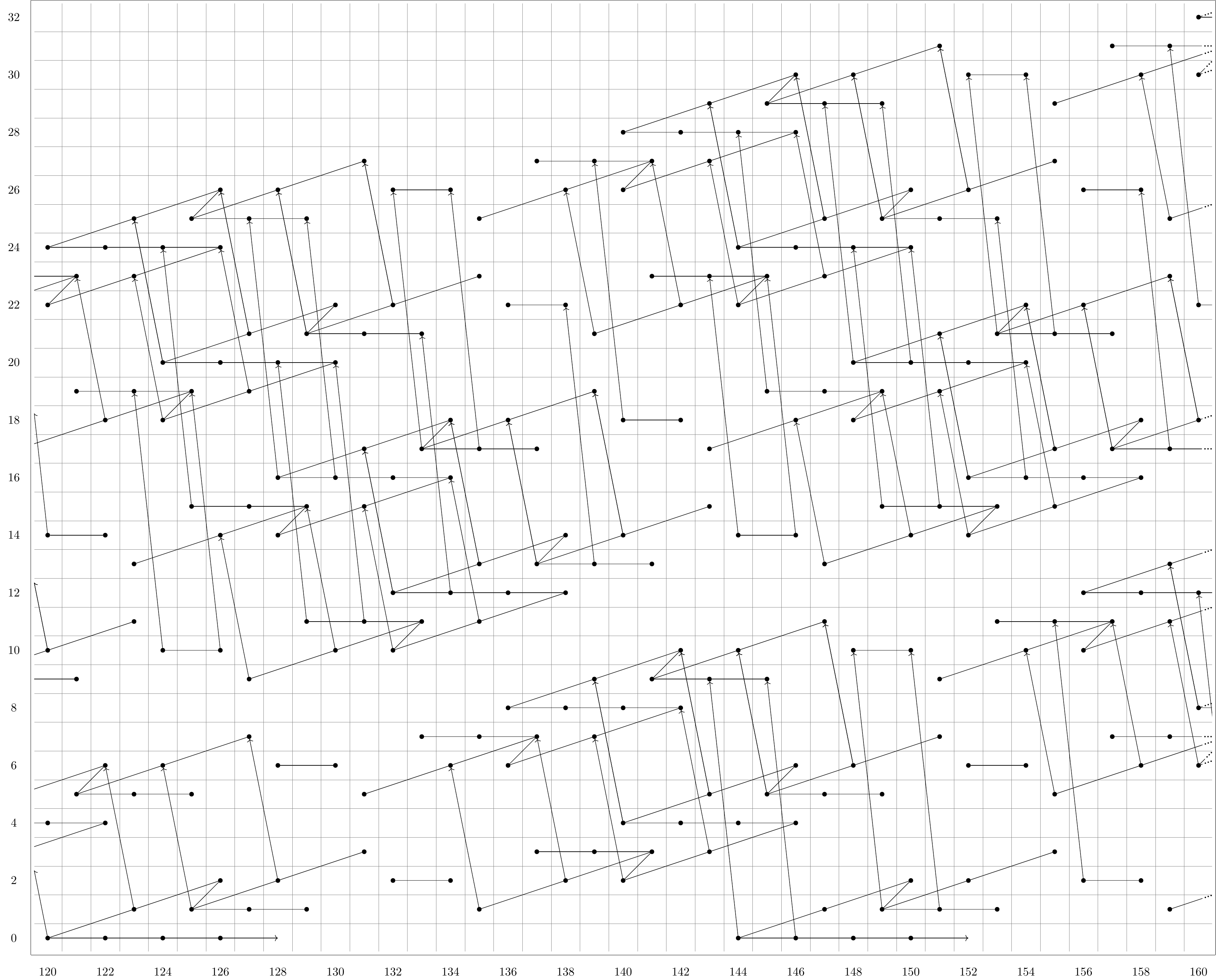}
\caption{$d_5$ and $d_9$ differentials in stems $86$ to $160$}
\label{d5d9three}
\end{figure}


\begin{figure}[H]
\includegraphics[page=1, width=\textwidth]{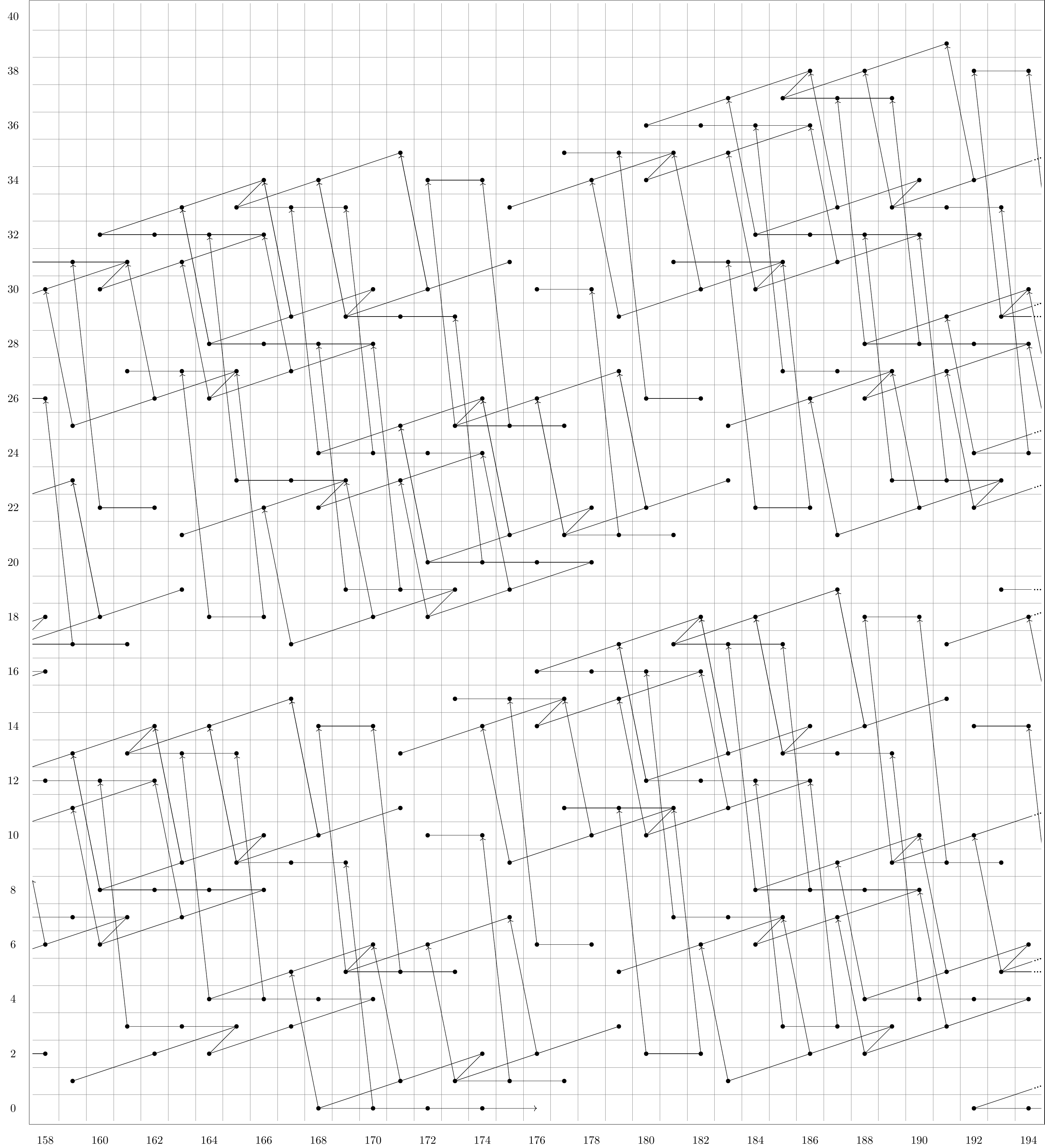}
\caption{$d_5$ and $d_9$ differentials in stems $160$ to $194$}
\label{d5d9four}
\end{figure}

\begin{Proposition}
The following classes are $\kappabar$-free permanent cycles:
\begin{align*}
(A) : & & \begin{array}{ccccc}
\Delta^{2}a[4,0]&\Delta^{2}a[9,1]&\Delta^{2}a[14,2]&\Delta^{2}a[19,3]&\Delta^{2}a[20,2]
\\
\Delta^{2}a[30,0] & \Delta^{2}a[35,3]&\Delta^{2}a[45,3] &\Delta^{4}a[17,3] &\Delta^{4}a[20,2] 
\end{array}
\end{align*}
and the following classes are not permanent cycles: 
\begin{align*}
(B) : & &  \begin{array}{ccccc}
\Delta^{6}a[4,0]
&\Delta^{6}a[9,1] 
 &\Delta^{6}a[14,2]
 &\Delta^{6}a[19,3] 
  &\Delta^{6}a[20,2]\\
\Delta^{6}a[30,0]
&\Delta^{6}a[30,2]
&\Delta^{6}a[33,1]
& \Delta^{6}a[35,3]
&\Delta^{6}a[45,3]
\end{array}
\end{align*}
Consequently, in the elliptic spectral sequence for $tmf \wedge Y$, each generator in (B) truncates some $\kappabar$-multiple of one and only one generator in (A).
\end{Proposition}

\begin{proof}
These are the remaining generators of $\kappabar$-free families.
No class in (B) can be a permanent cycle because the $\kappabar$-family of a class of (B) cannot be truncated. 
This means that all the $10$ classes of (B) are non-permanent cycles, and so all the $10$ classes of (A) are permanent cycles.
\end{proof}

\begin{lem}
We have the following differentials:
\begin{enumerate}[(1)]
\item $d_{19}(\Delta^{6}a[4,0])=\kappabar^{4}\Delta^{2}a[19,3]$ 
\item $d_{19}(\Delta^{6}a[9,1])=\kappabar^{5}\Delta^{2}a[4,0]$
\item $d_{19}(\Delta^{6}a[14,2])=\kappabar^{5}\Delta^{2}a[9,1]$
\item $d_{19}(\Delta^{6}a[19,3])=\kappabar^{5}\Delta^{2}a[14,2]$
\item $d_{17}(\Delta^{6}a[20,2])=\kappabar^{4}\Delta^{2}a[35,3]$ 
\item $d_{13}(\Delta^{6}a[33,1])=\kappabar^{3}\Delta^{4}a[20,2]$
\item $d_{17}(\Delta^{6}a[35,3])=\kappabar^{5}\Delta^{2}a[30,0]$ 
\item $d_{23}(\Delta^{6}a[45,3])=\kappabar^{6}\Delta^{2}a[20,2]$ 
\end{enumerate}
\end{lem}
\begin{proof}
Taking into account the differentials shown above,  these are only possible pairings remaining between the classes in (B) which are the sources in (1)--(8) and classes of (A).\end{proof}

\begin{rem}\label{rem2left}
There are only two generators in (B) left living in the same topological degree, namely $\Delta^6a[30,0]$ and $\Delta^6 a[30,2]$. Each of these supports a differential truncating the $\kappabar$-families of either $\Delta^4a[17,3]$ or $\Delta^2a[45,3]$ and one differential determines the other. 
\end{rem}

Determining the last differential pattern turns out to be unfortunately tricky (as far as we know). A crucial step towards settling the last differentials is to establish the following extension in the $E_{\infty}$-term of the elliptic spectral sequence for $tmf \wedge Y$.
\begin{proposition}\label{keyExt} There is an exotic extension $$\nu^2 (\nu\Delta^6a[0,0]) = \kappabar^2 \Delta^4 a[17,3].$$
 \end{proposition} 
To prove this extension, we need some intermediate results.

\begin{lemma}\label{Massey-prod} In $\Ext_{\Lambda'}^{*,*}(A', A'/(2,a_1)\otimes \mathcal{F}_*(Y))$, there is a Massey product 
\[\langle\eta, \nu, \Delta^4a[12,2] \rangle = \Delta^4a[17,3].\]
\end{lemma}
\begin{proof} 
Since $\Delta^4a[12,2] = \eta \Delta^4 a[11,1]$ (see \Cref{fig:Amod2a1}), we have that
\[\langle\eta, \nu, \Delta^4a[12,2] \rangle = \langle \eta, \nu, \eta \Delta^4a[11,1]\rangle \supseteq \langle \eta, \nu, \eta\rangle \Delta^4a[11,1]  = \nu^2 a[11,1] = a[17,1].\]
The indeterminacy is zero since
\[\eta\Ext_{\Lambda'}^{2,114}(A', A'/(2,a_1)\otimes \mathcal{F}_*(Y)) +  \Ext_{\Lambda'}^{1,6}(A', A'/(2))\Delta^4a[12,2] = 0. \qedhere\]
\end{proof}

\begin{proposition}\label{Massey-Y} In $\Ext_{\Lambda'}^{*,*}(A', \mathcal{F}_*(Y))$, there is a Massey product 
\[\langle\eta, \nu, \Delta^4a[12,2] \rangle = \Delta^4a[17,3].\]
\end{proposition}
\begin{proof}  
Let $f_* \colon \Ext_{\Lambda'}^{*,*}(A', \mathcal{F}_*(Y)) \rightarrow \Ext_{\Lambda'}^{*,*}(A', A'/(2,a_1)\otimes \mathcal{F}_*(Y))$ be the map
induced by the $\Lambda$-comodule homomorphism $\mathcal{F}_*(Y) \rightarrow A'/(2,a_1)\otimes \mathcal{F}_*(Y)$. 
By naturality of Massey products, we have that 
\begin{equation*} 
f_*(\langle \eta, \nu, \Delta^4a[12,2]\rangle) \subseteq \langle \eta, \nu, f_*(\Delta^4a[12,2])\rangle.
\end{equation*} 
Further, $f_{*}(\Delta^4a[12,2]) = \Delta^4a[12,2]$. 
By \cref{Massey-prod}, the above equation gives
\[f_*(\langle \eta, \nu, \Delta^4a[12,2]\rangle) = \Delta^4 a[17,3].\]
The pre-image of $\Delta^4 a[17,3]$ is the same-named class. The indeterminacy is zero.
\end{proof}

\begin{Lemma}
There is an element of $\pi_{108}(tmf\wedge Y)$ detected by $\Delta^4a[12,2]$ and annihilated by $\kappabar^2$. 
\end{Lemma}
\begin{proof} We have already determined $E_{\infty}(Y)$ in stems $t-s=108,148$. We see that there is a short exact sequence
\[0\rightarrow \Z/2\{\kappabar^2\Delta^2 a[20,2]\}\rightarrow G \rightarrow \Z/2\{\Delta^4a[12,2]\}\rightarrow 0\]
where $G\subseteq \pi_{108}(tmf\wedge Y)$ is the subgroup of elements detected in positive filtration. 
At the $E_{\infty}$-term in stem  $t-s=148$, the only non-zero class in positive filtration is $\kappabar^4\Delta^2a[20,2]$. In particular, $\kappabar^2 \Delta^4 a[12,2] =0$. So, one of the classes detected by  $\Delta^4a[12,2]$ satisfies the claim.
\end{proof}
We will denote also by $\Delta^4 a[12,2]$ the element in $\pi_{108}(tmf\wedge Y)$, which is detected by $\Delta^4a[12,2]$ and is annihilated by $\kappabar^2$.

\begin{proposition}\label{SomeExo} There are the following relations in $\pi_*(tmf\wedge Y)$:
	\begin{enumerate}
	\item
	$\nu^2 [\nu\Delta^6a[0,0]] \ne 0$
	\item
	$\eta [\nu\Delta^6a[0,0]] = 0$
	\end{enumerate}
\end{proposition}
\begin{proof} The class detected by $\nu\Delta^6a[0,0]$ lifts to $\pi_*(tmf\wedge V(0))$ and there is a lift detected by $\nu\Delta^6$. But in $\pi_*(tmf \wedge V(0))$, $\nu^2 [\nu\Delta^6]$ is not divisible by $\eta$.
\end{proof}

 Now, we use the truncated spectral sequences of \Cref{sectrunss}, applied to the elliptic spectral sequence of $tmf\wedge Y$. As in  \Cref{sectrunss}, let 
 \[\mathrm{sk}_{16}(tmf\wedge Y) = X_0/X_{17}\]
 for $X_n$ the $n$th term of the $X(4)$-Adams tower of $tmf \wedge Y$. Then $E_{r,<17}^{*,*}(Y)$ as in \Cref{sectrunss} is a spectral sequence computing $\pi_*\mathrm{sk}_{16}(tmf\wedge Y) $, and it satisfies $E_{r,<17}^{s,*}(Y)=0$ for $s\geq 17$. Furthermore, we have a map of spectral sequences 
 \[T_{r}^{s,t} \colon E_{r}^{s,t}(Y) \to E_{r,<17}^{s,t}(Y).\]
 
\begin{proposition}
In $\pi_*(\mathrm{sk}_{16}(tmf\wedge Y))$, we have
\[\langle \eta, \nu, \Delta^4a[12,2]\rangle =  \Delta^4 a[17,3]. \]
\end{proposition}

\begin{proof} 
In $\pi_*(tmf\wedge Y)$, the product $\nu \Delta^4a[12,2]$, if not trivial, is detected in filtration $17$. It follows that $\nu \Delta^4a[12,2]$ is equal to zero in $\pi_*(\mathrm{sk}_{16}(tmf\wedge Y))$. Thus, the Toda bracket $\langle \eta,\nu, \Delta^4a[12,2]\rangle$ can be formed. \Cref{Massey-Y} means that in $E_{2,<17}^{s,t}(Y)$, there is Massey product 
\[\langle\eta, \nu, \Delta^4a[12,2] \rangle = \Delta^4a[17,3].\] 
The conditions of the Moss Convergence Theorem \cite{Mos70} are satisfied, so the Toda bracket  
$\langle \eta,\nu, \Delta^4 a[12,2] \rangle$ contains $\Delta^4 a[17,3]$ and the indeterminacy is zero.  
\end{proof}

\begin{proposition} \label{propetaext}
In the elliptic spectral sequence for $tmf \wedge Y$, there is an exotic extension
\[\eta a[152,2] = \kappabar^2\Delta^4 a[17,3].\]
\end{proposition}

\begin{proof} Since $\kappabar^2 \Delta^4 a[17,3]$ lives in filtration $s=11$, it suffices to prove that extension in the $E_{\infty}$-term of the spectral sequence for $\mathrm{sk}_{16}(tmf\wedge Y)$.
The above proposition and the choice of $\Delta^4a[12,2]$ imply that
\[\kappabar^2\Delta^4 a[17,3] = \langle \eta, \nu, \Delta^4a[12,2]\rangle \kappabar^2 = \eta\langle \nu, \Delta^4 a[12,2],  \kappabar^2\rangle. \]
Since $\kappabar^2\Delta^4 a[17,3]\ne 0$ at $E_{\infty}$, $\langle\nu, \Delta^4 a[12,2] ,  \kappabar^2\rangle$
must be non-trivial, and it must be detected by a class which is not in the kernel of $\eta$.
This forces $\langle \nu, \Delta^4 a[12,2],    \kappabar^2 \rangle$  to be detected by $a[152,2]$, and so $\eta a[152,2]$ is detected by $\kappabar^2 \Delta^4a[17,3]$.
\end{proof}

 \begin{proof}[Proof of \Cref{keyExt}] 
 Let $\beta = [\nu \Delta^6a[0,0]]$.
 By \Cref{SomeExo},  $\eta \beta= 0$ and we can form the Toda bracket $\langle \nu, \eta, \beta\rangle$. Then 
 \[\eta \langle \nu, \eta, \beta\rangle =\langle \eta , \nu, \eta\rangle \beta= \nu^2 \beta\]
 On the other hand, $\nu^2 \beta\ne 0$ by \Cref{SomeExo}. It follows that $\langle \nu, \eta,\beta\rangle \ne 0$. We see that it must be detected by $a[152,2]$. So 
 $\eta a[152,2] = \nu^2 \beta$ and
  \Cref{propetaext} implies that $\nu^2 \beta$ is detected by $\kappabar^2\Delta^4a[17,3]$.
 \end{proof}

\begin{lem}
There are differentials:
\begin{enumerate}[(1)]
\item $d_{13}(\Delta^{6}a[30,2])=\kappabar^{3}\Delta^{4}a[17,3]$
\item $d_{19}(\Delta^{6}a[30,0])=\kappabar^{4}\Delta^{2}a[45,3]$
\end{enumerate}
\end{lem}
\begin{proof}
Let
\[tmf\wedge Y \leftarrow (tmf\wedge Y)_1 \leftarrow (tmf\wedge Y)_2 \leftarrow \ldots \] 
be the Adams tower associated to the $X(4)$-based resolution of $tmf\wedge Y$. 
We consider its $1$-co-truncated tower and the induced map of spectral sequences
\[ cT_{r}^{s,t} \colon E_{r,\geq 1}^{s,t} \to E_{r}^{s,t}.
\]
By \Cref{cTruncSS}, $cT_r^{s,t}$ is surjective for $s\geq 1$.

Let $a=\nu^2 \Delta^6 a[0,0]\in E_{2}^{2,150+2}$. This  is a permanent cycle representing a unique non-zero element of $ \pi_{150}(tmf\wedge Y)$, which in this proof we denote by $\alpha$. Since $a$ has positive filtration, there is a class $\bar a \in E_{2,\geq 1}^{2,150+2}$ such that $cT_2(\bar a)=a$ and the surjectivity of $cT_{\infty}$ guarantees that we can choose $\bar a$ to be a permanent cycle. It then detects classes $\bar \alpha \in \pi_{150}((tmf\wedge Y)_1)$ that map to $\alpha$.

Since $\nu \alpha$ is detected by $b=\kappabar^2 \Delta^4 a[17,3]\in E_{\infty}^{11,153+11}$ (\Cref{keyExt}), $\nu \bar \alpha $ must be detected in $E_{\infty}^{s,153+s}(cT_1)$ for $3 \leq s\leq 11$.  Since $E_{2}^{s,153+s}(cT_1) = 0$ for $3\leq s\leq 10$ (this is true for $E_2^{*,*}$), $\nu \bar \alpha $ must be detected by a lift $\bar b$ of $b$.

The relation $\kappabar \nu = 0\in \pi_{*}tmf$ implies that $\kappabar\nu\bar \alpha = 0 \in \pi_*((tmf\wedge Y)_1)$. This implies that 
$d_r(\bar c) = \kappabar \bar b$ for some non-trivial element $\bar c \in E_{r,\geq 1}^{15-r , 174+(15-r)}$. As $E_{2, \geq 1}^{0,*} = 0$, $\bar c$ must live in filtration $1\leq s \leq 13$, and hence so does $cT_r( \bar c)$. In particular, $cT_r(\bar c) \neq \Delta^6a[30,0]$. 
However, we find that
\begin{align*}
d_r(cT_r( \bar c)) &= cT_r(\kappabar \bar b) 
= \kappabar \cdot cT_r(\bar b) 
= \kappabar^3\Delta^4a[17,3]. 
\end{align*}
The only way for this to make sense is if $cT_r(\bar c)$ is equal to $\Delta^6a[30,2]$ and this is the desired differential (1).

This differential then determines (2) as noted in \Cref{rem2left}.
\end{proof}

\begin{rem}\label{rem153decideblueblack}
From, this discussion, we also learn that there is a non-trivial class in $i_*\pi_{150}V(0)$ which is detected by $a[153,11]$.
\end{rem}

\subsection{Exotic extensions}

In this section we resolve the exotic $2$, $\eta$, $\nu$ and $v_1$ extensions in the elliptic spectral sequence of $tmf \wedge Y$. The extensions are depicted in Figures~\ref{exoextY096} and \ref{exoextY96144}.

We begin with the exotic $\eta$-extensions, which are few.  To determine them, we use the following strategies.
First, the long exact sequence
\begin{align*}
 \ldots \to tmf_{*+1}V(0) \xrightarrow{\eta} tmf_*V(0)\xrightarrow{i_*} tmf_*Y\xrightarrow{p_*} tmf_{*-1}V(0) \to \ldots  \end{align*}

We use the following basic, but useful facts. 
\begin{lem}\label{remlestool}
For $a\in tmf_*Y$ and $b\in tmf_*V(0)$, 
 \begin{enumerate}
 \item if $a = i_*b$, then $\eta a= i_*\eta b =0$, 
 \item $p_*\eta a = \eta p_*a=0$, and
 \item $v_1 \eta a= \eta v_1 a$.
\end{enumerate}
\end{lem}
\begin{proof}
These are easy consequences of the long exact sequence on homotopy groups combined with the fact that composition as well as the smash product induces the $\pi_*S^0$-module structure in the stable homotopy category.
\end{proof}

Note further that \Cref{lem:duality} as described in \Cref{rem:usingduality} gives a way to relate extensions in different stems between the $v_1$-power torsion classes. We also use \Cref{lem2exttrick} and \Cref{propextnutrick}

A stem-by-stem analysis using the above techniques then allows us to determine that the only non-trivial exotic $\eta$-extensions are as follows:
\begin{lem}
In the elliptic spectral sequence of $Y$, there are exotic extensions
\begin{enumerate}
\item $\eta [\Delta^2\nu a[5,1]] = \kappabar^2 a[17,3]$
\item $\eta [\Delta^4\nu a[5,1]]  = \kappabar^5 a[5,1]$
\item $\eta [\Delta^6 \nu a[5,1]] = \kappabar^2[\Delta^4a[17,3]]$
\item $\eta [\Delta^6\nu a[20,2]] = \kappabar^5 [\Delta^2 a[20,2]]$
\end{enumerate}
There are no other exotic $\eta$-extensions.
\end{lem}
\begin{proof}
The first extension (1) follows from \Cref{propextnutrick}. The extension (2) and (4) follow from duality: (2) from $\eta [ \Delta^2a[20,2]]= [\Delta^2v_1^2a[17,3]]$ and (4) from $\eta a[5,1]= \nu^2 a[0,0]$.  Finally, (3) is \Cref{propetaext}. 

All possible exotic $\eta$-extensions are shown not to occur using \Cref{remlestool}, duality and \Cref{propextnutrick}. In particular, the possible $\eta$-extensions with source in stems $52\leq t-s \leq 57$ are shown not to occur using \Cref{propextnutrick} and $v_1$-linearity. 
\end{proof}

Now, we turn to the exotic $2$-extensions.

\begin{thm}
There are no exotic $2$-extensions in the elliptic spectral sequence for $Y$ and, consequently,
\[2(\pi_*tmf\wedge Y)=0.\]
\end{thm}
\begin{proof}
Since we have a cofiber sequence
\[  tmf \wedge C_{\eta} \xrightarrow{2}  tmf \wedge C_{\eta} \xrightarrow{j} tmf \wedge Y \xrightarrow{q} \Sigma  tmf \wedge C_{\eta},\]
we can apply \Cref{lem2exttrick} with $X=tmf \wedge C_{\eta}$, $i=j$ and $p=q$. From this, we deduce that if $a' \in \pi_*tmf\wedge Y$ is in the image of $j_*$, then it has order $2$ and that if $q_*(a')=a$, then $2a' = j_*(\eta a)$. It follows that if $2a'\neq 0$, then $2a'$ is divisible by $\eta$.

This leaves one possible extension in stem 57. But such a $2$-extension would lead, by duality, to a $2$-extension in stem 116. However, there are no $\eta$-divisible classes in that stem. Since the $E_2$-term was $2$-torsion and there are no exotic $2$-extensions, $\pi_*tmf\wedge Y$ is annihilated by $2$.
\end{proof}

Next, we turn to the $\nu$ extensions. 

\begin{rem}We will use without mention that $\kappabar \nu=0$ in $tmf_*$-modules. This allows us to eliminate many possible exotic $\nu$-extensions.
\end{rem}

\begin{lem}\label{lemnufirstbatch}
In the elliptic spectral sequence of $Y$, there are exotic extensions
\begin{enumerate}
\item \label{nu26} $\nu a[26,0] =  a[29,5]$ 
\item  \label{nu41} $\nu a[41,3]  = a[44,8]$ 
\item  \label{nu52} $\nu a[52,0] = a[55,7]$ 
\item  \label{nu54} $\nu a[54,2] = \kappabar^2 a[17,3]$ 
\item  \label{nu67}  $\nu a[67,3] = \kappabar^2 a[30,0]$
\item  \label{nu98} $\nu a[98,0] =  a[101,15]$ 
\item  \label{nu102} $\nu a[102,2] = \kappabar^5 a[5,1]$
\item  \label{nu103} $\nu a[103,1] = a[106, 16]$ 
\item   \label{nu124} $\nu a[124,0]=a[127,15]$
\item  \label{nu129} $\nu a[129,1] =a[132,16]$  
\item  \label{nu150} $\nu a[150,2] = a[153,11]$ 
\item  \label{nu155} $\nu a[155,3] = a[158,16]$ 
\item  \label{nu165} $\nu a[165,3] = a[168,22]$ 
\end{enumerate}
\end{lem}

\begin{proof}
The extensions
\eqref{nu26} and \eqref{nu98} follow from the extensions $\nu a[26,0] = a[29,5]$ and $a[98,0] = a[101,7]$,  respectively, in $\pi_*tmf \wedge V(0)$ by applying $i_*$. 
The extensions \eqref{nu41}, \eqref{nu52}, \eqref{nu67} and \eqref{nu124} follow from examining the effect of $p_*$ and the extensions $\nu a[39,3]=a[42,10]$, $\nu a[50,2] = a[53,7]$, $\nu a[65,3] =a[68,10]$ and $\nu a[122,2]= a[125,21]$ in $\pi_*tmf \wedge V(0)$, respectively.

Extensions \eqref{nu54}, \eqref{nu102}, \eqref{nu155} and \eqref{nu165} are obtained by duality from algebraic extensions. The extensions \eqref{nu129} and \eqref{nu103} follow by duality from \eqref{nu41} and \eqref{nu67}.

The extension \eqref{nu150}  is proved in \Cref{keyExt}.
\end{proof}

\begin{lem}\label{lemnuelse}
In the elliptic spectral sequence of $Y$, there are exotic extensions
\begin{enumerate} 
\item  \label{nu57}$\nu a[57,1] = \kappabar^2 a[20,2]$ 

\item  \label{nu62}  $\nu a[62,2] = \kappabar a[45,3]$

Dually, we have

\item  \label{nu108} $\nu a[108,2] = a[111,17]$ 

\item  \label{nu113} $\nu a[113,3] = a[116,18]$

\end{enumerate}
Together with \Cref{lemnufirstbatch}, there are no other non-trivial exotic $\nu$-extensions.
\end{lem}

To prove \Cref{lemnuelse}, we use the $tmf$-based Atiyah--Hirzebruch spectral sequence for $Y$, whose filtration comes from the cellular filtration of $Y$. To set up notation, we have the $E_1$-page of this spectral sequence
$$E_1 = \oplus_{n=0}^3 \pi_{*}tmf \Longrightarrow \pi_{*+n} tmf \wedge Y.$$
For a homotopy class $\beta$ in $\pi_* tmf \wedge Y$, we denote by $\alpha[n]$ the element that detects it in the $E_1$-page of the $tmf$-based Atiyah--Hirzebruch spectral sequence, where $n$ is the Atiyah--Hirzebruch filtration of $\beta$, and $\alpha$ is a class in $\pi_* tmf$. The stem of $\beta$ is then the stem of $\alpha$ plus $n$.

\begin{proof}[Proof of \Cref{lemnuelse}.]
In our Atiyah--Hirzebruch notation, we can rewrite the two $\nu$-extensions of \Cref{lemnuelse} as
\begin{enumerate}
\item  \label{nu57-2} $\nu \cdot \kappabar^2 \kappa [3] = \Delta \eta \kappa \kappabar [1]$,
\item \label{nu62-2} $\nu \cdot \kappabar^3 [2] = \Delta^2 \nu \kappa [0]$.
\end{enumerate}

We first prove \eqref{nu62-2}, namely, that $\nu \cdot \kappabar^3 [2] = \Delta^2 \nu \kappa [0]$.
In $\pi_*tmf \wedge C_\eta$, we have
$$\nu \cdot \kappabar^3 [2] = \langle \nu, \kappabar^3, \eta \rangle [0]$$
by Lemma~5.3 of \cite{WX}.
By Moss's Theorem and the differential $d_{11}(\Delta^2 \kappa) = \eta \kappabar^3$ in the elliptic spectral sequence of $tmf$, we have
$$\langle \nu, \kappabar^3, \eta \rangle = \Delta^2 \nu \kappa.$$
Mapping this relation along the inclusion $C_\eta \rightarrow Y$ gives us \eqref{nu62-2}.

For \eqref{nu57-2}, note that in $\pi_*tmf \wedge \Sigma C_\eta$, we have
$$\nu \cdot \kappabar^2 \kappa [3] = \langle \nu,  \kappabar^2 \kappa, \eta \rangle [1]$$
by Lemma~5.3 of \cite{WX}.
Since $\kappabar^2 \kappa$ is $\nu$-divisible in $\pi_* tmf$, we may shuffle
$$\langle \nu,  \kappabar^2 \kappa, \eta \rangle = \langle \kappabar^2 \kappa, \nu,  \eta \rangle.$$
By Moss's theorem and the differential $d_{5}(\Delta \kappa \kappabar) = \nu \kappabar^2 \kappa$ in $tmf$, we have
$$\langle \kappabar^2 \kappa, \nu,  \eta \rangle= \Delta \eta \kappa \kappabar.$$
Pulling back this relation along the quotient map $Y \rightarrow \Sigma C_\eta$ gives \eqref{nu57-2}.

Extensions (3) and (4) follow by duality. The fact that there are no other exotic $\nu$-extensions is discussed below.
\end{proof}

Most possibilities for other exotic $\nu$-extensions are ruled out in a straightforward way by analyzing $i_*$ and $p_*$, duality, the fact that $\kappabar\nu=0$. However, the following two extensions require us to analyze the classical Adams Spectral Sequence. The following proof depends on checking algebraic extensions in 
\[\Ext_{\mathcal{A}}((H\mathbb{F}_2)^*(tmf\wedge Y),(H\mathbb{F}_2)^*)\] 
using Bruner's $\Ext$-program \cite{Brunerprogram}. See \Cref{brunerextcharts} for classical Adams $E_2$-charts for $tmf \wedge V(0)$ and $tmf \wedge Y$, and see \cite[Chapter 13]{tmfbook} for $tmf$.

\begin{lem}\label{lemnuelsenot}
In $\pi_*tmf \wedge Y$,
\begin{enumerate} 

\item  \label{nu31} $\nu a[31,1] =0$ 

\item  \label{nu36} $\nu a[36,2] =0$

Dually, we have
\item \label{nu134}  $\nu a[134,2] =0$ 
\item  \label{nu39}  $\nu a[139,3] =0$ 
\end{enumerate}

\end{lem}

\begin{proof}
To show this, we need to prove that
\begin{enumerate} 
\item   $\nu a[31,1]  \neq a[34,6]$, 

\item $\nu a[36,2] \neq a[39,7]$.
\end{enumerate}
In our Atiyah--Hirzebruch notation, we can rewrite these extensions as
\begin{enumerate}
\item  \label{nu31-2} $\nu \cdot \kappa^2[3] \neq \kappa \kappabar [0]$,
\item  \label{nu36-2} $\nu \cdot \Delta \nu^3 [3] \neq \Delta \eta \kappa [0]$.
\end{enumerate}
We give a proof for \eqref{nu31-2} that $\nu \cdot \kappa^2[3] \neq \kappa \kappabar [0]$ using the classical Adams Spectral Sequence.
We consider the Adams Spectral Sequence for $tmf \wedge Y$ and its subquotients. We will show that the Adams filtration of $\kappa^2[3]$ is 7 and the Adams filtration of $\kappa \kappabar [0]$ is 8. The fact that there is no such $\nu$-extension follows from the algebraic fact that on the Adams $E_2$-page, the $h_2$-multiple of the first element is not the second element, which is checked by a computer program.

For the class $\kappa \kappabar [0]$, it is clear that the Adams filtration of $\kappa \kappabar$ in $\pi_{34} tmf$ is 8, (it is detected by the element $d_0g$,) and it maps nontrivially on the Adams $E_2$-pages along the map $tmf \rightarrow tmf \wedge Y$. The image under this map, which we denoted by $d_0g[0]$, is a permanent cycle. It cannot be killed due to filtration reasons. Therefore, the class $\kappa \kappabar [0]$ is detected by $d_0g[0]$ and, in particular, it has Adams filtration 8.

For the class $\kappa^2[3]$, we first consider the class $\kappa^2[1]$ in $\pi_{29} tmf \wedge V(0)$. 
Since $\pi_{29}tmf = 0, \pi_{30}tmf =0$, we have $\pi_{30}tmf \wedge V(0)=0$. This forces three nonzero Adams differentials eliminating the three elements in the Adams $E_2$-page for $tmf \wedge V(0)$.  In particular, we learn that $\kappa^2[1]$ in $\pi_{29} tmf \wedge V(0)$ is detected by the only remaining element $j[0]$ in Adams filtration 7, and that there is a nonzero $d_3$-differential from $(t-s, s)$-bidegrees $(31,6)$ to $(30,9)$. 

Considering the quotient map $tmf \wedge Y \rightarrow tmf \wedge \Sigma^2 V(0)$, we learn that $\kappa^2[3]$ is detected in Adams filtration at most 7. Considering the induced map on the Adams $E_2$-pages, we also learn that it is an isomorphism on the $(t-s, s)$-bidegrees $(31,6)$ and $(30,9)$. So, in particular, the element in $(t-s, s)$-bidegree $(31,6)$ does not survive. Therefore, $\kappa^2[3]$ is detected in Adams filtration exactly 7.

For \eqref{nu36-2}, that $\nu \cdot \Delta \nu^3 [3] \neq \Delta \eta \kappa [0]$, we use the Adams spectral sequence again in a very similar way.
We will show that the Adams filtration of $\Delta \nu^3 [3]$ is 8 and the Adams filtration of $\Delta \eta \kappa [0]$ is 9. The fact that there is no such extensions then follows as in \eqref{nu31-2}.

For the class $\Delta \eta \kappa [0]$, it is clear that the Adams filtration of $\Delta \eta \kappa$ in $\pi_{39} tmf$ is 9, (it is detected by the element $u$,) and it maps nontrivially on the Adams $E_2$-pages along the map $tmf \rightarrow tmf \wedge Y$. The image under this map, which we denoted by $d_0g[0]$, is a permanent cycle. It cannot be killed due to filtration reasons. Therefore, the class $\Delta \eta \kappa [0]$ is detected by $u[0]$, and in particular it has Adams filtration 9.

For the class $\Delta \nu^3[3]$, we first consider the class $\Delta \nu^3$ in $\pi_{33}tmf$. 
The class $\Delta \nu^3$ in $\pi_{33}tmf$ is detected in the Adams filtration 8. Considering the quotient map $tmf \wedge Y \rightarrow \Sigma^3 tmf $, we learn that $\Delta \nu^3[3]$ is detected in Adams filtration at most 8. To show that it is detected in  Adams filtration 8, we will show that the only other element in lower filtration, the class in $(t-s, s)$-bidegree $(36,7)$, supports a nonzero $d_2$-differential.   

The maps in the zigzag 
\[\xymatrix{
tmf \wedge S^1 & tmf \wedge V(0) \ar[l] \ar[r] & tmf \wedge Y
}\]
 are isomorphisms in $(t-s, s)$-bidegrees $(36,7)$ and $(35,9)$ on Adams $E_2$-pages. So the claimed nonzero $d_2$-differential follows from the one in the Adams spectral sequence of $tmf$, from $(t-s, s)$-bidegrees $(35,7)$ and $(34,9)$.
\end{proof}

We now turn to the study of the $v_1$-extensions. First, recall the discussion on $v_1$-self maps and $A_1$ from \Cref{v1selfmaps}. The homotopy groups of $tmf\wedge A_1$ are studied by the third author in \cite{PhamA1}. Furthermore, the knowledge of the homotopy groups of $tmf\wedge A_1$ are sufficient to allow us to deduce much of the action of $v_1$ on the homotopy groups of $tmf\wedge Y$, via the long exact sequence on homotopy of the cofiber sequence
\[ tmf\wedge \Sigma^2 Y\xrightarrow{v_1} tmf\wedge Y \rightarrow tmf\wedge A_1.
\]
Since the outcome depends on the choice of the $v_1$-self-map, we call a $v_1$-self-map of type $I$ if its cofiber is $A_1[01]$ or $A_1[10]$ and of type $II$, otherwise. Again, see \Cref{v1selfmaps} for the definition of $A_1[ij]$.
\begin{lem}\label{lemv1ext}
 (a) For all $v_1$-self maps of $Y$, there exotic $v_1$-extensions and those induced by $\kappabar$-linearity:
\begin{enumerate}
	\item \label{normal1} $v_1 a[9,1] = a[11,3]$ 
	\item $v_1  a[15,1] =  a[17,3]$ 
	\item $v_1 a[30,2] =  \kappabar a[12,2]$
	\item $v_1 a[33,1] = a[35,3]$
	\item $v_1 a[38,2] = \kappabar a[20,2]$
 \item $v_1 \Delta^{2} a[9,1] = \Delta^{2}a[11,3]$ 
	\item $v_1 a[99,1] = \kappabar^3 a[21,3]$
	\item $v_1 a[104,2] = \kappabar^4 a[26,0]$
	\item \label{intrus1} $v_1 a[105,1] = a[107,3]$ 
	\item $v_1(v_1 a[108,2]) = \kappabar^3 a[52,0]$
	\item $v_1 a[114,2] = \kappabar^4 a[36,2]$
	\item either \label{intrus2}$v_1 a[116,2] = \kappabar^2 a[78,0]$ or $v_1 a[116,2]  = \kappabar a[98,0]$ 
	\item $v_1 \kappabar a[105,1] = \kappabar^3 a[67,3]$ 
	\item $v_1 a[129,1] = a[131,3]$ 
	\item\label{intrus3} either $v_1 a[131,3] = \kappabar^2 a[93,3]$ or $v_1  a[131,3] = \kappabar a[113,3]$.
	\item $v_1 a[134,2] = \kappabar a[116,2]$ 
	\item $ v_1 \kappabar a[115,3] = \kappabar a[117,13]$
		\item $v_1(v_1 a[139,3]) = \kappabar^3 a[83,3]$ 
\item $v_1 \kappabar a[120,3] = \kappabar a[122,14]$
	\item $v_1(v_1 \kappabar a[124,0]) = \kappabar^4 a[68,2]$ 
	\item $v_1 a[147,1] = \kappabar a[129,1]$ 
	\item $v_1 a[152,2] = \kappabar a[134,2]$ 
	\item $v_1 a[156,10] = a[158,16]$
	\item $v_1 a[162,2] = \kappabar^2 a[124,0]$ 
\end{enumerate}

(b) For $v_1$-self-maps of type $I$, there are also the following $v_1$-extensions, and those induced from these by $\kappabar$-linearity:
\begin{enumerate}
\item $v_1 a[68,2] = \kappabar^2 a[30,2]$ 
	\item $v_1 a[83,3] = \kappabar^4 a[15,1]$ 
\end{enumerate}
\end{lem}
\begin{proof}
 For all parts, except for $(\ref{intrus1})$, $(\ref{intrus2})$, $(\ref{intrus3})$, we see, by inspecting the relevant parts of the homotopy groups of appropriate $tmf\wedge A_1[ij]$, that the targets of the stated $v_1$ extensions are sent to zero via the natural map
 \[\pi_*(tmf\wedge Y)\rightarrow \pi_*(tmf\wedge A_1[ij]).\] 
 Therefore, they are in the image of a $v_1$-multiplication and the stated $v_1$-extensions are the only possibilities.
 
 For part (\ref{intrus1}), consider 
 \[\mathrm{sk}_{4}(tmf\wedge Y) = (tmf\wedge Y) / (tmf\wedge Y)_{5},\] 
 where $(tmf\wedge Y)_{5}$ is the $5$th term in the $X(4)$-Adams tower of $tmf\wedge Y$. It is a module over $\mathrm{sk}_{4}(tmf)$. 
 Since $\Delta^4 \in \pi_{96}(\mathrm{sk}_{4}(tmf))$,
 this element acts on $\pi_*\mathrm{sk}_{4}(tmf\wedge Y)$. We see that the induced map $\pi_*(tmf\wedge Y) \rightarrow \pi_*\mathrm{sk}_{4}(tmf\wedge Y) $ sends $a[9,1]$ and $a[11,3]$ to non-trivial elements, which we denote by the same names. Furthermore, it sends $a[105,1]$ and $a[107,3]$ to elements detected by the products $\Delta^4 a[9,1]$ and $\Delta^4 a[11,3]$. Since $v_1 a[9,1] = a[11,3]$ by part (\ref{normal1}), 
 \[v_1 \Delta^4 a[9,1] = \Delta^4 v_1 a[9,1] = \Delta^4 a[11,1]\] 
 in $\pi_*\mathrm{sk}_{4}(tmf\wedge Y)$. It follows that $v_1 a[105,1]$ must be detected by $a[107,3]$ in the $E_{\infty}$-term of the elliptic spectral sequence of $tmf\wedge Y$.
\end{proof}

\begin{rem}\label{remundetermined}
 We are left with two undecided $v_1$-extensions, namely \eqref{intrus2} and \eqref{intrus3} in \Cref{lemv1ext}. We expect that some of these unsettled $v_1$-extensions can be resolved using comparison with the classical Adams Spectral Sequence for $tmf\wedge V(0)$, $tmf\wedge Y$ and $tmf\wedge V(0)/v_1^4$. These will soon appear in upcoming work of \cite{BrunerRognesbook}.
\end{rem}
\newpage

\begin{figure}[H]
\includegraphics[page=1, width=0.9\textwidth]{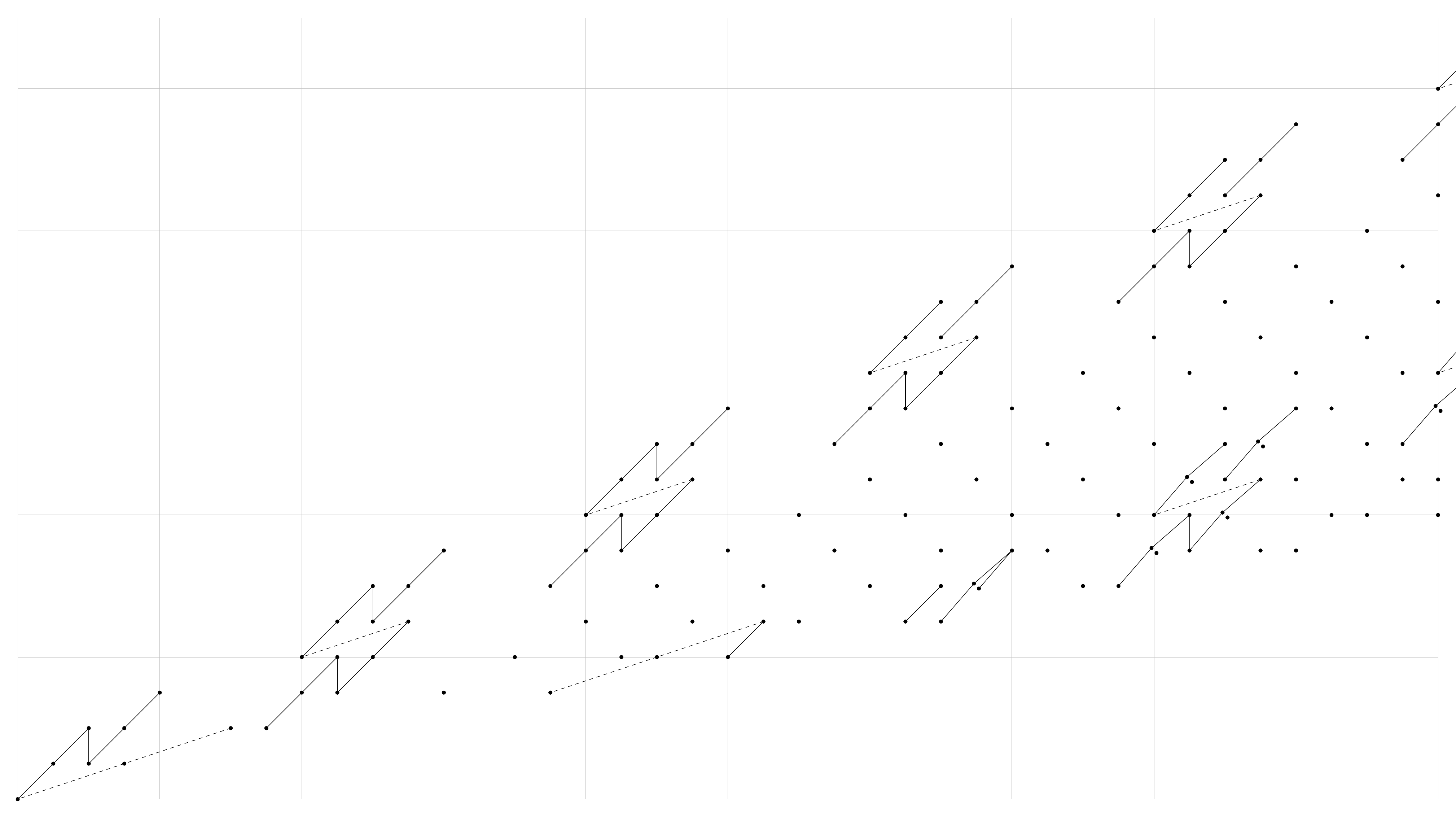}
\vspace{0.3in}
\includegraphics[page=1, width=0.9\textwidth]{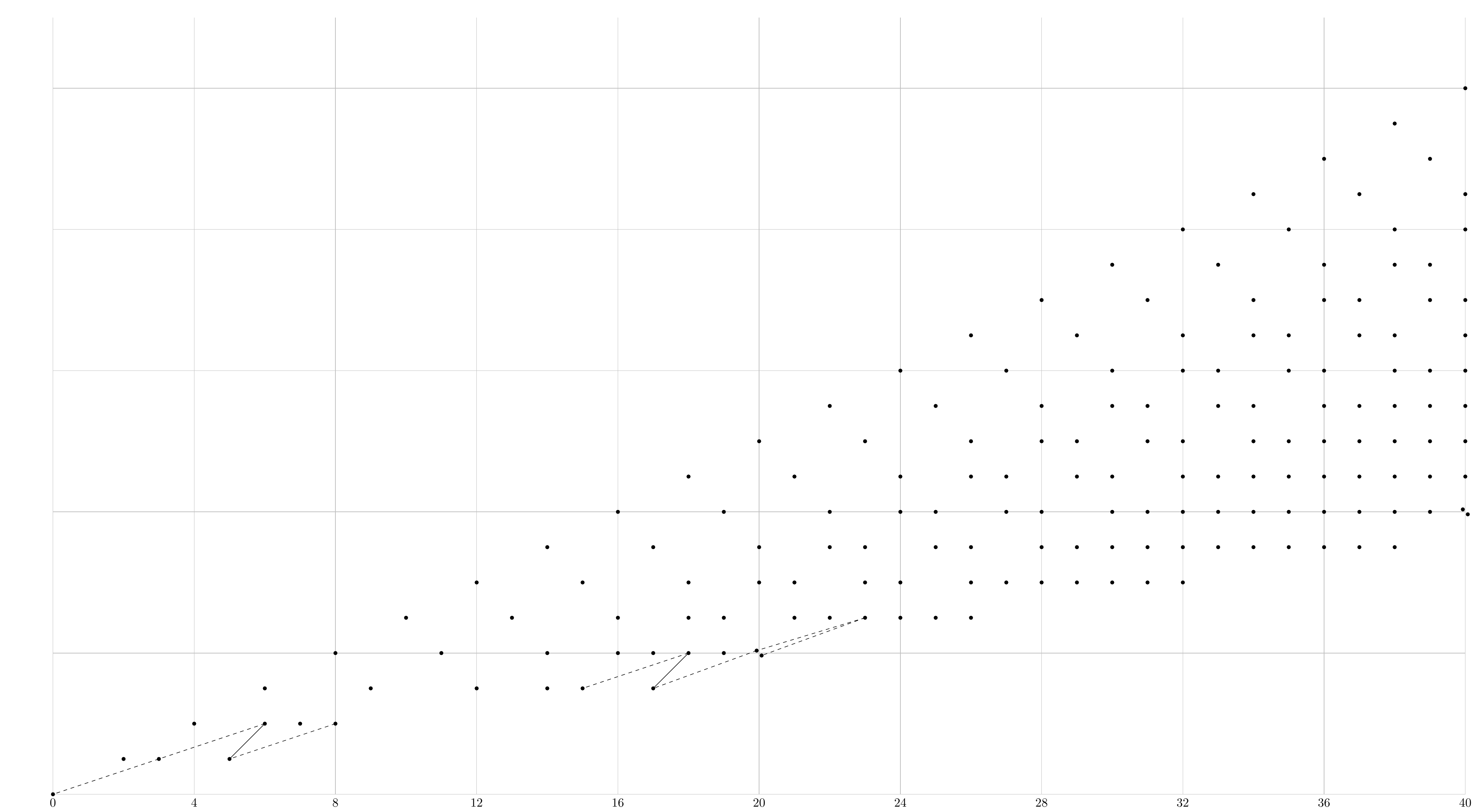}
\caption{Classical Adams Spectral Sequence $E_2$-pages for $tmf \wedge V(0)$ (top) and $tmf \wedge Y$ (bottom) computed with Bruner's $\Ext$-program \cite{Brunerprogram}.}
\label{brunerextcharts}
\end{figure}


\newpage 

\begin{figure}[H]
\includegraphics[page=1, width=\textwidth]{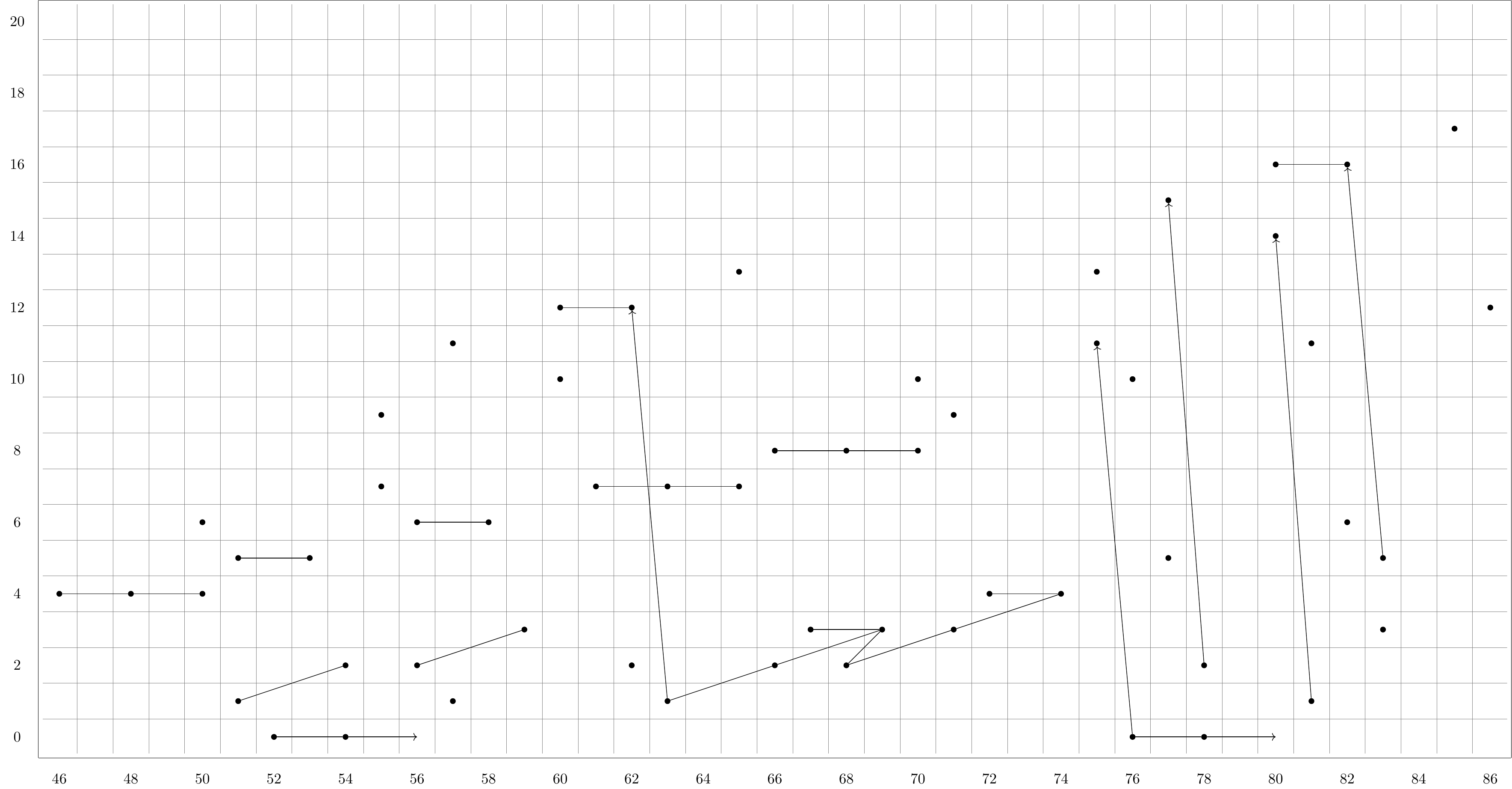}

\vspace{0.3in}

\includegraphics[page=1, width=\textwidth]{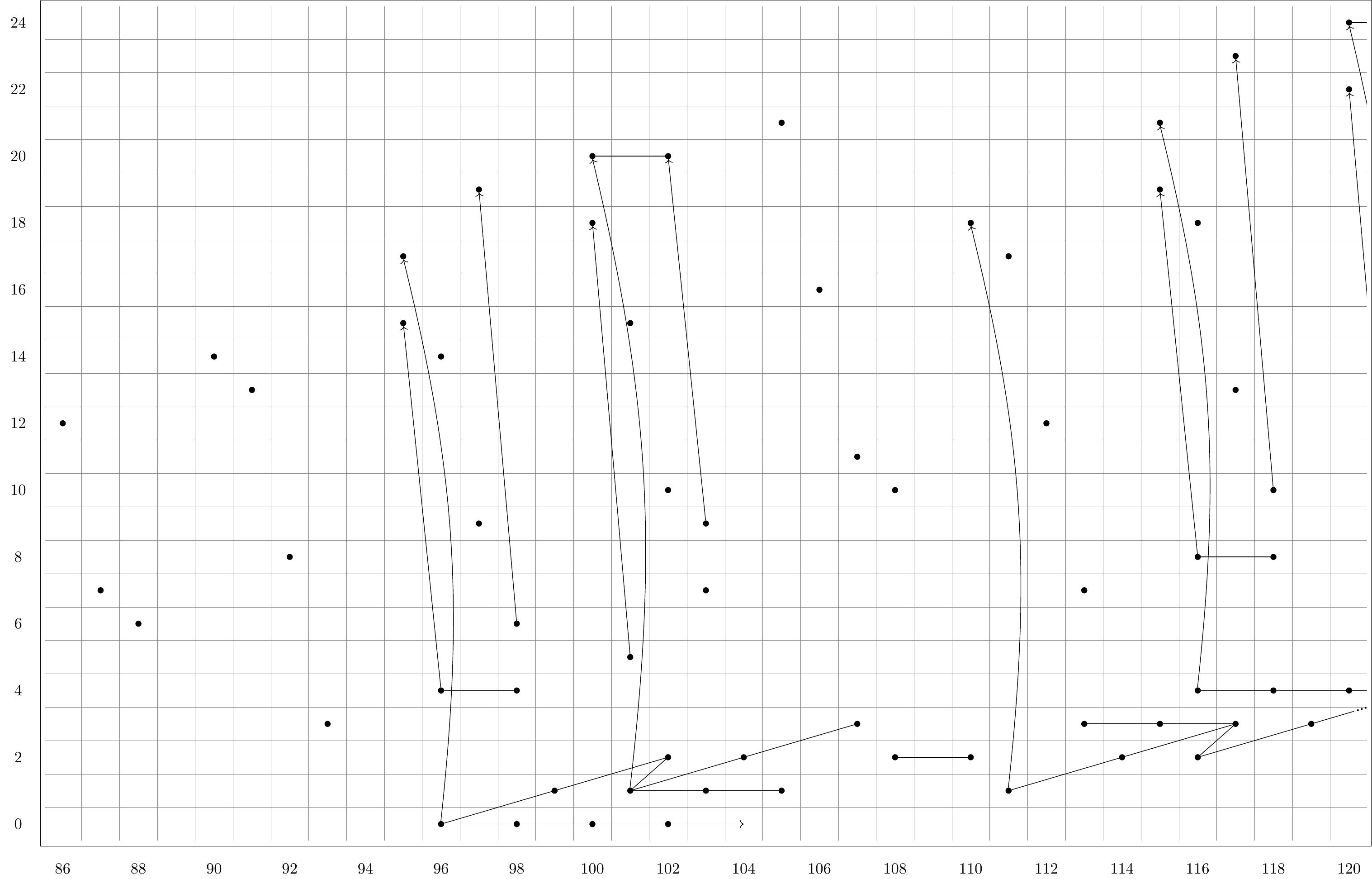}
\caption{$d_{11}$ to $d_{23}$ differentials in stems 46 to 120}
\label{d11d23one}
\end{figure}

\newpage

\begin{figure}[H]
\includegraphics[page=1, width=\textwidth]{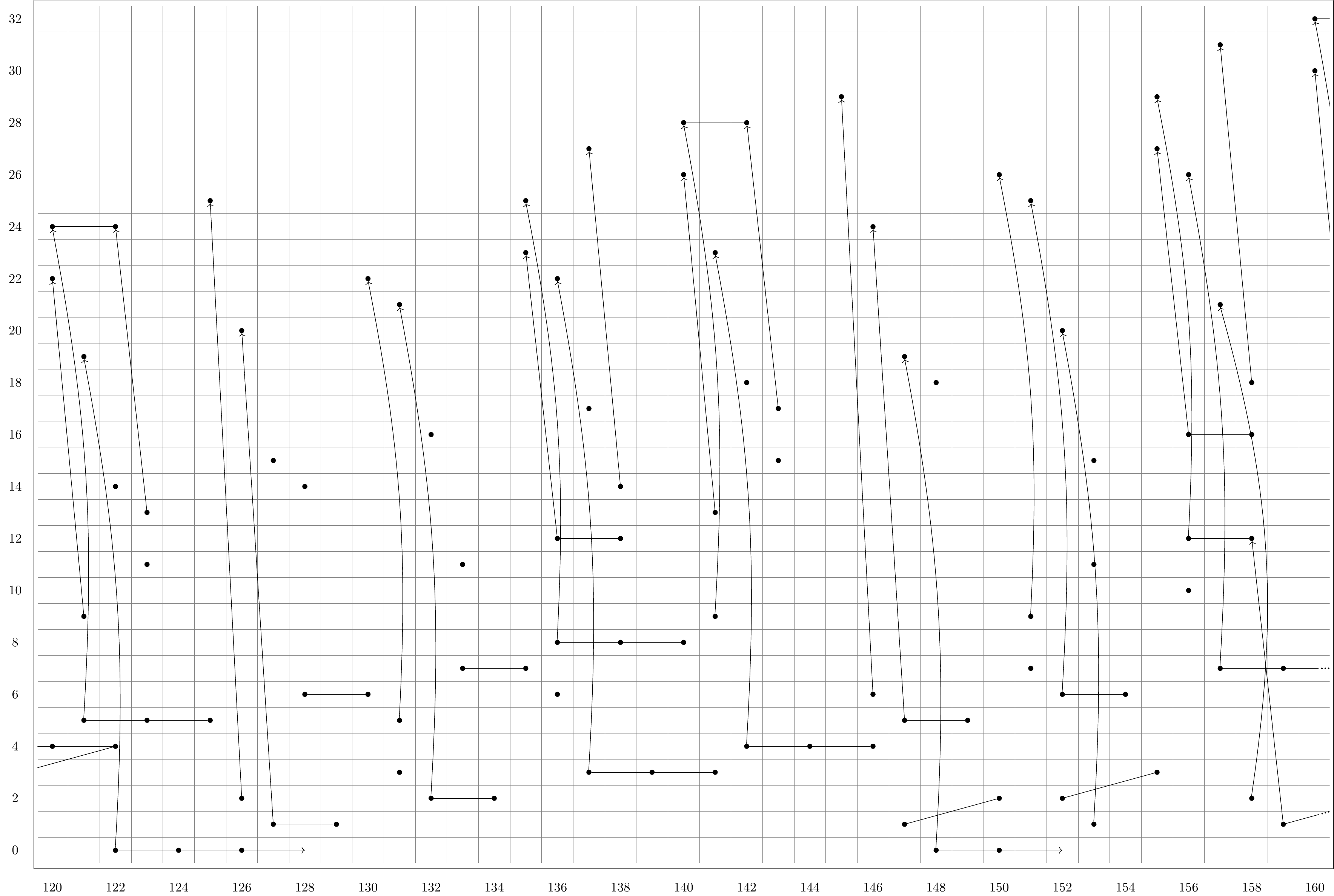}

\vspace{0.3in}

\includegraphics[page=1, width=\textwidth]{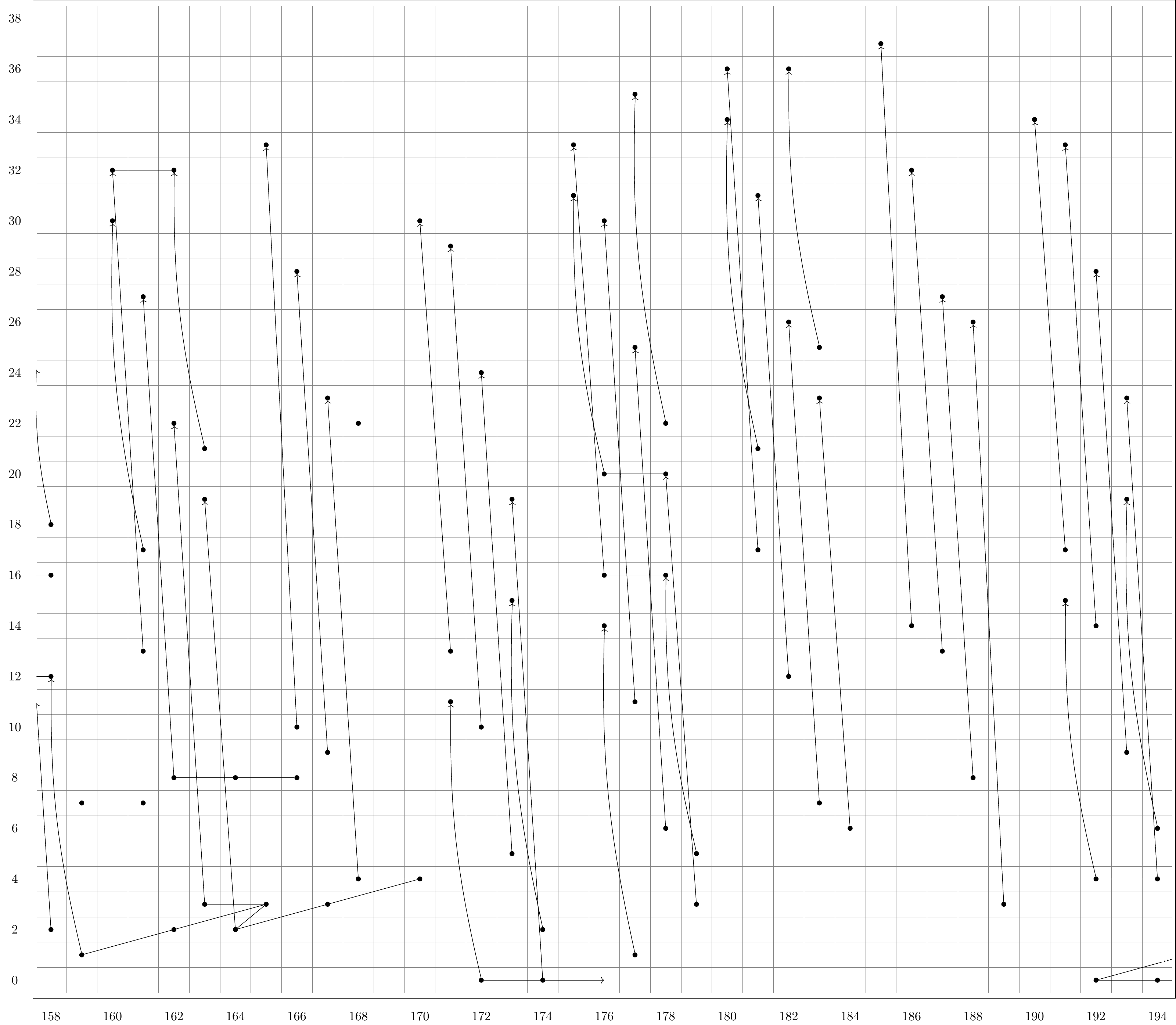}
\caption{$d_{11}$ to $d_{23}$ differentials in stems $120$ to $194$}
\label{d11d23four}
\end{figure}
\newpage

\begin{figure}[H]
\includegraphics[page=1, width=\textwidth]{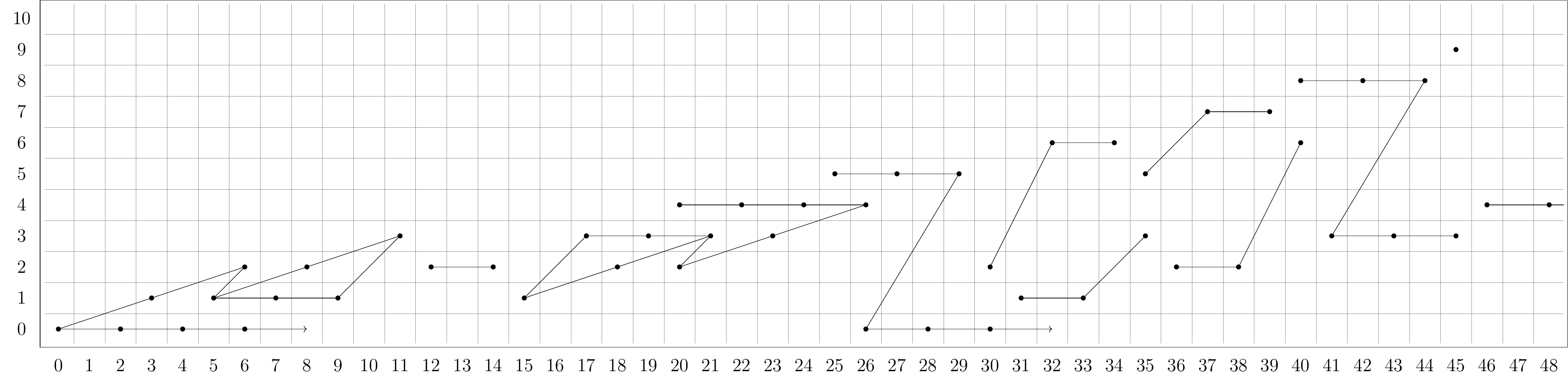}

\vspace{0.3in}

\includegraphics[page=1, width=\textwidth]{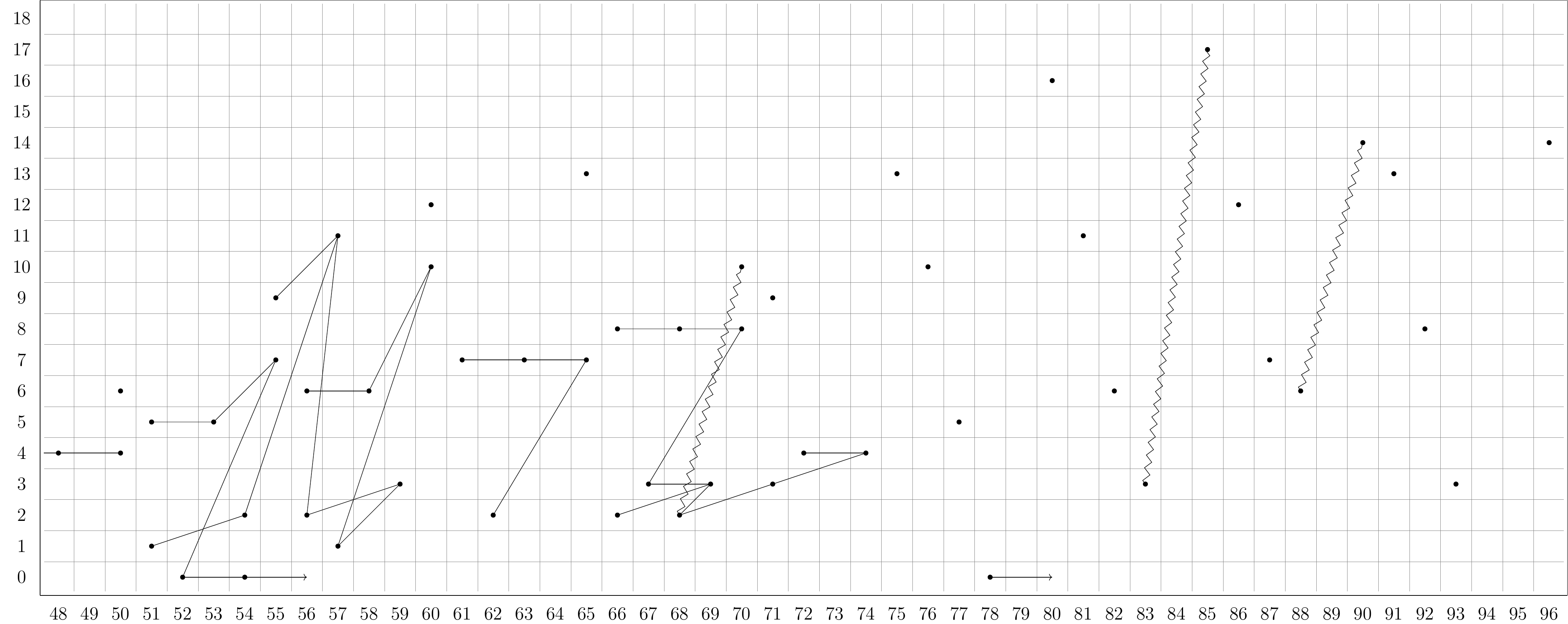}
\caption{Exotic extensions in the elliptic spectral sequence of $tmf\wedge Y$. This records $tmf_*Y \cong \widetilde{tmf}_{*+3}(\R P^2 \wedge \mathbb{C} P^2)$. The zigzags denote exotic $v_1$-extensions that occur only for certain choices of $v_1$ self-maps. }
\label{exoextY096}
\end{figure}

\begin{figure}[H]
\includegraphics[page=1, width=\textwidth]{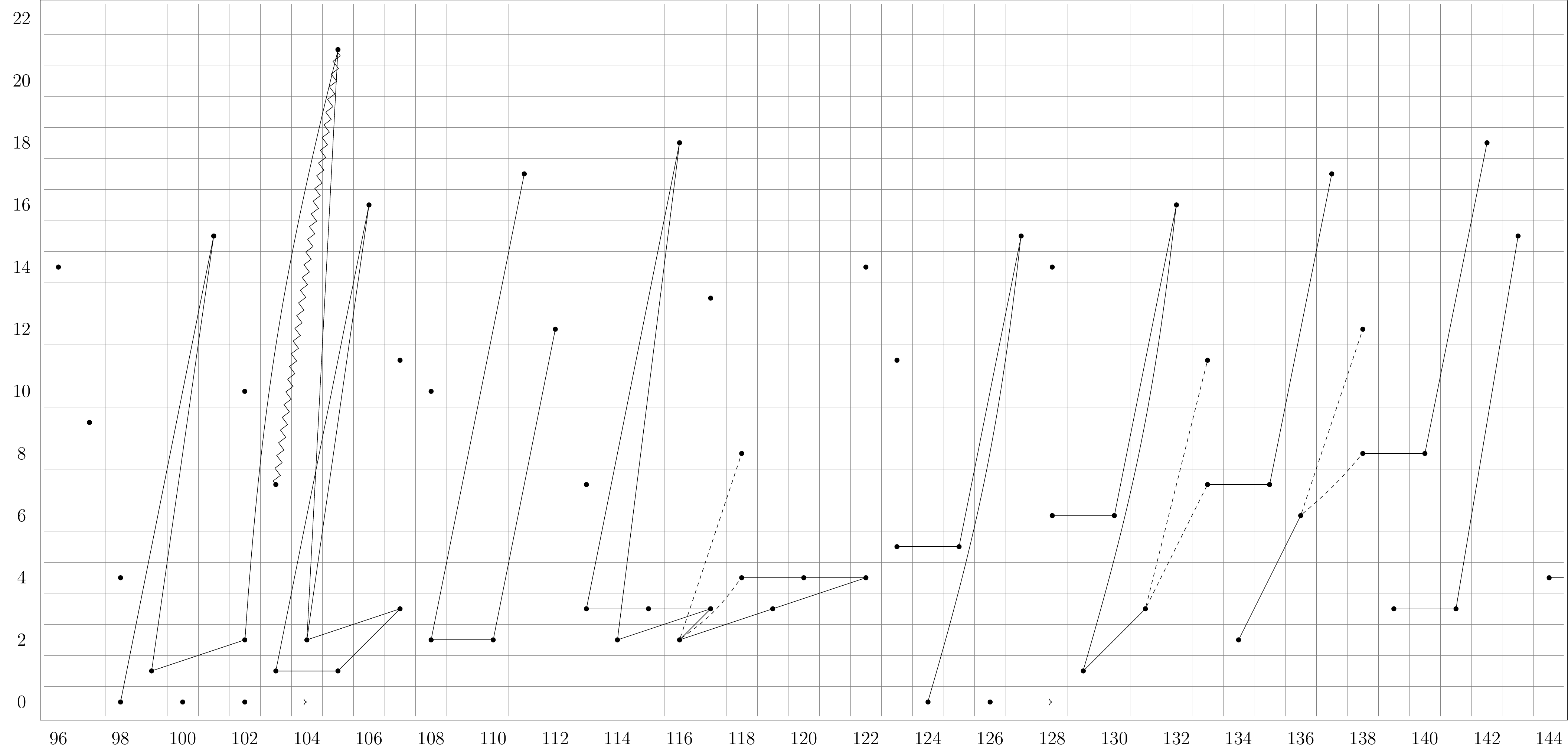}

\vspace{0.3in}

\includegraphics[page=1, width=\textwidth]{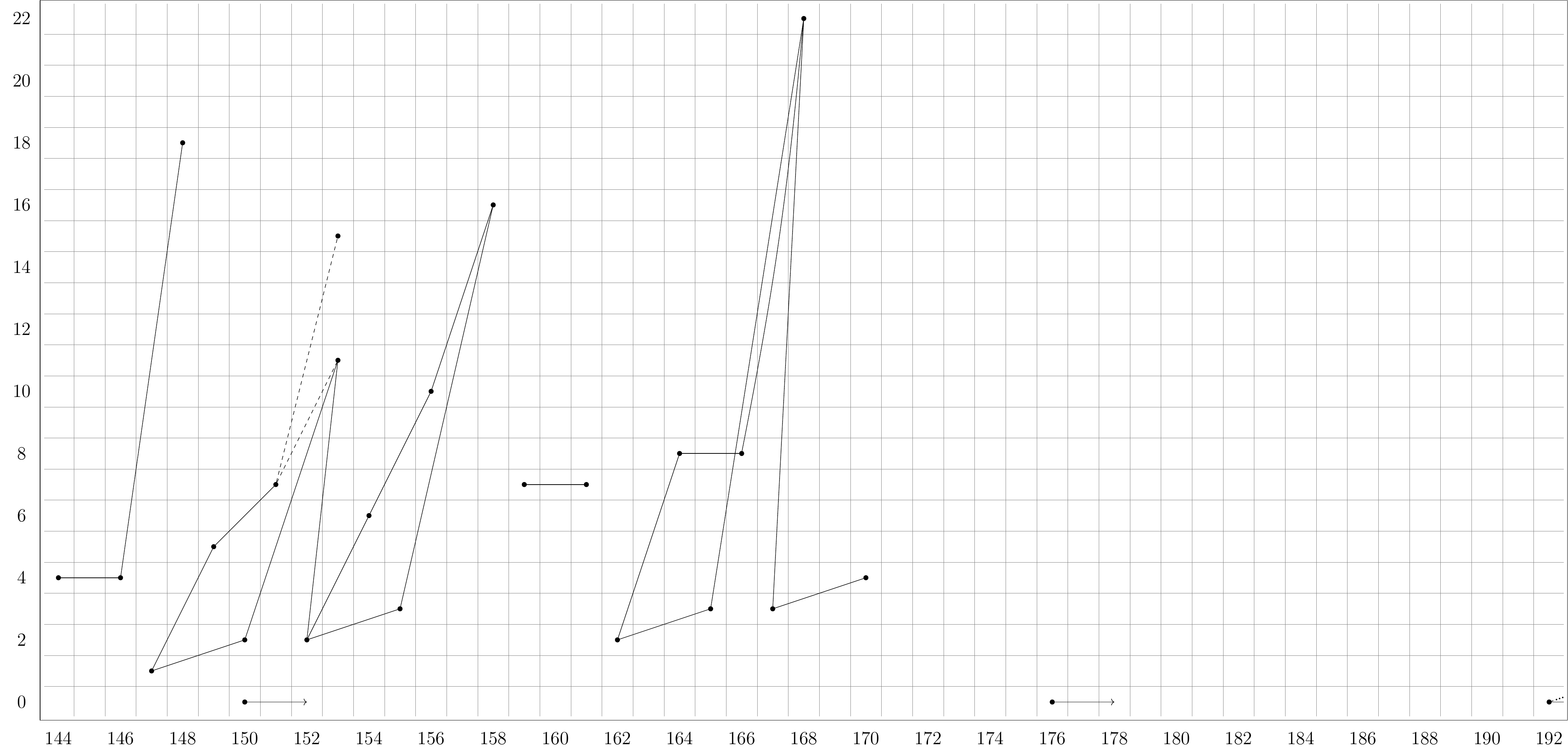}
\caption{Exotic extensions in the elliptic spectral sequence of $tmf\wedge Y$. This records $tmf_*Y \cong \widetilde{tmf}_{*+3}(\R P^2 \wedge \mathbb{C} P^2)$. The zigzags denote exotic $v_1$-extensions that occur only for certain choices of $v_1$ self-maps. }
\label{exoextY96144}
\end{figure}

\bibliographystyle{alpha}
\bibliography{bib}

\end{document}